\documentclass[a4paper, 11pt]{article}

\usepackage[utf8]{inputenc}
\usepackage[T1]{fontenc}
\usepackage{lmodern}     
\usepackage[english]{babel}
\usepackage{cite, hyperref}

\usepackage{amsfonts, amsmath}
\usepackage{amssymb, amsthm}
\usepackage{bbm}
\usepackage{calc}
\usepackage[shortlabels]{enumitem}

\usepackage[a4paper, top=2.5cm, bottom=2.5cm, left=3cm, right=3cm]{geometry}
\usepackage{xcolor}
\usepackage{graphicx}
\usepackage[font=small, labelfont= bf]{caption}
\allowdisplaybreaks

\numberwithin{equation}{section}
\theoremstyle{plain}

\newtheorem{Lemma}{Lemma}[section]
\newtheorem{Proposition}[Lemma]{Proposition}
\newtheorem{Theorem}[Lemma]{Theorem}
\newtheorem{Corollary}[Lemma]{Corollary}

\theoremstyle{definition}
\newtheorem{Definition}[Lemma]{Definition}
\newtheorem{Example}[Lemma]{Example}
\newtheorem{Remark}[Lemma]{Remark}

\def\E{{\mathbb{E}}}
\def\N{\mathbb{N}}
\def\P{\mathbb{P}}
\def\R{{\mathbb{R}}}
\DeclareMathOperator{\sgn}{sgn}

\begin{document}
\title{\textbf{On McKean--Vlasov SDEs with polynomial drifts for SIS epidemic models}}
\author{Alexander Kalinin\footnote{Department of Mathematics, LMU Munich, Germany. E-mail addresses: {\tt kalinin@math.lmu.de}, {\tt meyerbra@math.lmu.de} and {\tt steibel@math.lmu.de}}
\and Thilo Meyer-Brandis\footnotemark[1]
\and Annika Steibel\footnotemark[1]}
\date{April 21, 2026}
\maketitle

\begin{abstract}
We present a tractable class of one-dimensional McKean--Vlasov equations that allow for unique strong solutions and extend the dynamics of various SIS epidemic models that are well-established in the literature. While the distribution-dependent drift coefficients are of polynomial type, the diffusion coefficients may involve sums of power functions. Our analysis includes various scenarios of extinction and persistence of the disease and an effective Euler--Maruyama scheme, for which we derive an explicit strong error estimate in $p$th moment for $p\geq 2$.
\end{abstract}

\noindent
{\bf MSC2020 classification:} 60H20, 92D30, 60F15, 65C30.\\
{\bf Keywords:} McKean--Vlasov equation, SIS epidemic model, polynomial drift, extinction, persistence, Euler--Maruyama scheme, strong error estimate.

\section{Introduction}

The coronavirus pandemic at the beginning of this decade had a profound impact on societies worldwide. This crisis emphasised the crucial role of epidemic modelling, which aims to analyse and forecast the spread of infectious diseases within a population. By understanding the mechanisms that drive an epidemic and predicting the effectiveness of interventions, such as vaccinations or social distancing, mathematical models can help to make informed decisions that mitigate the course of an epidemic. 

In compartmental models, as proposed by Kermack and McKendrick~\cite{KerMcK27}, a given population is divided into categories according to the health status, and the transition between these compartments is originally described by a system of ordinary differential equations (ODEs). In this work, we focus on susceptible-infected-susceptible (SIS) models, in which the population consists of susceptible and infected compartments and there is no permanent immunity upon recovery. Namely, susceptible individuals become infected through contact with infectious individuals at a certain transmission rate, and infectious individuals recover at a constant rate and return to the susceptible compartment.

In a classical deterministic SIS model for a population of size $N\in\N$ considered by Hethcote and Yorke~\cite{HetYor84}, the following system of ODEs is used to describe the dynamics of the respective numbers $S$ and $I$ of susceptible and infected individuals:
\begin{align*}
\dot{S}(t) &= \mu N - \beta S(t) I(t) + \gamma I(t) - \mu S(t),\\
\dot{I}(t) &= \beta S(t) I(t) - (\mu + \gamma) I(t)
\end{align*}
for $t\geq 0$ with initial conditions $S(0) = N - i_{0}$ and $I(0) = i_{0}$, where $i_{0}\in ]0,N[$. Here, $\beta\geq 0$ is the disease transmission coefficient that captures the average number of potentially infectious contacts of an individual and the probability of infection per contact, $\gamma\geq 0$ is the recovery rate and $\mu\geq 0$ represents the per capita birth and death rates. Since $S(t) + I(t) = N$ for any $t\geq 0$, the system can be reduced to the one-dimensional ODE
\begin{equation}\label{eq:deterministic SIS epidemic model}
\dot{I}(t) = \beta I(t)(N-I(t)) - (\mu+\gamma)I(t)
\end{equation}
for $t\geq 0$ with initial condition $I(0) = i_{0}$. There are various extensions of the deterministic SIS epidemic model~\eqref{eq:deterministic SIS epidemic model}, in which the dynamics of the number of infected individuals are described by a stochastic differential equation (SDE) driven by a standard Brownian motion $B$ that is defined on a complete probability space $(\Omega,\mathcal{F},\P)$.

In the work~\cite{GraGreHuMaoPan11} by Gray, Greenhalgh, Hu, Mao and Pan, a stochastic SIS model is derived from~\eqref{eq:deterministic SIS epidemic model} by allowing for random perturbations around the disease transmission coefficient $\beta$, which accounts for uncertainty due to environmental factors or behavioural changes. Their proposed SDE that generalises~\eqref{eq:deterministic SIS epidemic model} and describes the dynamics of the random number of infected individuals is of the following form:
\begin{equation}\label{eq:standard SIS epidemic model}
\mathrm{d}I_{t} = \big(\beta I_{t}(N-I_{t}) - (\mu+\gamma)I_{t}\big)\,\mathrm{d}t + \sigma I_{t}(N-I_{t})\,\mathrm{d}B_{t}
\end{equation} 
for $t\geq 0$ with the same initial condition $I_{0} = i_{0}$, and the parameter $\sigma\geq 0$ reflects the intensity of the stochastic fluctuations around $\beta$ with respect to $B$. Cai, Cai and Mao~\cite{CaiCaiMao19, CaiCaiMao19-2} extend this SIS model by incorporating independent or possibly correlated Brownian motions as random perturbations around $\mu + \gamma$.

Regarding other types of parameter perturbations, the SIS model~\cite{WanCaiDinGui18} by Wang, Cai, Ding and Gui is obtained from~\eqref{eq:deterministic SIS epidemic model} by representing the disease transmission coefficient as a mean-reverting Ornstein-Uhlenbeck process. Alternatively, Bernardi and Lanconelli~\cite{BerLan22} consider stochastic fluctuations of $\beta$ around piecewise linear approximations of $W$. By applying the Wong-Zakai Theorem, they deduce an SDE from~\eqref{eq:deterministic SIS epidemic model} that governs the epidemic dynamics.\smallskip

In this work, we propose an SIS model based on the broad class of SDEs~\eqref{eq:SIS epidemic model} with law-dependent drift coefficients of polynomial type, generalising all these models. Before introducing the mathematical framework in detail, we motivate the choice of these drift coefficients through an illustrative example. To this end, we recall that \eqref{eq:standard SIS epidemic model} provides a realistic approximation of the evolution of the number of infected individuals under the assumption of a large and homogeneously mixing population, as considered in~\cite{ArmBec17}.

Now consider a system of $n\in\N$ such populations, each undergoing an endogenous SIS epidemic with identical parameters, supplemented by inter-population transmission. We assume that these subpopulations are homogeneously mixing -- that is, their interaction is not constrained by network structures, such as geographical segregation. Moreover, the aggregate impact of the inter-population interaction on any given subpopulation depends solely on the distribution of the infected states throughout the system.

As a motivating example, let us consider a vast collection of honeybee colonies, each containing tens of thousands of bees that interact freely within their respective hives. The colonies may be clustered in an apiary or distributed in a natural habitat, sharing common foraging areas and resources that facilitate inter-population transmission of a disease. Further, the pathogens are transmitted either directly through contact with conspecifics within the same colony or indirectly via environmental contamination, such as spore-laden food or polluted water, as mentioned in~\cite{MuhEbe20, MazGaj22}.

Thus, the average infection level across all colonies influences the disease transmission probability within any colony. Indeed, the environmental pathogen exposure can increase the susceptibility of individual bees, thereby enhancing the likelihood of infection during direct contact. In addition, a high average infection level may induce behavioural or management-related changes -- including stress responses, avoidance of contaminated areas or human interventions -- which further impact the likelihood of infection.

A straightforward way to incorporate such a scenario into the model~\eqref{eq:standard SIS epidemic model} is to assume that the impact of the inter-population interaction on the disease transmission rate is proportional to the average infection level across all subpopulations. More specifically, let $I^{(n,\ell)}$ be the random number of infected individuals in each subpopulation $\ell\in\{1,\dots,n\}$. Then the disease transmission coefficient acting on each subpopulation could be of the form $\beta + \beta_{1}\frac{1}{N}\frac{1}{n}\sum_{m=1}^{n}I^{(n,m)}_{t}$ at any time $t\geq 0$.

The parameter $\beta\geq 0$ models the disease transmission rate within any subpopulation in the absence of inter-population interactions and $\beta_{1}\geq -\beta$ governs both the magnitude and direction of the impact of the average infection level. Since the subpopulations may tend to reduce their contact with one another in response to a rising average infection level, we allow $\beta_{1}$ to take negative values, capturing behavioural adaptations and precautionary measures. This leads to the following dynamics:
\begin{equation}\label{eq:standard SIS epidemic model with empirical means}
\begin{split}
\mathrm{d}I_{t}^{(n,\ell)} &= \bigg(\bigg(\beta + \beta_{1}\frac{1}{N}\frac{1}{n}\sum_{m=1}^{n}I_{t}^{(n,m)}\bigg)I_{t}^{(n,\ell)}(N - I^{(n,\ell)}_{t}) - (\mu + \gamma)I_{t}^{(n,\ell)}\bigg)\,\mathrm{d}t\\
&\quad + \sigma I_{t}^{(n,\ell)}(N-I_{t}^{(n,\ell)})\,\mathrm{d}B_{t}^{(\ell)}
\end{split}
\end{equation}
for $t\geq 0$, where $B^{(1)},\dots,B^{(n)}$ are independent standard Brownian motions. As the number $n$ of subpopulations tends to infinity, the impact of any single subpopulation on the entire system vanishes asymptotically, and the numbers of infected individuals in each subpopulation behave as if they were independent. This probabilistic phenomenon arises in a variety of interacting particle systems and is formally referred to as \emph{propagation of chaos}.

It entails that, under appropriate regularity conditions, the system's dynamics decouple and the particles become approximately independent and identically distributed as the system size tends to infinity. To be precise, let $(I^{(n,1)},\dots,I^{(n,n)})$ be a strong solution to the $n$-dimensional system~\eqref{eq:standard SIS epidemic model with empirical means} of SDEs with initial conditions $I_{0}^{(n,1)} = \cdots = I_{0}^{(n,n)} = i_{0}$. Furthermore, for each $\ell\in\{1,\dots,n\}$ let $I^{(\ell)}$ denote the unique strong solution to the McKean--Vlasov SDE
\begin{equation}\label{eq:standard SIS epidemic model with an expected value}
\mathrm{d}I_{t}^{(\ell)} = \bigg(\bigg(\beta + \beta_{1}\frac{\E[I_{t}^{(\ell)}]}{N}\bigg) I_{t}^{(\ell)}(N-I_{t}^{(\ell)}) - (\mu+\gamma)I_{t}^{(\ell)}\bigg)\,\mathrm{d}t + \sigma I_{t}^{(\ell)}(N-I_{t}^{(\ell)})\,\mathrm{d}B_{t}^{(\ell)}
\end{equation}
for $t\geq 0$ with initial condition $I_{0}^{(\ell)} = i_{0}$. Then $I^{(1)},\dots,I^{(n)}$ are independent and it follows that $\lim_{n\uparrow\infty}\max_{\ell=1,\dots,n}\P(\sup_{t\in [0,T]}|I_{t}^{(n,\ell)} - I_{t}^{(\ell)}|\geq\varepsilon) = 0$  for any $T,\varepsilon > 0$. For details, we refer to~\cite{Szn91, Mel96, ReiEngSmi22}. Consequently, the mean-field SDE~\eqref{eq:standard SIS epidemic model with an expected value} characterises the dynamics of a representative subpopulation, and the law-dependent drift coefficient accounts for the mean-field interaction within a large system of homogeneous subpopulations undergoing an endogenous SIS epidemic with identical parameters.

To account for alternative mechanisms of inter-population contagion described by mean-field equations, let us suppose instead that a disease is transmitted exclusively through direct contact across all colonies. In addition, beyond interactions within their own subpopulation, every individual maintains a constant rate of potentially infectious contacts with members of other subpopulations.

Let $\beta$ and $\beta_{0}$ denote the intra-population and inter-population disease transmission rates, respectively. Then the dynamics of the number $I^{(n,\ell)}$ of infected individuals in each subpopulation $\ell\in\{1,\dots,n\}$
can be modelled by~\eqref{eq:standard SIS epidemic model with empirical means} once the expression
\begin{equation*}
\bigg(\beta + \beta_{1}\frac{1}{N}\frac{1}{n}\sum_{m=1}^{n}I_{t}^{(n,m)}\bigg)I_{t}^{(n,\ell)}\quad\text{is replaced by}\quad \beta I^{(n,\ell)}_{t} + \beta_{0}\frac{1}{n}\sum_{m=1}^{n}I^{(n,m)}_{t}.
\end{equation*}
By propagation of chaos, as $n$ grows large, it follows that the dynamics of a representative subpopulation can be described by~\eqref{eq:standard SIS epidemic model with an expected value} if $(\beta + \beta_{1}\frac{\E[I_{t}^{(\ell)}]}{N})I_{t}^{(\ell)}$ is replaced by $\beta I_{t} + \beta_{0}\E[I_{t}^{(\ell)}]$.

For the precise formulation of our SIS model, let $\mathbb{F} = (\mathcal{F}_{t})_{t\geq 0}$ be a right-continuous filtration of $\mathcal{F}$ such that $\mathcal{F}_{0}$ contains all null events and $W$ be a standard $d$-dimensional $\mathbb{F}$-Brownian motion for some $d\in\N$.

For $p\geq 1$ we write $\mathcal{P}_{p}(\R)$ for the Polish space of all Borel probability measures $\nu$ on $\R$ with finite $p$th absolute moment $\int_{\R} |x|^{p}\,\nu(\mathrm{d}x)$, endowed with the \emph{$p$th Wasserstein metric} given by
\begin{equation*}
\mathcal{W}_{p}(\nu,\tilde{\nu}) := \inf_{\theta\in\mathcal{P}(\nu,\tilde{\nu})} \bigg(\int_{\R\times \R} |x - y|^{p}\,\mathrm{d}\theta(x,y)\bigg)^{\frac{1}{p}},
\end{equation*}
where $\mathcal{P}(\nu,\tilde{\nu})$ is the convex space of all Borel probability measures $\theta$ on $\R\times \R$ with first and second marginal distributions $\nu$ and $\tilde{\nu}$, respectively, for any $\nu,\tilde{\nu}\in\mathcal{P}_{p}(\R)$. We endow the convex space $\mathcal{P}_{0}(\R)$ of all Borel probability measures $\nu$ on $\R$ whose support $\mathrm{supp}(\nu)$ is compact with a pseudometric $\vartheta$ that is ought to satisfy
\begin{equation}\label{eq:domination condition}
\vartheta(\nu,\tilde{\nu}) \leq \mathcal{W}_{p}(\nu,\tilde{\nu})\quad\text{for all $\nu,\tilde{\nu}\in \mathcal{P}_{0}(\R)$.}
\end{equation}
This allows for the choice $\vartheta = \mathcal{W}_{q}$ for any $q\in [1,p]$. In this context, we recall that the support of a Borel probability measure $\nu$ on the Polish space $\R$ is the smallest closed $C$ set in $\R$ that satisfies $\nu(C) = 1$.

For $k\in\N$ let $b_{0},\dots,b_{k}\colon\R_{+}\times\R_{+}\times\mathcal{P}_{0}(\R)\rightarrow\R$ and $f\colon\R_{+}\times\R_{+}\times\R_{+}\rightarrow\R^{1\times d}$ be product measurable and $N\colon\R_{+}\rightarrow\R_{+}$ be locally absolutely continuous. By means of these maps we introduce the one-dimensional McKean--Vlasov SDE
\begin{equation}\label{eq:SIS epidemic model}
\mathrm{d}I_{t} = \sum_{i=0}^{k} b_{i}\big(t,N(t),\mathcal{L}(I_{t})\big)I_{t}^{i}\,\mathrm{d}t + f(t,I_{t},N(t) - I_{t})\,\mathrm{d}W_{t}
\end{equation}
for $t\geq 0$ with initial condition $I_{0} = \xi_{0}$, where $\xi_{0}$ is an $\mathcal{F}_{0}$-measurable random variable taking all its values in $[0,N(0)]$. Here, $\mathcal{L}(X)$ denotes the distribution of any random variable $X$, and the closed and pathwise connected set $D:=\{(t,x)\in \R_{+}\times\R_{+}\mid x\leq N(t)\}$ is the domain of the underlying diffusion coefficient
\begin{equation}\label{eq:diffusion coefficient}
D\rightarrow\R^{1\times d},\quad (t,x)\mapsto f(t,x,N(t)-x).
\end{equation}
So, any solution $I$ has to satisfy $(t,I_{t})\in D$ for all $t\geq 0$, which is equivalent to $0\leq I \leq N$. We will solve~\eqref{eq:SIS epidemic model} uniquely in Propositions~\ref{pr:existence, uniqueness and a growth estimate} and~\ref{pr:existence, uniqueness and a growth estimate 2} by combining the pathwise uniqueness and strong existence results in~\cite{KalMeyPro24-2, KalMeyPro24} with the value analysis in~\cite{BriGraKal24}.\smallskip

The epidemic model~\eqref{eq:SIS epidemic model} generalises, in multiple ways, various well-established SIS models from the literature, including those in~\cite{GraGreHuMaoPan11, WanCaiDinGui18, CaiCaiMao19, CaiCaiMao19-2, BerLan22}, which are special cases of the \emph{representative model} in Example~\ref{ex:a representative SIS epidemic model}. First, we allow all coefficients and the population size to be time-dependent, providing a more realistic description of the epidemic dynamics. For instance, this enables $\beta$, $\gamma$, $\mu$, $\sigma$ and $N$ in the model~\eqref{eq:standard SIS epidemic model} to vary over time, thereby reflecting temporal changes caused by factors such as seasonality, behavioural adaptations or vaccination programs.

Secondly, the diffusion coefficient~\eqref{eq:diffusion coefficient} may in fact comprise sums of power functions, accommodating a wide range of parameter perturbations. Namely, for $\zeta,\eta \in [\frac{1}{2},\infty[^{d\times m}$ and a measurable locally bounded map $g\colon\R_{+}\rightarrow\R^{d\times m}$, we may take
\begin{equation*}
f_{i}(\cdot,x,y) = g_{i,1}x^{\zeta_{i,1}}y^{\eta_{i,1}} + \cdots + g_{i,m}x^{\zeta_{i,m}}y^{\eta_{i,m}}
\end{equation*}
for all $i\in\{1,\dots,d\}$ and $x,y\geq 0$. Thirdly, the law-dependence of the drift coefficient captures interactions among a large number of homogeneously mixing subpopulations. In our analysis, we will mainly focus on the dependence of the disease transmission on the expected infection level, as illustrated by~\eqref{eq:standard SIS epidemic model with an expected value}.

To the best of our knowledge, the existing literature does not address two-level mixing compartmental epidemic models formulated via McKean--Vlasov SDEs driven by Brownian motions, such as \eqref{eq:SIS epidemic model}. Forien and Pardoux~\cite{ForPar22} study an SIS epidemic model structured by households of relatively small size and derive, as the number of households tends to infinity, a McKean--Vlasov SDE driven by Poisson processes via a propagation of chaos result. Ball, Sirl and Trapman~\cite{BalSirTra24} consider an interacting population structure, similar to our setting, and analyse the final outcome of an SIR epidemic within large subcommunities.

To evaluate the pathwise asymptotic behaviour of the epidemic, we will investigate three fundamental regimes: \emph{exponential extinction}, \emph{persistence above a threshold $x_{0} > 0$} and \emph{persistence around $x_{0}$}. While the first regime describes the eventual clearance of the pathogen at an exponential rate, the second implies that the number of infected individuals exceeds any level $\varepsilon\in ]0,x_{0}[$ infinitely often, ensuring the long-term prevalence of the disease. Lastly, persistence around $x_{0}$ characterises the sustained oscillation of the number of infected individuals near the equilibrium $x_{0}$, as considered in~\cite{GraGreHuMaoPan11, WanCaiDinGui18, CaiCaiMao19, CaiCaiMao19-2, BerLan22}. In Propositions~\ref{pr:extinction} and~\ref{pr:persistence} we provide sufficient conditions for the epidemic model~\eqref{eq:SIS epidemic model} to exhibit extinction and persistence. Furthermore, we pay particular attention to the representative model in Example~\ref{ex:a representative SIS epidemic model}. In doing so, we demonstrate that our results recover and extend the corresponding conditions in~\cite{GraGreHuMaoPan11, WanCaiDinGui18, CaiCaiMao19, CaiCaiMao19-2, BerLan22}.

For the numerical simulation of the epidemic model~\eqref{eq:SIS epidemic model}, we introduce an extended Euler--Maruyama scheme that applies to interacting particle systems, as considered in~\cite{BosTal97}. In this well-established approximation approach for McKean--Vlasov SDEs, every particle follows an SDE derived from the mean-field SDE in which the law of the solution is replaced by the empirical measure of the particle system. The SDEs are then approximated using an Euler--Maruyama method, and convergence follows from propagation of chaos as the number of particles tends to infinity. For example, the interacting particle system for the mean-field SDE~\eqref{eq:standard SIS epidemic model with an expected value} is given by the SDEs~\eqref{eq:standard SIS epidemic model with empirical means}.

Although~\cite{KloPla92} recalls that the standard Euler--Maruyama scheme achieves a strong convergence rate of $\frac{1}{2}$ in the step size for SDEs with Lipschitz continuous coefficients, the method may diverge if the drift coefficient grows superlinearly, as shown in~\cite{HutJenKlo11}. To address this issue, several modified numerical schemes have been developed, including the tamed, stopped and truncated Euler--Maruyama methods in~\cite{Mao15, HutJenKlo12, LiuMao13}. Concerning McKean--Vlasov SDEs with drifts of superlinear growth, propagation of chaos and strong convergence in $L^{2}$ are shown in~\cite{ReiEngSmi22} for both the tamed Euler--Maruyama method and an implicit scheme applied to the corresponding particle systems. 

Under a Lipschitz condition on the diffusion coefficient~\eqref{eq:diffusion coefficient}, for $p\geq 2$ we prove the strong $L^{p}$-convergence of the Euler--Maruyama scheme for an interacting particle system approximating~\eqref{eq:SIS epidemic model}, despite the polynomial growth of the drift. To this end, we provide a unified argument that treats propagation of chaos and time discretizations simultaneously, from which the \emph{explicit strong $L^{p}$-error bound} in Theorem~\ref{thm:strong error estimate} is derived. Under common regularity conditions, the classical strong convergence rate $\frac{1}{2}$ is recovered. Finally, we illustrate our theoretical results on the epidemiological dynamics numerically.\smallskip

{\bf Outline.}~Section~\ref{se:2} establishes the existence of a unique strong solution to~\eqref{eq:SIS epidemic model}  and provides quantitative moment estimates. Moreover, we introduce the representative model in Example~\ref{ex:a representative SIS epidemic model} for particular coefficients. Section~\ref{se:3} describes the pathwise asymptotic behaviour of the epidemic in terms of extinction and persistence. The Euler--Maruyama scheme for~\eqref{eq:SIS epidemic model} based on particle systems is developed and implemented in Section~\ref{se:4}. Lastly, Section~\ref{se:5} contains the proofs of all preliminary and main results.

\section{Derivation of a unique strong solution}\label{se:2}

In what follows, $|\cdot|$ stands for the absolute value function, the Euclidean norm or the Hilbert-Schmidt norm, depending on the context. Further, the Dirac measure at a point $x\in\R$ is denoted by $\delta_{x}$.

\subsection{A probabilistic analysis in the first moment}\label{se:2.1}

First, to ensure pathwise uniqueness for the McKean--Vlasov SDE~\eqref{eq:SIS epidemic model}, we introduce a \emph{boundedness and Lipschitz condition} on the functions $b_{0},\dots,b_{k}$:
\begin{enumerate}[label=(C.\arabic*), ref=C.\arabic*, leftmargin=\widthof{(C.1)} + \labelsep]
\item\label{co:1} For each $i\in\{0,\dots,k\}$ there are measurable functions $\hat{b}_{i}\colon\R_{+}\times\R_{+}\rightarrow\R$ and $\hat{\lambda}_{i}\colon\R_{+}\times\R_{+}\rightarrow\R_{+}$ such that
\begin{equation*}
b_{i}(s,y,\nu) \leq \hat{b}_{i}(s,y)\quad\text{and}\quad |b_{i}(s,y,\nu) - b_{i}(s,y,\tilde{\nu})| \leq \hat{\lambda}_{i}(s,y)\vartheta(\nu,\tilde{\nu})
\end{equation*}
for all $s\geq 0$, $y > 0$ and $\nu,\tilde{\nu}\in\mathcal{P}_{0}(\R)$. Further, $\hat{b}_{i}(\cdot,N)$ and $\hat{\lambda}_{i}(\cdot,N)$ are locally integrable and $b_{i}$ is bounded on the set
\begin{equation*}
\{(s,y,\nu)\in [0,n]\times [0,n]\times\mathcal{P}_{0}(\R)\mid \vartheta(\nu,\delta_{0})\leq n\}\quad\text{for any $n\in\N$.}
\end{equation*}
\end{enumerate}

\begin{Remark}\label{re:bounded on bounded sets}
By calling a non-empty set $\mathcal{P}$ in $\mathcal{P}_{0}(\R)$ bounded if $\sup_{\nu\in\mathcal{P}} \vartheta(\nu,\delta_{0}) < \infty$, the last condition in~\eqref{co:1} means that $b_{0},\dots,b_{k}$ are bounded on bounded sets.
\end{Remark}

Under~\eqref{co:1}, the drift coefficient $D\times\mathcal{P}_{0}(\R)\rightarrow\R$, $(t,x,\nu)\mapsto \sum_{i=0}^{k}b_{i}(t,N(t),\nu)x^{i}$ of~\eqref{eq:SIS epidemic model} satisfies the following \emph{partial Lipschitz condition}:
\begin{equation} \label{eq:partial Lipschitz condition on the drift coefficient}
\begin{split}
\sgn(x-\tilde{x})\sum_{i=0}^{k}\big(b_{i}(s,N(s),\nu)x^{i} &- b_{i}(s,N(s),\tilde{\nu})\tilde{x}^{i}\big)\\
&\leq \hat{b}_{k+1}(s,N(s))|x-\tilde{x}| + \hat{\lambda}_{k+1}(s,N(s))\vartheta(\nu,\tilde{\nu})
\end{split}
\end{equation}
for all $s\geq 0$, $x,\tilde{x}\in [0,N(s)]$ and $\nu,\tilde{\nu}\in\mathcal{P}_{0}(\R)$, where $\hat{b}_{k+1}\colon\R_{+}\times\R_{+}\rightarrow\R$ is any measurable function such that
\begin{equation}\label{eq:regularity coefficient 1 for the drift coefficient}
\max_{x\in [0,y]}\sum_{i=1}^{k}b_{i}(\cdot,y,\nu)ix^{i-1} \leq \hat{b}_{k+1}(\cdot,y)
\end{equation}
for all $y\geq 0$ and $\nu\in\mathcal{P}_{0}(\R)$ and $\hat{b}_{k+1}(\cdot,N)$ is locally integrable and the measurable function $\hat{\lambda}_{k+1}\colon\R_{+}\times\R_{+}\rightarrow\R_{+}$ is defined by
\begin{equation}\label{eq:regularity coefficient 2 for the drift coefficient}
\hat{\lambda}_{k+1}(s,y) := \sum_{i=0}^{k}\hat{\lambda}_{i}(s,y)y^{i}.
\end{equation}
For instance, we may take $\hat{b}_{k+1}(\cdot,y) = \max_{x\in [0,y]} \sum_{i=1}^{k}\hat{b}_{i}(\cdot,y) i x^{i-1}$ for all $y\geq 0$. Further, we require an \emph{$\frac{1}{2}$-Hölder continuity condition on compact sets} on the map $f$.
\begin{enumerate}[label=(C.\arabic*), ref=C.\arabic*, leftmargin=\widthof{(C.2)} + \labelsep]
\setcounter{enumi}{1}
\item\label{co:2} For each $n\in\N$ there is a measurable locally bounded function $\lambda_{n}\colon\R_{+}\rightarrow\R_{+}$ such that
\begin{equation*}
|f(s,x,y) - f(s,\tilde{x},\tilde{y})| \leq \lambda_{n}(s)\big(|x-\tilde{x}|^{\frac{1}{2}} + |y-\tilde{y}|^{\frac{1}{2}}\big)
\end{equation*}
for all $s\geq 0$ and $x,\tilde{x},y,\tilde{y}\in [0,n]$.
\end{enumerate}

Condition~\eqref{co:2} implies that the diffusion coefficient of~\eqref{eq:SIS epidemic model} that is given by~\eqref{eq:diffusion coefficient} satisfies the following \emph{$\frac{1}{2}$-Hölder continuity condition}:
\begin{equation*}
|f(s,x,N(s) - x) - f(s,\tilde{x},N(s)-\tilde{x})| \leq  2\lambda(s,N(s))|x-\tilde{x}|^{\frac{1}{2}}
\end{equation*}
for all $s\geq 0$ and $x,\tilde{x}\in [0,N(s)]$ with the $\R_{+}$-valued measurable locally bounded function $\lambda$ on $\R_{+}\times\R_{+}$ given by $\lambda(\cdot,y) = \lambda_{\lceil y \rceil}$ for $y > 0$ and $\lambda(\cdot,0) = 0$.

\begin{Example}[Sums of power functions]\label{ex:sums of power functions}
For $m\in\N$ and $\zeta,\eta\in [\frac{1}{2},\infty[^{d\times m}$ let $g$ be an $\R^{d\times m}$-valued measurable locally bounded map on $\R_{+}$ such that
\begin{equation*}
f_{i}(s,x,y) = g_{i,1}(s)x^{\zeta_{i,1}}y^{\eta_{i,1}} + \cdots + g_{i,m}(s)x^{\zeta_{i,m}}y^{\eta_{i,m}}
\end{equation*}
for all $i\in\{1,\dots,d\}$, $s\geq 0$ and $x,y\geq 0$. Then~\eqref{co:2} is satisfied. 
\end{Example}

These conditions ensure that~\eqref{eq:SIS epidemic model} admits at most a unique solution satisfying a given initial condition. In fact, for any $i\in\{0,\dots,k\}$ let us define a Borel measurable function $\overline{b}_{i}\colon\R_{+}\times\R_{+}\times\mathcal{P}_{p}(\R)\rightarrow\R$ and an extension $\overline{f}$ of $f$ to $\R_{+}\times\R\times\R$ by
\begin{equation}\label{eq:extended functions}
\overline{b}_{i}(s,y,\nu) := b_{i}\big(s,y,\nu\circ b_{k+1}(\cdot,y)^{-1}\big)\quad\text{and}\quad \overline{f}(s,x,y) := f(s,x^{+},y^{+})
\end{equation}
with the Lipschitz continuous function $b_{k+1}\colon\R\times\R_{+}\rightarrow\R_{+}$ given by $b_{k+1}(x,y) := x^{+}\wedge y$. We note that $\nu\circ b_{k+1}(\cdot,y)^{-1} = \nu$ for all $y\geq 0$ and $\nu\in\mathcal{P}_{p}(\R)$ with $\mathrm{supp}(\nu)\subset [0,y]$. Hence, any solution to~\eqref{eq:SIS epidemic model} solves the extended McKean--Vlasov SDE
\begin{equation}\label{eq:extended McKean--Vlasov SDE}
\mathrm{d}I_{t} = \sum_{i=0}^{k}\overline{b}_{i}\big(t,N_{t},\mathcal{L}(I_{t})\big)\big(I_{t}^{+}\wedge N(t)\big)^{i}\,\mathrm{d}t + \overline{f}(t,I_{t},N(t) - I_{t})\,\mathrm{d}W_{t}
\end{equation}
for $t\geq 0$. Since the drift and diffusion coefficients of~\eqref{eq:extended McKean--Vlasov SDE} are defined on $\R_{+}\times\R\times\mathcal{P}_{p}(\R)$ and $\R_{+}\times\R$, respectively, Corollary~3.1 and Theorem~4.2 in~\cite{KalMeyPro24-2} entail \emph{uniqueness} and a \emph{quantitative $L^{1}$-comparison estimate}.

\begin{Proposition}\label{pr:pathwise uniqueness}
Let $p=1$ and~\eqref{co:1} and~\eqref{co:2} be valid. Then pathwise uniqueness holds for~\eqref{eq:SIS epidemic model}, and any two solutions $I$ and $\tilde{I}$ to~\eqref{eq:SIS epidemic model} satisfy
\begin{equation}\label{eq:comparison estimate}
\E\big[|I_{t} - \tilde{I}_{t}|\big] \leq e^{\int_{0}^{t}(\hat{b}_{k+1} + \hat{\lambda}_{k+1})(s,N(s))\,\mathrm{d}s}\E\big[|I_{0} - \tilde{I}_{0}|\big]
\end{equation}
for all $t\geq 0$. In particular, if $(\hat{b}_{k+1} + \hat{\lambda}_{k+1})^{+}(\cdot,N)$ is integrable, then $\sup_{t\geq 0} \E[|I_{t} -\tilde{I}_{t}|]$ is finite. If additionally
\begin{equation*}
\int_{0}^{\infty}(\hat{b}_{k+1}+\hat{\lambda}_{k+1})^{-}(s,N(s))\,\mathrm{d}s = \infty,\quad\text{then}\quad \lim_{t\uparrow\infty} \E\big[|I_{t}-\tilde{I}_{t}|\big] = 0.
\end{equation*}
\end{Proposition}

Now we turn to the \emph{derivation of a unique strong solution} $I$ to~\eqref{eq:SIS epidemic model}  that satisfies the initial condition $I_{0} = \xi_{0}$ a.s. To this end, we observe that~\eqref{co:1} implies
\begin{equation}\label{eq:partial growth condition on the drift coefficient}
\sum_{i=0}^{k}b_{i}(s,N(s),\nu)x^{i} \leq \hat{b}_{0}(s,N(s)) + \hat{b}_{k+2}(s,N(s))x
\end{equation}
for all $s\geq 0$, $x\in [0,N(s)]$ and $\nu\in\mathcal{P}_{0}(\R)$. This is a \emph{partial affine growth condition} on the drift coefficient of~\eqref{eq:SIS epidemic model} and $\hat{b}_{k+2}\colon\R_{+}\times\R_{+}\rightarrow\R$ is any measurable function such that
\begin{equation}\label{eq:regularity coefficient 3 for the drift coefficient}
\max_{x\in [0,y]} \sum_{i=1}^{k}b_{i}(\cdot,y,\nu)x^{i-1} \leq \hat{b}_{k+2}(\cdot,y)
\end{equation}
for any $y\geq 0$ and $\nu\in\mathcal{P}_{0}(\R)$ and $\hat{b}_{k+2}(\cdot,N)$ is locally integrable. For example, $\hat{b}_{k+2}(\cdot,y)$ $= \max_{x\in [0,y]} \sum_{i=1}^{k}\hat{b}_{i}(\cdot,y)x^{i-1}$ for any $y\geq 0$ is possible.

In combination with~\eqref{co:1} and~\eqref{co:2}, the next condition ensures that any solution $\overline{I}$ to~\eqref{eq:extended McKean--Vlasov SDE} with $\overline{I}_{0} = \xi_{0}$ a.s.~satisfies $0 \leq \overline{I} \leq N$ a.s., which yields a solution to~\eqref{eq:SIS epidemic model}.
\begin{enumerate}[label=(C.\arabic*), ref=C.\arabic*, leftmargin=\widthof{(C.3)} + \labelsep]
\setcounter{enumi}{2}
\item\label{co:3} We have $b_{0}(\cdot,N,\nu)\geq 0$ and $
\dot{N} \geq \sum_{i=0}^{k}b_{i}(\cdot,N,\nu)N^{i}$ for any $\nu\in\mathcal{P}_{0}(\R)$ and $f(\cdot,0,y)$ $= f(\cdot,x,0) = 0$ for all $x,y\geq 0$.
\end{enumerate}

As will be shown in the proof of Proposition~\ref{pr:existence, uniqueness and a growth estimate}, if in addition $\xi_{0} = i_{0}$ a.s.~for some $i_{0}\in ]0,N(0)[$, then $0 < \overline{I} < N$ a.s.~once the subsequent condition is valid:
\begin{enumerate}[label=(C.\arabic*), ref=C.\arabic*, leftmargin=\widthof{(C.4)} + \labelsep]
\setcounter{enumi}{3}
\item\label{co:4} There are $\varepsilon > 0$, a measurable locally integrable function $c\colon\R_{+}\rightarrow\R_{+}$ and an increasing function $\varphi\colon ]0,\varepsilon]\rightarrow\R_{+}$ satisfying
\begin{align*}
\frac{1}{2}\frac{|f(\cdot,x,N-x)|^{2}}{x} &\leq b_{0}(\cdot,N,\nu) + c x\varphi(x),\\
\frac{1}{2}\frac{|f(\cdot,N-x,x)|^{2}}{x} &\leq \dot{N} - \sum_{i=0}^{k}b_{i}(\cdot,N,\nu)N^{i} + cx\varphi(x)
\end{align*}
for any $x\in ]0,\varepsilon[$ with $x\leq N$ and $\nu\in\mathcal{P}_{0}(\R)$.
\end{enumerate}

\begin{Remark}
By recalling the domain $D$ of the diffusion coefficient~\eqref{eq:diffusion coefficient}, we see that $0 < \overline{I}_{t}(\omega) < N(t)$ is equivalent to $(t,\overline{I}_{t}(\omega))\in D^{\circ}$ for any $t > 0$ and $\omega\in\Omega$.
\end{Remark}

\begin{Example}\label{ex:power sum condition}
Let the representation for $f$ in Example~\ref{ex:sums of power functions} hold, which involves sums of power functions. Then Lemma~\ref{le:verifiable estimates} verifies that~\eqref{co:4} follows from the estimates
\begin{equation}\label{eq:power sum condition}
\begin{split}
\frac{1}{2}\sum_{i=1}^{d}\sum_{\substack{j_{1},j_{2}=1,\\ \zeta_{i,j_{1}} + \zeta_{i,j_{2}} < 2}}^{m}(g_{i,j_{1}}g_{i,j_{2}})^{+} N^{\eta_{i,j_{1}} + \eta_{i,j_{2}}} &\leq b_{0}(\cdot,N,\nu)\quad\text{and}\\
\frac{1}{2}\sum_{i=1}^{d}\sum_{\substack{j_{1},j_{2}=1,\\ \eta_{i,j_{1}} + \eta_{i,j_{2}} < 2}}^{m}(g_{i,j_{1}}g_{i,j_{2}})^{+}N^{\zeta_{i,j_{1}} + \zeta_{i,j_{2}}}  &\leq \dot{N} - \sum_{i=0}^{k}b_{i}(\cdot,N,\nu)N^{i}
\end{split}
\end{equation}
for all $\nu\in\mathcal{P}_{0}(\R)$. In the particular case that $\zeta,\eta\in [1,\infty[^{d\times m}$, the left-hand terms in~\eqref{eq:power sum condition} vanish and we obtain the first two estimates in~\eqref{co:3}.
\end{Example}

From Theorems~3.1 and~4.2 in~\cite{KalMeyPro24-2} and Proposition~3.3 in~\cite{BriGraKal24} we now obtain an \emph{existence and uniqueness result} with an \emph{explicit $L^{1}$-growth estimate}.

\begin{Proposition}\label{pr:existence, uniqueness and a growth estimate}
Let $p=1$ and~\eqref{co:1}--\eqref{co:3} be satisfied. Then~\eqref{eq:SIS epidemic model} admits a unique strong solution $I$ such that $I_{0} = \xi_{0}$ a.s.~and
\begin{equation}\label{eq:first moment growth estimate}
\E\big[I_{t}\big] \leq e^{\int_{0}^{t}\hat{b}_{k+2}(s,N(s))\,\mathrm{d}s}\E\big[\xi_{0}\big] + \int_{0}^{t}e^{\int_{s}^{t}\hat{b}_{k+2}(\tilde{s},N(\tilde{s}))\,\mathrm{d}\tilde{s}}\hat{b}_{0}(s,N(s))\,\mathrm{d}s
\end{equation}
for all $t\geq 0$. Moreover, the following two assertions hold:
\begin{enumerate}[(i)]
\item If $\hat{b}_{k+2}^{+}(\cdot,N)$ and $\hat{b}_{0}(\cdot,N)$ are integrable, then $\sup_{t\geq 0}\E[I_{t}] < \infty$. In this case, we have $\lim_{t\uparrow\infty} \E[I_{t}] = 0$ as soon as $\int_{0}^{\infty}\hat{b}_{k+2}^{-}(s,N(s))\,\mathrm{d}s = \infty$.

\item If $\xi_{0} = i_{0}$ a.s.~and~\eqref{co:4} holds, then $0 < I < N$ a.s.
\end{enumerate}
\end{Proposition}

\begin{Example}[A tractable class of McKean--Vlasov SDEs]\label{ex:a tractable class of McKean--Vlasov SDEs}
For $k=3$ let $c_{0}\colon\R_{+}\rightarrow\R_{+}$ and $c\colon\R_{+}\rightarrow\R^{2\times 2}$ be measurable and locally bounded and
\begin{equation*}
\varphi_{0},\varphi_{1},\varphi_{2}\colon\R_{+}\times\R_{+}\rightarrow\R
\end{equation*}
be measurable, locally bounded and locally Lipschitz continuous in the second variable, locally uniformly in the first variable, such that
\begin{align*}
b_{0}(\cdot,y,\nu) &= c_{0} + y\int_{[0,y]}\varphi_{0}(\cdot,x)\,\nu(\mathrm{d}x),\\
b_{1}(\cdot,y,\nu) &= \int_{[0,y]}\varphi_{1}(\cdot,x)\,\nu(\mathrm{d}x) + c_{1,1}y + c_{1,2}y^{2},\\
b_{2}(\cdot,z,\nu) &= \frac{1}{z}\int_{[0,z]}\varphi_{2}(\cdot,x)\,\nu(\mathrm{d}x) - c_{1,1} + c_{2,1}z + c_{2,2}z^{2},\\
b_{2}(\cdot,0,\nu) &= - c_{1,1},\quad b_{3}(\cdot,y) = - c_{1,2} - c_{2,1} - c_{2,2}y
\end{align*}
for any $y \geq 0$, $z > 0$ and $\nu\in\mathcal{P}_{0}(\R)$, assuming that $b_{3}$ is simply defined on $\R_{+}\times\R_{+}$. Further, for $d=2$ let $\zeta\in [1,\infty[^{2}$ and $\eta\in [\frac{1}{2},\infty[^{2\times 2}$ be such that
\begin{equation*}
\quad f_{1}(\cdot,x,y) = x^{\zeta_{1}}(g_{1,1} y^{\eta_{1,1}} + g_{1,2} y^{\eta_{1,2}})\quad\text{and}\quad f_{2}(\cdot,x,y) = x^{\zeta_{2}}(g_{2,1}y^{\eta_{2,1}} + g_{2,2}y^{\eta_{2,2}})
\end{equation*}
for all $x,y\geq 0$ and some measurable locally bounded map $g\colon\R_{+}\rightarrow\R^{2\times 2}$. Then the McKean--Vlasov SDE~\eqref{eq:SIS epidemic model} takes the form
\begin{equation}\label{eq:a class of McKean--Vlasov SDEs}
\begin{split}
\mathrm{d}I_{t} &= \big(c_{0}(t) + N(t)\E\big[\varphi_{0}(t,I_{t})\big] + \big(\E\big[\varphi_{1}(t,I_{t})\big] + c_{1,1}(t)N(t) + c_{1,2}(t)N(t)^{2}\big)I_{t}\big)\,\mathrm{d}t\\
&\quad + \bigg(\frac{\E\big[\varphi_{2}(t,I_{t})\big]}{N(t)} - c_{1,1}(t) + c_{2,1}(t)N(t) + c_{2,2}(t)N(t)^{2}\bigg)I_{t}^{2}\,\mathrm{d}t\\
&\quad - \big(c_{1,2}(t) + c_{2,1}(t) + c_{2,2}(t)N(t)\big)I_{t}^{3}\,\mathrm{d}t\\
&\quad + I_{t}^{\zeta_{1}}\big(g_{1,1}(t)(N(t) - I_{t})^{\eta_{1,1}} + g_{1,2}(t)\big(N(t) - I_{t})^{\eta_{1,2}}\big)\,\mathrm{d}W_{t}^{(1)}\\
&\quad + I_{t}^{\zeta_{2}}\big(g_{2,1}(t)(N(t) - I_{t})^{\eta_{2,1}} + g_{2,2}(t)(N(t) - I_{t})^{\eta_{2,2}}\big)\,\mathrm{d}W_{t}^{(2)}
\end{split}
\end{equation}
for $t\geq 0$. We suppose that $N > 0$, $\varphi_{0}\geq 0$, $\varphi_{2}(\cdot,0) = 0$ and
\begin{equation}\label{eq:specific value condition}
\dot{N} \geq c_{0} + (\varphi_{0} + \varphi_{1} + \varphi_{2})(\cdot,x)N\quad\text{for all $x\in [0,N]$.}
\end{equation}
Then Proposition~\ref{pr:existence, uniqueness and a growth estimate} applies, since~\eqref{co:1}--\eqref{co:3} are satisfied for $\vartheta = \mathcal{W}_{1}$. Indeed, in the boundedness and Lipschitz condition~\eqref{co:1} we may take
\begin{equation}
\begin{split}\label{eq:specific regularity coefficients 1 for the drift coefficient}
\hat{b}_{0}(\cdot,y) &= c_{0} + y\max_{x\in [0,y]}\varphi_{0}(\cdot,x),\quad \hat{b}_{1}(\cdot,y) = \max_{x\in [0,y]} \varphi_{1}(\cdot,x) + c_{1,1}y + c_{1,2}y^{2},\\
\hat{b}_{2}(\cdot,z) &= \frac{1}{z}\max_{x\in [0,y]}\varphi_{2}(\cdot,x) - c_{1,1} + c_{2,1}z + c_{2,2}z^{2},\quad  \hat{b}_{3} = b_{3}
\end{split}
\end{equation}
for any $y \geq 0$ and $z > 0$. Further, for any $i\in\{0,1,2\}$ there is $L_{i}\colon\R_{+}\times\R_{+}\rightarrow\R_{+}$ such that $|\varphi_{i}(s,x) - \varphi(s,\tilde{x})| \leq L_{i}(s,y)|x - \tilde{x}|$ for any $s,y\geq 0$ and $x,\tilde{x}\in [0,y]$, and we take
\begin{equation}\label{eq:specific regularity coefficients 2 for the drift coefficient}
\hat{\lambda}_{0}(\cdot,y) = yL_{0}(\cdot,y),\quad \hat{\lambda}_{1} = L_{1},\quad \hat{\lambda}_{2}(\cdot,z) = \frac{L_{2}(\cdot,z)}{z},\quad \hat{\lambda}_{3} = 0
\end{equation}
for all $y\geq 0$ and $z > 0$. The Hölder condition~\eqref{co:2} holds, by Example~\ref{ex:sums of power functions}, and~\eqref{co:3} follows from~\eqref{eq:specific value condition}. Further, if~\eqref{eq:specific value condition} is replaced by the sharper constraint
\begin{equation}\label{eq:specific strict value condition}
\begin{split}
&\dot{N}\geq c_{0} + (\varphi_{0} + \varphi_{1} + \varphi_{2})(\cdot,x)N\\
&\quad + \frac{1}{2}N^{2\zeta_{1}}\big(g_{1,1}^{2}\mathbbm{1}_{[0,1[}(\eta_{1,1}) + g_{1,2}^{2}\mathbbm{1}_{[0,1[}(\eta_{1,2})\big) + N^{2\zeta_{1}}(g_{1,1}g_{1,2})^{+}\mathbbm{1}_{[0,2[}(\eta_{1,1} + \eta_{1,2})\\
&\quad + \frac{1}{2}N^{2\zeta_{2}}\big(g_{2,1}^{2}\mathbbm{1}_{[0,1[}(\eta_{2,1}) + g_{2,2}^{2}\mathbbm{1}_{[0,1[}(\eta_{2,2})\big) + N^{2\zeta_{2}}(g_{2,1}g_{2,2})^{+}\mathbbm{1}_{[0,2[}(\eta_{2,1} + \eta_{2,2}),
\end{split}
\end{equation}
for each $x\in [0,N]$, then~\eqref{co:4} is satisfied, according to Example~\ref{ex:power sum condition}.
\end{Example}

Proposition~\ref{pr:existence, uniqueness and a growth estimate} generalises the existence and uniqueness results of a range of widely studied SIS models. Namely, Theorem~2.1 in~\cite{WanCaiDinGui18}, Theorem~3.1 in~\cite{GraGreHuMaoPan11}, Theorem~2.1 in~\cite{CaiCaiMao19,CaiCaiMao19-2} and the first assertion of Theorem~1 in~\cite{BerLan22} emerge as special cases, as shown below.

\begin{Example}[A representative SIS epidemic model]\label{ex:a representative SIS epidemic model}
Let the setting of Example~\ref{ex:a tractable class of McKean--Vlasov SDEs} be valid and assume that $c_{0} = 0$, $c_{1,1} = \beta$, $\zeta = (1,1)$, $g_{2,2} = 0$ and
\begin{align*}
\varphi_{0}(\cdot,x) &= \beta_{0}x,\quad \varphi_{1}(\cdot,x) = - (\mu + \gamma) + (\beta_{1} - \beta_{0})x,\quad \varphi_{2}(\cdot,x) = - \beta_{1} x,
\end{align*}
for all $x\in\R$, where $\beta_{0},\beta,\gamma,\mu\colon\R_{+}\rightarrow\R_{+}$ and $\beta_{1}\colon\R_{+}\rightarrow\R$ are measurable and locally bounded. Then the McKean--Vlasov SDE~\eqref{eq:a class of McKean--Vlasov SDEs} can be written in the form
\begin{equation}\label{eq:a representative SIS epidemic model}
\begin{split}
\mathrm{d}I_{t} &= \bigg(\bigg(\beta_{0}(t)\E\big[I_{t}\big] + \bigg(\beta(t) + \beta_{1}(t)\frac{\E[I_{t}]}{N(t)}\bigg)I_{t}\bigg)(N(t) - I_{t}) - (\mu + \gamma)(t)I_{t}\bigg)\,\mathrm{d}t\\
&\quad + \big(c_{1,2}(t)N(t)^{2}I_{t} + \big(c_{2,1}(t) + c_{2,2}(t)N(t)\big)N(t)I_{t}^{2}\big)\,\mathrm{d}t\\
&\quad - \big(c_{1,2}(t) + c_{2,1}(t) + c_{2,2}(t)N(t)\big)I_{t}^{3}\,\mathrm{d}t\\
&\quad + I_{t}\big(g_{1,1}(t)(N(t) - I_{t})^{\eta_{1,1}} + g_{1,2}(t)\big(N(t) - I_{t})^{\eta_{1,2}}\big)\,\mathrm{d}W_{t}^{(1)}\\
&\quad + I_{t}g_{2,1}(t)(N(t) - I_{t})^{\eta_{2,1}}\,\mathrm{d}W_{t}^{(2)}
\end{split}
\end{equation}
for $t\geq 0$ and~\eqref{eq:specific value condition} reduces to $\dot{N} \geq - (\mu + \gamma)N$. In particular, for $\beta_{0} = c_{2,2} = 0$, $\eta_{1,1} = 1$ and $\eta_{1,2} = \eta_{2,1} = \eta_{0}$ with $\eta_{0}\in\{\frac{1}{2},1\}$, we obtain extended time- and distribution-dependent versions of the following SIS models:
\begin{enumerate}[label=(M.\arabic*), ref=M.\arabic*, leftmargin=\widthof{(M.3)} + \labelsep]
\item\label{model:1} The \emph{model by Wang, Cai, Ding and Gui} defined via the SDE~(1.13) in~\cite{WanCaiDinGui18} for $\beta(t)$ $= \beta_{e} + (\beta_{0} - \beta_{e})e^{-\theta t}$, $c_{1,2} = 0$, $c_{2,1} = 0$,
\begin{equation*}
\eta_{0} = 1,\quad g_{1,1}(t) = \frac{\xi}{\sqrt{2\theta}}\sqrt{1-e^{-2\theta t}}\quad \text{and}\quad g_{1,2}=g_{2,1}= 0
\end{equation*}
for $t\geq 0$ with $\beta_{0},\beta_{e}\geq 0$ and $\xi,\theta > 0$. Then~\eqref{eq:a representative SIS epidemic model} turns into
\begin{equation*}
\begin{split}
\mathrm{d}I_{t} &= \bigg(\bigg(\beta_{e} + (\beta_{0} - \beta_{e})e^{-\theta t} + \beta_{1}(t)\frac{\E[I_{t}]}{N(t)}\bigg)I_{t}(N(t) - I_{t}) -(\mu + \gamma)(t)I_{t}\bigg)\,\mathrm{d}t\\
&\quad + I_{t}\frac{\xi}{\sqrt{2\theta}}\sqrt{1-e^{-2\theta t}}(N(t) - I_{t})\,\mathrm{d}W_{t}^{(1)}
\end{split}
\end{equation*}
for $t\geq 0$. As $\lim_{t\uparrow\infty} \beta(t) = \beta_{e}$ and $\lim_{t\uparrow\infty}g_{1,1}(t) = \frac{\xi}{\sqrt{2\theta}}$, there is a direct relation to the SDE~\eqref{eq:standard SIS epidemic model} describing the popular model~\cite{GraGreHuMaoPan11}.

\item\label{model:2} The \emph{model by Cai, Cai and Mao} determined by the SDE~(5) in~\cite{CaiCaiMao19-2}, stemming from two correlated Brownian motions, for $c_{1,2}= c_{2,1} = 0$,
\begin{equation*}
\eta_{0} = \frac{1}{2},\quad g_{1,1}= a_{1}\sigma_{1},\quad g_{1,2} = -a_{2}\sigma_{2}\quad\text{and}\quad g_{2,1} = -a_{3}\sigma_{2},
\end{equation*}
where $a\colon\R_{+}\rightarrow\R^{3}$ and $\sigma\colon\R_{+}\rightarrow\R_{+}^{2}$ are measurable and locally bounded such that $a_{1},a_{3}\geq 0$. Then~\eqref{eq:a representative SIS epidemic model} becomes
\begin{equation*}
\begin{split}
\mathrm{d}I_{t} &= \bigg(\bigg(\beta(t) + \beta_{1}(t)\frac{\E[I_{t}]}{N(t)}\bigg)I_{t}(N(t) - I_{t}) -(\mu + \gamma)(t)I_{t}\bigg)\,\mathrm{d}t\\
&\quad + I_{t}\bigg((a_{1}\sigma_{1})(t)(N(t) -I_{t}) - (a_{2}\sigma_{2})(t)\sqrt{N(t) -I_{t}}\bigg)\,\mathrm{d}W_{t}^{(1)}\\
&\quad - I_{t}(a_{3}\sigma_{2})(t)\sqrt{N(t) -I_{t}}\,\mathrm{d}W_{t}^{(2)}
\end{split}
\end{equation*}
for $t\geq 0$. Specifically, for $a = (1,0,1)$ we obtain an extension of the SDE~(1.5) in~\cite{CaiCaiMao19}, which reduces to the SDE~(2.4) in~\cite{GraGreHuMaoPan11} if in addition $\sigma_{2} = 0$.

\item\label{model:3} The \emph{model by Bernardi and Lanconelli} given by the SDE~(1.9) in~\cite{BerLan22} for $c_{1,2}$ $= \frac{1}{2}\sigma^{2}$, $c_{2,1} = - \frac{3}{2}\sigma^{2}$,
\begin{equation*}
\eta_{0} = 1,\quad g_{1,1} = \sigma\quad\text{and}\quad g_{1,2}= g_{2,1} = 0,
\end{equation*}
for some measurable locally bounded function $\sigma\colon\R_{+}\rightarrow\R_{+}$. In this case,~\eqref{eq:a representative SIS epidemic model} is of the form
\begin{equation*}
\begin{split}
\mathrm{d}I_{t} &= \bigg(\bigg(\beta(t) + \beta_{1}(t)\frac{\E[I_{t}]}{N(t)}\bigg)I_{t}(N(t) - I_{t}) -(\mu + \gamma)(t)I_{t}\bigg)\,\mathrm{d}t\\
&\quad + \frac{1}{2}\sigma(t)^{2}I_{t}(N(t) - I_{t})(N - 2I_{t})\,\mathrm{d}t + \sigma(t)I_{t}(N(t) - I_{t})\,\mathrm{d}W_{t}^{(1)}\quad\text{for $t\geq 0$.}
\end{split}
\end{equation*}
\end{enumerate}
While in models~\eqref{model:1} and~\eqref{model:3} the constraint~\eqref{eq:specific strict value condition} agrees with $\dot{N}$ $\geq -(\mu + \gamma)N$, in~\eqref{model:2} it reduces to
\begin{equation*}
\dot{N} + (\mu + \gamma)N \geq \bigg(\frac{1}{2}\big(a_{2}^{2} + a_{3}^{2}\big)\sigma_{2} + (a_{1}a_{2})^{-}\sigma_{1}\bigg)\sigma_{2}N^{2}.
\end{equation*}
In particular, we recover the condition $\mu + \gamma$ $\geq \frac{1}{2}(a_{2}^{2} + a_{3}^{2})\sigma_{2}^{2}N$ in~\cite[Theorem 2.1]{CaiCaiMao19-2}, where $N$ is constant and $a_{1}a_{2}\geq 0$ is implicitly required.
\end{Example}

Finally, we determine the coefficients $\hat{b}_{k+1}$, $\hat{\lambda}_{k+1}$ and $\hat{b}_{k+2}$ in~\eqref{eq:regularity coefficient 1 for the drift coefficient},~\eqref{eq:regularity coefficient 2 for the drift coefficient} and~\eqref{eq:regularity coefficient 3 for the drift coefficient}.

\begin{Example}
Let the setting of Example~\ref{ex:a tractable class of McKean--Vlasov SDEs} hold and suppose that $c_{2,2},c_{2,1} + c_{1,2}\leq 0$, as in models~\eqref{model:1}--\eqref{model:3} in Example~\ref{ex:a representative SIS epidemic model}. Then Lemma~\ref{le:maximisation 1} allows us to take
\begin{align*}
&\hat{b}_{4}(\cdot,y) = \hat{b}_{1}(\cdot,y)\\
&\quad + \bigg(\max_{x\in [0,y]}\varphi_{1}(\cdot,x) - \max_{x\in [0,y]}(\varphi_{1} + 2\varphi_{2})(\cdot,x) + \big(2c_{1,1} + (3c_{1,2} + c_{2,1})y + c_{2,2}y^{2}\big)y\bigg)^{-},\\
&\hat{b}_{5}(\cdot,y) = \hat{b}_{1}(\cdot,y) + \bigg(\max_{x\in [0,y]}\varphi_{1}(\cdot,x) - \max_{x\in [0,y]}(\varphi_{1} + \varphi_{2})(\cdot,x) + \big(c_{1,1} + c_{1,2}y\big)y\bigg)^{-},\\
&\hat{\lambda}_{4}(\cdot,y) = \big(L_{0} + L_{1} + L_{2}\big)(\cdot,y)y
\end{align*}
for all $y\geq 0$ with the functions $\hat{b}_{1}$, $L_{0}$, $L_{1}$ and $L_{2}$ in~\eqref{eq:specific regularity coefficients 1 for the drift coefficient} and~\eqref{eq:specific regularity coefficients 2 for the drift coefficient}. In particular, in the setting of Example~\ref{ex:a representative SIS epidemic model} these formula reduce to
\begin{align*}
\hat{b}_{4}(\cdot,y) &= \hat{b}_{1}(\cdot,y) + \big((\beta_{1} - \beta_{0})^{+} - (\beta_{0} + \beta_{1})^{-} + 2\beta + (3c_{1,2} + c_{2,1})y + c_{2,2}y^{2}\big)^{-}y,\\
\hat{b}_{5}(\cdot,y) &= \hat{b}_{1}(\cdot,y) + \big((\beta_{1} - \beta_{0})^{+} - \beta_{0}^{-} + \beta + c_{1,2}y\big)^{-}y,\\
\hat{\lambda}_{4}(\cdot,y) &= \big(|\beta_{0}| + |\beta_{1} - \beta_{0}| + |\beta_{1}|\big)y
\end{align*}
for any $y\geq 0$. If in addition, $\beta_{0} = c_{2,2} = 0$, $c_{1,2} \geq 0$ and $3c_{1,2} + c_{2,1}\geq 0$, as in~\eqref{model:1}--\eqref{model:3}, then $\hat{b}_{4} = \hat{b}_{5} = \hat{b}_{1}$ as soon as $\beta_{1} \geq -2\beta$.
\end{Example}

\subsection{A probabilistic analysis in arbitrary higher moments}\label{se:2.2}
 
In this section, we consider the case $p\geq 2$ in~\eqref{eq:domination condition} and set $c_{p} := \frac{p-1}{2}$. First, we impose a \emph{Lipschitz continuity condition on compact sets} on $f$, implying~\eqref{co:2}.
\begin{enumerate}[label=(C.\arabic*), ref=C.\arabic*, leftmargin=\widthof{(C.5)} + \labelsep]
\setcounter{enumi}{4}
\item\label{co:5} For any $n\in\N$ there is a measurable locally bounded function $\lambda_{n}\colon\R_{+}\rightarrow\R_{+}$ such that
\begin{equation*}
|f(s,x,y) - f(s,\tilde{x},\tilde{y})| \leq \lambda_{n}(s)\big(|x-\tilde{x}| + |y-\tilde{y}|\big)
\end{equation*}
for all $s\geq 0$ and $x,\tilde{x},y,\tilde{y}\in [0,n]$.
\end{enumerate}

Assuming~\eqref{co:5}, the diffusion coefficient~\eqref{eq:diffusion coefficient} of~\eqref{eq:SIS epidemic model} satisfies the following \emph{Lipschitz continuity condition}:
\begin{equation}\label{eq:Lipschitz condition on the diffusion coefficient}
|f(s,x,N(s) - x) - f(s,\tilde{x},N(s)-\tilde{x})| \leq  \lambda(s,N(s))|x-\tilde{x}|
\end{equation}
for all $s\geq 0$ and $x,\tilde{x}\in [0,N(s)]$, where the measurable and locally bounded function $\lambda\colon\R_{+}\times\R_{+}\rightarrow\R_{+}$ can be defined by $\lambda(\cdot,y) := 2\lambda_{\lceil y \rceil}$ for $y > 0$ and $\lambda(\cdot,0) := 0$.

Based on these considerations, from Corollary~3.5 and Theorem~4.6 in~\cite{KalMeyPro24} we may infer \emph{uniqueness} and a \emph{quantitative $L^{p}$-comparison estimate}.

\begin{Proposition}\label{pr:pathwise uniqueness 2}
Let $p\geq 2$ and~\eqref{co:1} and~\eqref{co:5} hold. Then pathwise uniqueness holds for~\eqref{eq:SIS epidemic model}, and any two solutions $I$ and $\tilde{I}$ to~\eqref{eq:SIS epidemic model} satisfy
\begin{equation*}
\E\big[|I_{t} - \tilde{I}_{t}|^{p}\big] \leq e^{\int_{0}^{t}\overline{\lambda}_{p}(s,N(s))\,\mathrm{d}s}\E\big[|I_{0} - \tilde{I}_{0}|^{p}\big]
\end{equation*}
for all $t\geq 0$ with $\overline{\lambda}_{p} := p(\hat{b}_{k+1} + \hat{\lambda}_{k+1} + c_{p}\lambda^{2})$. In particular, the integrability of $\overline{\lambda}_{p}^{+}(\cdot,N)$ entails that $\sup_{t\geq 0}\E[|I_{t} - \tilde{I}_{t}|^{p}] < \infty$. In this case,
\begin{equation*}
\int_{0}^{\infty}\overline{\lambda}_{p}^{-}(s,N(s))\,\mathrm{d}s = \infty\quad\text{implies}\quad \lim_{t\uparrow\infty} \E\big[|I_{t} - \tilde{I}_{t}|^{p}\big] = 0.
\end{equation*}
\end{Proposition}

Next,~\eqref{co:5} and the hypothesis $f(\cdot,0,y) = 0$ for all $y\geq 0$ from~\eqref{co:3} yield a \emph{linear growth condition} for the diffusion coefficient of~\eqref{eq:SIS epidemic model}. Specifically,
\begin{equation}\label{eq:linear growth condition on the diffusion coefficient}
|f(s,x,N(s)-x)| \leq l(s,N(s))x
\end{equation}
for all $s\geq 0$ and $x\in [0,N(s)]$, where the measurable and locally bounded function $l\colon\R_{+}\times\R_{+}\rightarrow\R_{+}$ is given by $l(\cdot,y) := \lambda_{\lceil y \rceil}$ for $y > 0$ and $l(\cdot,0) := 0$, for instance.

\begin{Example}\label{ex:sums of power functions 2}
Let the representation in Example~\ref{ex:sums of power functions} hold with $\zeta,\eta\in [1,\infty[^{d\times m}$. Then~\eqref{co:5} is valid and feasible choices for the functions in~\eqref{eq:Lipschitz condition on the diffusion coefficient} and~\eqref{eq:linear growth condition on the diffusion coefficient} are
\begin{align}\label{eq:specific Lipschitz coefficient for the diffusion coefficient}
\lambda(\cdot,y) &= \max_{x\in [0,y]}\bigg(\sum_{i=1}^{d}\bigg(\sum_{j=1}^{m}g_{i,j}\big(\zeta_{i,j}x^{\zeta_{i,j} - 1}(y-x)^{\eta_{i,j}} - \eta_{i,j}x^{\zeta_{i,j}}(y-x)^{\eta_{i,j}-1}\big)\bigg)^{2}\bigg)^{\frac{1}{2}},\\\label{eq:specific linear growth coefficient for the diffusion coefficient}
l(\cdot,y) &= \max_{x\in [0,y]} \bigg(\sum_{i=1}^{d}\bigg(\sum_{j=1}^{m}g_{i,j} x^{\zeta_{i,j} - 1}(y-x)^{\eta_{i,j}}\bigg)^{2}\bigg)^{\frac{1}{2}}
\end{align}
for any $y\geq 0$, as the mean value theorem shows in the first case.
\end{Example}

Theorems~3.24 and~4.6 in~\cite{KalMeyPro24}, together with Proposition~3.3 in~\cite{BriGraKal24}, yield another \emph{existence and uniqueness result} with an \emph{explicit $L^{p}$-growth estimate}.

\begin{Proposition}\label{pr:existence, uniqueness and a growth estimate 2}
Let $p\geq 2$ and~\eqref{co:1}, \eqref{co:3} and \eqref{co:5} be satisfied. Then there is a unique strong solution $I$ to~\eqref{eq:SIS epidemic model} such that $I_{0} = \xi_{0}$ a.s.~and
\begin{equation}\label{eq:pth moment growth estimate}
\E\big[I_{t}^{p}\big] \leq e^{\int_{0}^{t}\overline{l}_{p}(s,N(s))\,\mathrm{d}s}\E\big[\xi_{0}^{p}\big] + \int_{0}^{t}e^{\int_{s}^{t}\overline{l}_{p}(\tilde{s},N(\tilde{s}))\,\mathrm{d}\tilde{s}}\hat{b}_{0}(s,N(s))\,\mathrm{d}s
\end{equation}
for all $t\geq 0$ with $\overline{l}_{p} := (p-1)\hat{b}_{0} + p(\hat{b}_{k+2} + c_{p}l^{2})$. Furthermore, the following two assertions hold:
\begin{enumerate}[(i)]
\item The integrability of $\overline{l}_{p}^{+}(\cdot,N)$ and $\hat{b}_{0}(\cdot,N)$ implies that $\sup_{t\geq 0}\E[I_{t}^{p}]$ is finite. If in addition $\int_{0}^{\infty}\overline{l}_{p}^{-}(s,N(s))\,\mathrm{d}s = \infty$, then $\lim_{t\uparrow\infty} \E[I_{t}^{p}] = 0$.

\item If $\xi_{0}  = i_{0}$ a.s.~and~\eqref{co:4} is valid, then $0 < I < N$ a.s.
\end{enumerate}
\end{Proposition}

Since Propositions~\ref{pr:pathwise uniqueness 2} and~\ref{pr:existence, uniqueness and a growth estimate 2} are applicable to Example~\ref{ex:a tractable class of McKean--Vlasov SDEs}, we derive explicit expressions for the coefficients $\lambda$ and $l$ in~\eqref{eq:specific Lipschitz coefficient for the diffusion coefficient} and~\eqref{eq:specific linear growth coefficient for the diffusion coefficient}.

\begin{Example}
Let the setting of Example~\ref{ex:a tractable class of McKean--Vlasov SDEs} hold when $\zeta = (1,1)$ and there is $\eta_{0}\geq 1$ such that $\eta_{i,j} = \eta_{0}$ for all $i,j\in\{1,2\}$ with $g_{i,j}\neq 0$. Then~\eqref{co:5} holds and we have
\begin{equation*}
\lambda(\cdot,y) = \big((g_{1,1} + g_{1,2})^{2} + (g_{2,1} + g_{2,2})^{2}\big)^{\frac{1}{2}}y^{\eta_{0}}\quad\text{and}\quad l = \lambda
\end{equation*}
for all $y\geq 0$. In particular, the formulas are valid in models~\eqref{model:1}--\eqref{model:3} from Example~\ref{ex:a representative SIS epidemic model}, provided that $\sigma_{2} = 0$ in~\eqref{model:2}, as otherwise~\eqref{co:5} may be violated.
\end{Example}

\section{Pathwise asymptotic behaviour}\label{se:3}

\subsection{Epidemiological notions and preliminary almost sure identities}\label{se:3.1}

We characterise the long-term behaviour of the epidemic in terms of the following three notions of pathwise asymptotics, assuming that $N$ takes positive values only.

\begin{Definition}
Let $I$ be a solution to~\eqref{eq:SIS epidemic model} with positive paths and $\rho,x_{0} > 0$.
\begin{enumerate}[(i)]
\item We say that $I$ becomes \emph{extinct with exponential rate $\rho$} if 
\begin{equation}\label{eq:extinction with an exponential rate}
\limsup_{t\uparrow\infty} \frac{1}{t^{\rho}}\log(I_{t}) < 0\quad\text{a.s.}
\end{equation}
In the case $\rho = 1$ we will simply say that $I$ becomes \emph{exponentially extinct}.

\item The solution $I$ is said to \emph{persist above $x_{0}$} if $\limsup_{t\uparrow\infty}I_{t}\geq x_{0}$ a.s. Furthermore, \emph{persistence around $x_{0}$} holds for $I$ if
\begin{equation}\label{eq:persistence around a point}
\liminf_{n\uparrow\infty} I_{t} \leq x_{0} \leq \limsup_{t\uparrow\infty}I_{t} \quad\text{a.s}
\end{equation}
\end{enumerate}
\end{Definition}

Our analysis relies on a \emph{linear growth condition on compact sets} for $f$, which follows immediately from~\eqref{co:5} and the requirement $f(\cdot,0,y) = 0$ for all $y\geq 0$ in~\eqref{co:3}.
\begin{enumerate}[label=(C.\arabic*), ref=C.\arabic*, leftmargin=\widthof{(C.6)} + \labelsep]
\setcounter{enumi}{5}
\item\label{co:6} For each $n\in\N$ there is a measurable locally bounded function $l_{n}\colon\R_{+}\rightarrow\R_{+}$ such that $|f(s,x,y)| \leq l_{n}(s) x$ for any $s\geq 0$ and $x,y\in [0,n]$.
\end{enumerate}

From~\eqref{co:6} we obtain the \emph{linear growth condition}~\eqref{eq:linear growth condition on the diffusion coefficient}, where the measurable and locally bounded function $l\colon\R_{+}\times\R_{+}\rightarrow\R_{+}$ is defined by $l(\cdot,y) := l_{\lceil y\rceil}$ for $y > 0$ and $l(\cdot,0) := 0$.

Now It{\^o}'s formula and the strong law of large numbers for continuous local martingales in~\cite[Theorem~1]{Lip80} imply a \emph{preliminary pathwise asymptotic result}. This involves the measurable function $h\colon\R_{+}\times\{(x,y)\in\R_{+}\times\R_{+}\mid 0 < x\leq y\}\times\mathcal{P}_{0}(\R)\rightarrow\R$,
\begin{equation}\label{eq:transformed function}
h(s,x,y,\nu):= \sum_{i=0}^{k}b_{i}(s,y,\nu)x^{i-1} - \frac{1}{2x^{2}}|f(s,x,y-x)|^{2}.
\end{equation}

\begin{Lemma}\label{le:pathwise asymptotic identities}
Let~\eqref{co:6} hold and $\rho > 0$ be such that
\begin{equation}\label{eq:integrability condition}
\int_{0}^{\infty}\frac{l(s,N(s))^{2}}{1 + s^{2\rho}}\,\mathrm{d}s < \infty.
\end{equation}
Then any solution $I$ to~\eqref{eq:SIS epidemic model} with positive paths satisfies
\begin{align*}
\liminf_{t\uparrow\infty}\frac{1}{t^{\rho}}\log(I_{t}) &= \liminf_{t\uparrow\infty} \frac{1}{t^{\rho}}\int_{0}^{t}h\big(s,I_{s},N(s),\mathcal{L}(I_{s})\big)\,\mathrm{d}s\\
 &\leq \limsup_{t\uparrow\infty} \frac{1}{t^{\rho}}\int_{0}^{t}h\big(s,I_{s},N(s),\mathcal{L}(I_{s})\big)\,\mathrm{d}s = \limsup_{t\uparrow\infty}\frac{1}{t^{\rho}}\log(I_{t})\quad\text{a.s.}
\end{align*}
\end{Lemma}

\begin{Example}
Assuming the representation of $f$ from Example~\ref{ex:sums of power functions}, it follows readily from Example~\ref{ex:sums of power functions 2} that~\eqref{eq:integrability condition} is satisfied whenever
\begin{equation}\label{eq:integrability condition 2}
\int_{0}^{\infty}\frac{g_{i,j}(s)^{2}}{1 + s^{2\rho}}N(s)^{2(\zeta_{i,j} + \eta_{i,j} - 1)}\,\mathrm{d}s < \infty
\end{equation}
for all $i\in\{1,\dots,d\}$ and $j\in\{1,\dots,m\}$. If $N$ is bounded, then~\eqref{eq:integrability condition 2} holds once there is $\alpha_{i,j} < \rho - \frac{1}{2}$ such that $g_{i,j}(s) = O(s^{\alpha_{i,j}})$ as $s\uparrow\infty$, or more specifically, 
\begin{equation*}
\text{$g_{i,j}$ is bounded}\quad\text{and}\quad \rho > \frac{1}{2}.
\end{equation*}
For the time- and distribution-dependent SIS models~\eqref{model:1}--\eqref{model:3} in Example~\ref{ex:a representative SIS epidemic model}, the integrability condition~\eqref{eq:integrability condition 2} takes the following respective forms:
\begin{enumerate}[(1)]
\item $\int_{0}^{\infty}\frac{1 - e^{-2\theta s}}{1 + s^{2\rho}}N(s)^{2}\,\mathrm{d}s < \infty$.

\item $\int_{0}^{\infty}\frac{a_{1}(s)^{2}\sigma_{1}(s)^{2}}{1 + s^{2\rho}}N(s)^{2}\,\mathrm{d}s < \infty$ and $\int_{0}^{\infty}\frac{a_{i}(s)^{2}\sigma_{2}(s)^{2}}{1 + s^{2\rho}}N(s)\,\mathrm{d}s < \infty$ for $i\in\{2,3\}$.

\item $\int_{0}^{\infty}\frac{\sigma(s)^{2}}{1 + s^{2\rho}}N(s)^{2}\,\mathrm{d}s < \infty$.
\end{enumerate}
In particular, for $\rho > \frac{1}{2}$ these conditions are satisfied in the original models~\cite{WanCaiDinGui18, CaiCaiMao19-2, BerLan22}, which consist of autonomous SDEs.
\end{Example}

Lastly, we provide the explicit form of the function~\eqref{eq:transformed function} for the generalised SIS models.

\begin{Example}\label{ex:transformed function in the representative model}
In the setting of Example~\ref{ex:a representative SIS epidemic model} assume that $\eta_{1,1} = 1$ and $\eta_{1,2} = \eta_{2,1} = \eta_{0}$ for some $\eta_{0}\in\{\frac{1}{2},1\}$, as in~\eqref{model:1}--\eqref{model:3}. Then we compute that
\begin{equation*}
h(\cdot,x,y,\nu) = b_{0}(\cdot,y,\nu)\frac{1}{x} + h_{1}(\cdot,y,\nu) + h_{2}(\cdot,y,\nu)x  - g_{1,1}g_{1,2}(y-x)^{\frac{3}{2}}\mathbbm{1}_{\{\frac{1}{2}\}}(\eta_{0}) + h_{3}(\cdot,y)x^{2}
\end{equation*}
for all $x,y > 0$ with $x\leq y$ and $\nu\in\mathcal{P}_{0}(\R)$, where $h_{1},h_{2}\colon\R_{+}\times\R_{+}\times\mathcal{P}_{0}(\R)\rightarrow\R$ and $h_{3}\colon\R_{+}\times\R_{+}\rightarrow\R$ are defined via
\begin{align*}
h_{1}(\cdot,y,\nu) &:= b_{1}(\cdot,y,\nu) - \frac{1}{2}(g_{1,2}^{2} + g_{2,1}^{2})y\mathbbm{1}_{\{\frac{1}{2}\}}(\eta_{0})\\
&\quad - \frac{1}{2}\big((g_{1,1} + g_{1,2}\mathbbm{1}_{\{1\}}(\eta_{0}))^{2} + g_{2,1}^{2}\mathbbm{1}_{\{1\}}(\eta_{0})\big)y^{2},\\
h_{2}(\cdot,y,\nu) &:= b_{2}(\cdot,y,\nu) + \frac{1}{2}(g_{1,2}^{2} + g_{2,1}^{2})\mathbbm{1}_{\{\frac{1}{2}\}}(\eta_{0})\\
&\quad + \big((g_{1,1} + g_{1,2}\mathbbm{1}_{\{1\}}(\eta_{0}))^{2} + g_{2,1}^{2}\mathbbm{1}_{\{1\}}(\eta_{0})\big)y,\\
h_{3}(\cdot,y) &:= b_{3}(\cdot,y) - \frac{1}{2}\big((g_{1,1} + g_{1,2}\mathbbm{1}_{\{1\}}(\eta_{0}))^{2} + g_{2,1}^{2}\mathbbm{1}_{\{1\}}(\eta_{0})\big).
\end{align*}
By definition, $h(\cdot,y,y,\nu) = b_{0}(\cdot,y,\nu)\frac{1}{y} + b_{1}(\cdot,y,\nu) + b_{2}(\cdot,y,\nu)y + b_{3}(\cdot,y)y^{2} = -(\mu + \gamma)$ for all $y > 0$ and $\nu\in\mathcal{P}_{0}(\R)$. In particular, we may set
\begin{equation*}
h(\cdot,0,y,\nu) := h_{1}(\cdot,y,\nu) - g_{1,1}g_{1,2}y^{\frac{3}{2}}\mathbbm{1}_{\{\frac{1}{2}\}}(\eta_{0}).
\end{equation*}
for all $y\geq 0$ and $\nu\in\mathcal{P}_{0}(\R)$ whenever $\beta_{0} = 0$, in which case $b_{0} = 0$.
\end{Example}

\subsection{Extinction}\label{se:3.2}

From the second pathwise asymptotic equality in Lemma~\ref{le:pathwise asymptotic identities} we infer a sufficient criterion for the extinction of the epidemic.

\begin{Proposition}\label{pr:extinction}
Let~\eqref{co:6} and~\eqref{eq:integrability condition} hold for $\rho > 0$ and $I$ be a positive solution to~\eqref{eq:SIS epidemic model} that admits a measurable locally integrable function $u\colon\R_{+}\rightarrow\R$ such that
\begin{equation*}
\limsup_{t\uparrow\infty} \frac{1}{t^{\rho}}\int_{0}^{t}u(s)\,\mathrm{d}s < 0\quad\text{and}\quad \limsup_{s\uparrow\infty} \frac{1}{s^{\rho - 1}}(h(s,I_{s},N(s),\mathcal{L}(I_{s})) - u(s))\leq 0.
\end{equation*}
Then $I$ becomes extinct with exponential rate $\rho$, that is,~\eqref{eq:extinction with an exponential rate} holds. If in addition $\rho = 1$ and there is $h_{\infty}\in\R$ such that
\begin{equation}\label{eq:a limit involving the transformed function}
\lim_{(s,x,\nu)\rightarrow (\infty,0,\delta_{0})} h(s,x,N(s),\nu) = h_{\infty},
\end{equation}
then $\lim_{t\uparrow\infty}\frac{1}{t}\log(I_{t}) = h_{\infty}$ a.s.~and $h_{\infty} < 0$.
\end{Proposition}

We demonstrate that Proposition~\ref{pr:extinction} covers the extinction results in the generalised SIS models by using the following conditions in the setting of Example~\ref{ex:a representative SIS epidemic model}:
\begin{enumerate}[label=(A.\arabic*), ref=A.\arabic*, leftmargin=\widthof{(A.2)} + \labelsep]
\item\label{a.1} We have $\beta_{0} = 0$, $\eta_{1,1} = 1$, $\eta_{1,2} = \eta_{2,1} = \eta_{0}$ for some $\eta_{0}\in\{\frac{1}{2},1\}$, as in \eqref{model:1}--\eqref{model:3}. Further, $g_{1,1}, g_{1,2}, g_{2,1}$ are bounded and there are $N_{\infty},\gamma_{\infty},\mu_{\infty}\geq 0$ such that
\begin{equation}\label{eq:specific limits}
\lim_{t\uparrow\infty} N(t) = N_{\infty}\quad\text{and}\quad \lim_{t\uparrow\infty} (\mu + \gamma)(t) = \mu_{\infty} + \gamma_{\infty}.
\end{equation}

\item\label{a.2} It holds that $c_{1,2} + c_{2,1} + c_{2,2}N > - \frac{1}{2}((g_{1,1} + g_{1,2}\mathbbm{1}_{\{1\}}(\eta_{0}))^{2} + g_{2,1}^{2}\mathbbm{1}_{\{1\}}(\eta_{0}))$, as in~\eqref{model:1}.
\end{enumerate}

Under the hypotheses stated below, we derive an upper bound for the function~\eqref{eq:transformed function} to obtain a condition that ensures exponential extinction in model~\eqref{model:1}.

\begin{Proposition}\label{pr:extinction 1}
In the setting of Example~\ref{ex:a representative SIS epidemic model} let~\eqref{a.1} and~\eqref{a.2} hold and suppose that if $\eta_{0} = \frac{1}{2}$, then $g_{1,1} = 0$ or $g_{1,2} = 0$. Further, let $u_{1}\in\R$ be such that
\begin{equation}\label{eq:extinction in the epidemic models 4}
\limsup_{t\uparrow\infty} (\beta + \beta_{1}^{+})(t) - \frac{1}{2}(g_{1,2}^{2} + g_{2,1}^{2})(t)\mathbbm{1}_{\{\frac{1}{2}\}}(\eta_{0}) \leq u_{1}
\end{equation}
and $c_{1,2,\infty},c_{2,1,\infty},c_{2,2,\infty},u_{2},u_{3}\in\R$ satisfy the following four conditions:
\begin{enumerate}[(i)]
\item $\lim_{t\uparrow\infty} c_{i,j}(t) = c_{i,j,\infty}$ for any $(i,j)\in\{(2,1),(1,2),(2,2)\}$.

\item $\lim_{t\uparrow\infty} (\beta_{1}^{-} - \beta)(t) + \frac{1}{2}(g_{1,2}(t)^{2} + g_{2,1}(t)^{2})\mathbbm{1}_{\{\frac{1}{2}\}}(\eta_{0}) = u_{2}$.

\item $\lim_{t\uparrow\infty} \big(g_{1,1}(t) + g_{1,2}(t)\mathbbm{1}_{\{1\}}(\eta_{0})\big)^{2} + g_{2,1}(t)^{2}\mathbbm{1}_{\{1\}}(\eta_{0}) = u_{3}$.

\item $u_{4} := c_{1,2,\infty} + c_{2,1,\infty} + c_{2,2,\infty}N_{\infty} + \frac{1}{2}u_{3}$ is positive.
\end{enumerate}
For a positive solution $I$ to~\eqref{eq:SIS epidemic model} set $u := -(\mu_{\infty} + \gamma_{\infty})$, if $u_{2} \geq (2c_{1,2,\infty} + c_{2,1,\infty})N_{\infty}$ $ +\, c_{2,2,\infty}N_{\infty}^{2}$, and
\begin{align*}
u &:= - (\mu_{\infty} + \gamma_{\infty} + u_{I}) + u_{1}N_{\infty} + \bigg(c_{1,2,\infty} - \frac{1}{2}u_{3}\bigg)N_{\infty}^{2}\\
&\quad + \frac{1}{4} \frac{\big((u_{2} + (c_{2,1,\infty} + u_{3})N_{\infty} + c_{2,2,\infty}N_{\infty}^{2})^{+}\big)^{2}}{u_{4}},
\end{align*}
otherwise, with the non-negative constant $u_{I} := \liminf_{t\uparrow\infty} \beta_{1}^{+}(t)N(t) - \beta_{1}(t)\E\big[I_{t}\big]$. Then $I$ becomes exponentially extinct if $u < 0$.
\end{Proposition}

\begin{Example}[Extinction]
The assumptions of Proposition~\ref{pr:extinction 1} hold for model \eqref{model:1} if there are $N_{\infty},\gamma_{\infty},\mu_{\infty}\geq 0$ and $\beta_{1,\infty}\in\R$ satisfying~\eqref{eq:specific limits} and $\lim_{t\uparrow\infty} \beta_{1}(t) = \beta_{1,\infty}$. Then
\begin{equation*}
u_{1} = \beta_{e} + \beta_{1,\infty}^{+},\quad u_{2} = \beta_{1,\infty}^{-} - \beta_{e},\quad u_{3} = \frac{\xi^{2}}{2\theta}
\end{equation*}
are feasible constants. Here,~\eqref{eq:a limit involving the transformed function} is valid for $h_{\infty} = \beta_{e}N_{\infty} - \frac{\xi^{2}}{4\theta}N_{\infty}^{2} - (\mu_{\infty} + \gamma_{\infty})$, and we obtain that
\begin{equation*}
u = \big(\beta_{e} + \beta_{1,\infty}^{+}\big)N_{\infty} - \frac{\xi^{2}}{4\theta}N_{\infty}^{2} -(\mu_{\infty} + \gamma_{\infty} + u_{I}) + \frac{\theta}{\xi^{2}}\bigg(\bigg(\beta_{1,\infty}^{-} - \beta_{e} + \frac{\xi^{2}}{2\theta}N_{\infty}\bigg)^{+}\bigg)^{2},
\end{equation*}
if $\beta_{e} > \beta_{1,\infty}^{-}$, and $u = - (\mu_{\infty} + \gamma_{\infty})$, otherwise. In particular, in the first case we have $u < 0$ if and only if either
\begin{equation*}
\frac{\xi^{2}}{2\theta}N_{\infty} \leq \beta_{e} - \beta_{1,\infty}^{-}\quad\text{and}\quad (\beta_{e} + \beta_{1,\infty}^{+})N_{\infty} - \frac{\xi^{2}}{4\theta}N_{\infty}^{2} < \mu_{\infty} + \gamma_{\infty} + u_{I}
\end{equation*}
or $\frac{\xi^{2}}{2\theta}N_{\infty} > \beta_{e} - \beta_{1,\infty}^{-}$ and $\frac{\theta}{\xi^{2}}(\beta_{1,\infty}^{-} - \beta_{e})^{2} + |\beta_{1,\infty}|N_{\infty} < \mu_{\infty} + \gamma_{\infty} + u_{I}$. These computations explicitly show that Theorem~3.1 in~\cite{WanCaiDinGui18} is a special case of Proposition~\ref{pr:extinction}.
\end{Example}

Next, we estimate function~\eqref{eq:transformed function} from above under a different set of assumptions, which yields a sufficient condition for exponential extinction in model~\eqref{model:2}.

\begin{Proposition}\label{pr:extinction 2}
In the setting of Example~\ref{ex:a representative SIS epidemic model} let~\eqref{a.1} and $\eta_{0} = \frac{1}{2}$ be valid. Further, suppose that either
\begin{equation*}
c_{1,2} = c_{2,1} = c_{2,2} = g_{1,1} = 0
\end{equation*}
or both~\eqref{a.2} and one of the following four scenarios hold:
\begin{enumerate}[(i)]
\item $\frac{1}{2}(g_{1,2}^{2} + g_{2,1}^{2}) \geq \beta + \beta_{1}^{+} + (2c_{1,2} + c_{2,1})N  + c_{2,2}N^{2}$ and $g_{1,1}g_{1,2}\geq 0$.

\item $-\frac{3}{2}g_{1,1}g_{1,2} \geq 4(c_{1,2} + c_{2,1} + c_{2,2}N + \frac{1}{2}g_{1,1}^{2})\sqrt{N}$.

\item $\frac{1}{2}(g_{1,2}^{2} + g_{2,1}^{2}) \leq \beta - \beta_{1}^{-} -(\frac{3}{2}g_{1,1}g_{1,2} + (c_{2,1} + g_{1,1}^{2})\sqrt{N} + c_{2,2}N^{\frac{3}{2}})\sqrt{N}$.

\item If $g_{1,1} > 0$, then
\begin{align*}
\frac{1}{2}(g_{1,2}^{2} + g_{2,1}^{2}) &\geq \beta + \beta_{1}^{+} + (2c_{1,2} + c_{2,1})N + c_{2,2}N^{2}\\
&\quad + \frac{1}{2}\bigg(\frac{3}{4}\bigg)^{2}\frac{g_{1,1}^{2}g_{1,2}^{2}}{\frac{1}{2}g_{1,1}^{2} + c_{1,2} + c_{2,1} + c_{2,2}N}.
\end{align*}
\end{enumerate}
 If there are $u_{1},u_{2},u_{3}\in\R$ satisfying~\eqref{eq:extinction in the epidemic models 4},
\begin{equation}\label{eq:extinction in the epidemic models 9}
\liminf_{t\uparrow\infty} g_{1,1}(t)g_{1,2}(t) \geq u_{2}\quad\text{and}\quad \limsup_{t\uparrow\infty} c_{1,2}(t) - \frac{1}{2}g_{1,1}(t)^{2} \leq u_{3},
\end{equation}
then any positive solution $I$ to~\eqref{eq:SIS epidemic model} becomes exponentially extinct if $\mu_{\infty} + \gamma_{\infty} > 0$ and
\begin{equation}\label{eq:extinction in the epidemic models 10}
(u_{1} + u_{3}N_{\infty} - u_{2}N^{\frac{1}{2}})N_{\infty} < \mu_{\infty} + \gamma_{\infty}.
\end{equation}
\end{Proposition}

\begin{Example}[Extinction]\label{ex:extinction 2}
In model~\eqref{model:2} each of the following four conditions implies the respective one in Proposition~\ref{pr:extinction 2}:
\begin{enumerate}[(i)]
\item $\frac{1}{2}(a_{2}^{2} + a_{3}^{2})\sigma_{2}^{2} \geq \beta + \beta_{1}^{+}$ and $a_{1}a_{2} \leq 0$.

\item $3\sgn(a_{1})a_{2}\sigma_{2} \geq 4a_{1}\sigma_{1}\sqrt{N}$.

\item $\frac{1}{2}(a_{2}^{2} + a_{3}^{2})\sigma_{2}^{2} \leq \beta - \beta_{1}^{-} + (\frac{3}{2}a_{2}\sigma_{2} - a_{1}\sigma_{1}\sqrt{N})a_{1}\sigma_{1}\sqrt{N}$.

\item $\frac{1}{2}(a_{2}^{2} + a_{3}^{2})\sigma_{2}^{2} \geq \beta + \beta_{1}^{+} + (\frac{3}{4})^{2}a_{2}^{2}\sigma_{2}^{2}$.
\end{enumerate}
Hence, let $a$ and $\sigma$ be bounded and $N_{\infty},\gamma_{\infty},\mu_{\infty}\geq 0$ satisfy~\eqref{eq:specific limits}. Then the requirements of Proposition~\ref{pr:extinction 2} are satisfied if either
\begin{equation*}
a_{1}\sigma_{1} = 0\quad\text{or both}\quad a_{1}\sigma_{1} > 0
\end{equation*}
and one of the above four conditions hold. If in addition there are $\beta_{\infty},a_{1,\infty},a_{3,\infty},\sigma_{1,\infty}$, $\sigma_{2,\infty}$ $\geq 0$ and $\beta_{1,\infty},a_{2,\infty}\in\R$ such that
\begin{equation*}
\lim_{t\uparrow\infty}\beta(t) = \beta_{\infty},\quad \lim_{t\uparrow\infty} \beta_{1}(t) = \beta_{1,\infty},\quad \lim_{t\uparrow\infty} a_{i}(t)  = a_{i,\infty},\quad \lim_{t\uparrow\infty}\sigma_{j}(t) = \sigma_{j,\infty}
\end{equation*}
for all $i\in\{1,2,3\}$ and $j\in\{1,2\}$, then for conditions~\eqref{eq:extinction in the epidemic models 4} and~\eqref{eq:extinction in the epidemic models 9} to be satisfied we can take $u_{1} = \beta_{\infty} + \beta_{1,\infty}^{+} - \frac{1}{2}(a_{2,\infty}^{2} + a_{3,\infty}^{2})\sigma_{2,\infty}^{2}$,
\begin{equation*}
u_{2} = - a_{1,\infty}a_{2,\infty}\sigma_{1,\infty}\sigma_{2,\infty}\quad\text{and}\quad u_{3} = - \frac{1}{2}a_{1,\infty}^{2}\sigma_{1,\infty}^{2}.
\end{equation*}
In this case,~\eqref{eq:a limit involving the transformed function} holds for $h_{\infty} = h_{0,\infty} - (\mu_{\infty} + \gamma_{\infty})$ with
\begin{equation*}
h_{0,\infty} : = \bigg(\beta_{\infty} - \frac{1}{2}(a_{2,\infty}^{2} + a_{3,\infty}^{2})\sigma_{2,\infty}^{2}\bigg)N_{\infty} - \frac{1}{2}a_{1,\infty}^{2}\sigma_{1,\infty}^{2}N_{\infty}^{2} + a_{1,\infty}a_{2,\infty}\sigma_{1,\infty}\sigma_{2,\infty}N_{\infty}^{\frac{3}{2}},
\end{equation*}
and~\eqref{eq:extinction in the epidemic models 10} is equivalent to the requirement that
\begin{equation}\label{eq:extinction in the epidemic models 11}
h_{0,\infty} + \beta_{1,\infty}^{+}N_{\infty} < \mu_{\infty} + \gamma_{\infty}.
\end{equation}
In particular, we recover Theorem~3.1 in~\cite{CaiCaiMao19-2}. Indeed, cases~(i)-(iv) appear there more restrictively, $\mu_{\infty}, \gamma_{\infty}, a_{1} > 0$, $\beta_{1} = \beta_{1,\infty} = 0$ and the stochastic reproduction number
\begin{equation*}
R_{0}^{S}
\end{equation*}
defined therein equals the ratio of $h_{0,\infty} + \beta_{1,\infty}^{+}N_{\infty}$ to $\mu_{\infty} + \gamma_{\infty}$ in our extended model~\eqref{model:2}. As $R_{0}^{S} < 1$ is equivalent to~\eqref{eq:extinction in the epidemic models 11}, Proposition~\ref{pr:extinction} entails Theorem 3.1 in~\cite{CaiCaiMao19-2}.
\end{Example}

We conclude with a final estimation of function~\eqref{eq:transformed function} in a suitable framework for model~\eqref{model:3}.

\begin{Lemma}\label{le:extinction 3}
In the setting of Example~\ref{ex:a representative SIS epidemic model} let~\eqref{a.1} be valid and~\eqref{a.2} be false and suppose that if $\eta_{0} = \frac{1}{2}$, then $g_{1,1} = 0$ or $g_{1,2} = 0$. If $u_{1},u_{2}\in\R$ satisfy~\eqref{eq:extinction in the epidemic models 4} and
\begin{equation}\label{eq:extinction in the epidemic models 12}
\limsup_{t\uparrow\infty} c_{1,2}(t) - \frac{1}{2}\big((g_{1,1} + g_{1,2}\mathbbm{1}_{\{1\}}(\eta_{0}))^{2} + g_{2,1}^{2}\mathbbm{1}_{\{1\}}(\eta_{0})\big) \leq u_{2},
\end{equation}
then any positive solution $I$ to~\eqref{eq:SIS epidemic model} becomes exponentially extinct whenever 
\begin{equation}\label{eq:extinction in the epidemic models 13}
\big(u_{1} + u_{2}N_{\infty}\big)^{+}N_{\infty} < \mu_{\infty} + \gamma_{\infty} + u_{I}
\end{equation}
with the non-negative constant $u_{I} := \liminf_{t\uparrow\infty} \beta_{1}^{+}(t)N(t) - \beta_{1}(t)\E\big[I_{t}\big]$.
\end{Lemma}

\begin{Example}[Extinction]
The hypotheses of Lemma~\ref{le:extinction 3} are valid in model \eqref{model:3} if $\sigma$ is bounded and there are $N_{\infty},\mu_{\infty},\gamma_{\infty}\geq 0$ satisfying~\eqref{eq:specific limits}. If in addition
\begin{equation*}
\lim_{t\uparrow\infty}\beta(t) = \beta_{\infty}\quad\text{and}\quad\lim_{t\uparrow\infty}\beta_{1}(t) = \beta_{1,\infty}
\end{equation*}
for some $\beta_{\infty}\geq 0$ and $\beta_{1,\infty}\in\R$, then conditions~\eqref{eq:extinction in the epidemic models 4} and~\eqref{eq:extinction in the epidemic models 12} hold for $u_{1} = \beta_{\infty} + \beta_{1,\infty}^{+}$ and $u_{2} = 0$. In this case,~\eqref{eq:a limit involving the transformed function} is satisfied for $h_{\infty} = \beta_{\infty}N_{\infty} - (\mu_{\infty} + \gamma_{\infty})$, and~\eqref{eq:extinction in the epidemic models 13} becomes
\begin{equation*}
\big(\beta_{\infty} + \beta_{1,\infty}^{+}\big)N_{\infty} < \mu_{\infty} + \gamma_{\infty} + u_{I}.
\end{equation*}
This extends the extinction assertion of Theorem~1 in~\cite{BerLan22}, since Theorem~5 there, from which the assertion is inferred, is a special case of Proposition~\ref{pr:extinction}.
\end{Example}

\subsection{Persistence}\label{se:3.3}

We use the pathwise asymptotic identities in Lemma~\ref{le:pathwise asymptotic identities} to derive sufficient conditions for the epidemic to persist above or around a given threshold.

\begin{Proposition}\label{pr:persistence}
Suppose that~\eqref{co:6} and~\eqref{eq:integrability condition} are valid for $\rho > 0$. Then for any positive solution $I$ to~\eqref{eq:SIS epidemic model} and $x_{0} \geq 0$, the following two assertions hold:
\begin{enumerate}[(i)]
\item If $x_{0} > 0$ and for each $\delta\in ]0,x_{0}[$ there is a measurable locally integrable function $w_{\delta}\colon\R_{+}\rightarrow\R$ such that $\limsup_{t\uparrow\infty}\frac{1}{t^{\rho}}\int_{0}^{t}w_{\delta}(s)\,\mathrm{d}s > 0$ and
\begin{equation*}
\liminf_{s\uparrow\infty}\inf_{x\in ]0,N(s)]:\,x < x_{0} - \delta}\frac{1}{s^{\rho -1}}\big(h\big(s,x,N(s),\mathcal{L}(I_{s})\big) - w_{\delta}(s)\big) \geq 0,
\end{equation*}
then $I$ persists above $x_{0}$, that is, $\limsup_{t\uparrow\infty} I_{t}\geq x_{0}$ a.s.

\item If for each $\delta > 0$ there is a measurable locally integrable function $w_{\delta}\colon\R_{+}\rightarrow\R$ such that $\liminf_{t\uparrow\infty}\frac{1}{t^{\rho}}\int_{0}^{t}w_{\delta}(s)\,\mathrm{d}s < 0$ and
\begin{equation*}
\limsup_{s\uparrow\infty}\sup_{x\in ]0,N(s)]:\,x > x_{0} + \delta}\frac{1}{s^{\rho -1}}\big(h\big(s,x,N(s),\mathcal{L}(I_{s})\big) - w_{\delta}(s)\big) \leq 0,
\end{equation*}
then $\liminf_{t\uparrow\infty} I_{t}\leq x_{0}$ a.s.
\end{enumerate}
\end{Proposition}

More specifically, as $\limsup_{t\uparrow\infty} N(t) = 0$ implies extinction of the disease, the sufficient conditions in Proposition~\ref{pr:persistence} can take the following form.

\begin{Corollary}\label{co:persistence}
Let~\eqref{co:6} and~\eqref{eq:integrability condition} hold for $\rho = 1$ and $N_{\infty} := \limsup_{t\uparrow\infty} N(t)$ be positive. Then each positive solution $I$ to~\eqref{eq:SIS epidemic model} satisfies the following two statements:
\begin{enumerate}[(i)]
\item If there is a continuous function $f_{\infty}\colon ]0,N_{\infty}]\cap\R_{+}\rightarrow\R$ that admits at least one zero and satisfies $\liminf_{x\downarrow 0} f_{\infty}(x) > 0$ and
\begin{equation*}
\liminf_{s\uparrow\infty} \inf_{x\in ]0,N(s)]:\,x < x_{0}}h\big(s,x,N(s),\mathcal{L}(I_{s})\big) - f_{\infty}(x) \geq 0
\end{equation*}
for the smallest zero $x_{0}$ of $f_{\infty}$, then $I$ persists above $x_{0}$.

\item If there is a continuous function $g_{\infty}\colon \R_{+}\rightarrow\R$ that has at least one zero in $[0,N_{\infty}[$ and satisfies $\limsup_{x\rightarrow N_{\infty}} g_{\infty}(x) < 0$ and
\begin{equation*}
\limsup_{s\uparrow\infty}\sup_{x\in ]0,N(s)]:\,x > y_{0}} h\big(s,x,N(s),\mathcal{L}(I_{s})\big) - g_{\infty}(x) \leq 0
\end{equation*}
for the largest zero $y_{0}$ of $g_{\infty}$ in $[0,N_{\infty}[$, then $\liminf_{t\uparrow\infty} I_{t} \leq y_{0}$ a.s.
\end{enumerate}
In particular, if the requirements of~(i) and~(ii) are met and $x_{0} = y_{0}$, then persistence~\eqref{eq:persistence around a point} around $x_{0}$ holds for $I$.
\end{Corollary}

To show persistence above or around a given level in the generalised SIS models of Example~\ref{ex:a representative SIS epidemic model}, we choose the functions $f_{\infty}$ and $g_{\infty}$ in Corollary~\ref{co:persistence} as follows.

\begin{Corollary}\label{co:persistence 2}
In the setting of Example~\ref{ex:a representative SIS epidemic model} let~\eqref{a.1} be valid with $N_{\infty} > 0$. Then for any positive solution $I$ to~\eqref{eq:a representative SIS epidemic model} the following two assertions hold:
\begin{enumerate}[(i)]
\item Assume that there are $a_{\infty}, b_{\infty}, c_{\infty}, d_{\infty}\in\R$ such that
\begin{align*}
\liminf_{s\uparrow \infty} h_{1}\big(s,N(s),\mathcal{L}(I_{s})\big) &\geq a_{\infty}, \quad \liminf_{s\uparrow \infty} h_{2}\big(s,N(s),\mathcal{L}(I_{s})\big) \geq b_{\infty},\\
\liminf_{s\uparrow \infty} h_{3}(s,N(s)) &\geq c_{\infty},\quad \limsup_{s\uparrow \infty} g_{1,1}(s)g_{1,2}(s)\mathbbm{1}_{\{\frac{1}{2}\}}(\eta_{0})\leq d_{\infty}.
\end{align*}
If $a_{\infty} > d_{\infty}N_{\infty}^{\frac{3}{2}}$, then $I$ persists above the smallest zero $x_{0}$ of $f_{\infty}\colon ]0,N_{\infty}]\rightarrow\R$ defined by
\begin{equation}\label{eq:auxiliary function for persistence}
f_{\infty}(x) := a_{\infty}+ b_{\infty}x + c_{\infty}x^{2} - d_{\infty}(N_{\infty}-x)^{\frac{3}{2}},
\end{equation}
and $x_{0} < N_{\infty}$ whenever $\gamma_{\infty} + \mu_{\infty} > 0$. Further, for $d_{\infty} = 0$ this zero is of the form $x_{0} = \frac{-b_{\infty}-\sqrt{b_{\infty}^{2}-4a_{\infty}c_{\infty}}}{2c_{\infty}}$, if $c_{\infty} \neq 0$, and $x_{0} = - \frac{a_{\infty}}{b_{\infty}}$, otherwise. 

\item Suppose that there are $\hat{a}_{\infty}, \hat{b}_{\infty},\hat{c}_{\infty},\hat{d}_{\infty}\in\R$ such that
\begin{align*}
\limsup_{s\uparrow \infty} h_{1}\big(s,N(s),\mathcal{L}(I_{s})\big) &\leq \hat{a}_{\infty}, \quad \limsup_{s\uparrow \infty} h_{2}\big(s,N(s),\mathcal{L}(I_{s})\big) \leq \hat{b}_{\infty},\\
\limsup_{s\uparrow \infty} h_{3}(s,N(s)) &\leq \hat{c}_{\infty},\quad \liminf_{s\uparrow \infty} g_{1,1}(s)g_{1,2}(s)\mathbbm{1}_{\{\frac{1}{2}\}}(\eta_{0})\geq \hat{d}_{\infty}.
\end{align*}
If $\hat{a}_{\infty} \geq \hat{d}_{\infty}N_{\infty}^{\frac{3}{2}}$ and $\hat{a}_{\infty} + \hat{b}_{\infty}N_{\infty} + \hat{c}_{\infty}N_{\infty}^{2} < 0$, then $\liminf_{t\uparrow\infty} I_{t} \leq y_{0}$ a.s.~for the largest zero $y_{0}$ of $g_{\infty}\colon [0,N_{\infty}[\rightarrow\R$ given by
\begin{equation}\label{eq:auxiliary function 2 for persistence}
g_{\infty}(x) := \hat{a}_{\infty} + \hat{b}_{\infty}x + \hat{c}_{\infty}x^{2} - \hat{d}_{\infty}(N_{\infty}-x)^{\frac{3}{2}}.
\end{equation}
Further, for $\hat{d}_{\infty} = 0$ this zero is unique and of the form $y_{0} = \frac{-\hat{b}_{\infty} - \sqrt{\hat{b}_{\infty}^{2}-4\hat{a}_{\infty}\hat{c}_{\infty}}}{2\hat{c}_{\infty}}$, if $\hat{c}_{\infty} \neq 0$, and $y_{0} = - \frac{\hat{a}_{\infty}}{\hat{b}_{\infty}}$, otherwise.
\end{enumerate}
\end{Corollary}

\begin{Remark}
Let $c_{2,2,\infty},u_{1}, v_{1}, v_{2}, v_{3}\in\R$ satisfy $\lim_{t\uparrow \infty} c_{2,2}(t)= c_{2,2,\infty}$, inequality~\eqref{eq:extinction in the epidemic models 4} and the following three conditions:
\begin{enumerate}[(i)]
\item $\liminf_{t\uparrow \infty} (\beta -\beta_{1}^{-})(t) - \frac{1}{2}(g_{1,2}(t)^{2} + g_{2,1}(t)^{2})\mathbbm{1}_{\{\frac{1}{2}\}}(\eta_{0}) \geq v_{1}$.

\item $\lim_{t\uparrow\infty} c_{1,2}(t) - \frac{1}{2}((g_{1,1}(t)+ g_{1,2}(t)\mathbbm{1}_{\{1\}}(\eta_{0}))^{2} + g_{2,1}(t)^{2}\mathbbm{1}_{\{1\}}(\eta_{0})) = v_{2}$.

\item $\lim_{t\uparrow\infty} c_{2,1}(t) + (g_{1,1}(t)+ g_{1,2}(t)\mathbbm{1}_{\{1\}}(\eta_{0}))^{2} + g_{2,1}(t)^{2}\mathbbm{1}_{\{1\}}(\eta_{0}) = v_{3}$.
\end{enumerate}
Then in Corollary~\ref{co:persistence 2} we may take the constants
\begin{equation}\label{eq:choice of coefficients}
\begin{split}
a_{\infty} & = - (\mu_{\infty} + \gamma_{\infty}) + v_{I} + v_{1}N_{\infty} + v_{2}N_{\infty}^{2},\quad \hat{a}_{\infty} = a_{\infty} - u_{I} - v_{I} + (u_{1} - v_{1})N_{\infty},\\
b_{\infty} & = -u_{1} + \frac{u_{I}}{N_{\infty}} + v_{3}N_{\infty} + c_{2,2,\infty}N_{\infty}^{2},\quad \hat{b}_{\infty} = b_{\infty} + u_{1} - v_{1} - \frac{u_{I} + v_{I}}{N_{\infty}},\\\
c_{\infty} & = \hat{c}_{\infty} = -v_{2} - v_{3} - c_{2,2,\infty}N_{\infty}
\end{split}
\end{equation}
for $u_{I} = v_{I} = 0$ or more generally any $u_{I},v_{I}\geq 0$ such that
\begin{equation}\label{eq:condition on the coefficients}
u_{I}\leq \liminf_{t\uparrow\infty} \beta_{1}^{+}(t)N(t) - \beta_{1}(t)\E\big[I_{t}\big]\quad\text{and}\quad  v_{I}\leq \liminf_{t\uparrow\infty} \beta_{1}(t)\E\big[I_{t}\big] + \beta_{1}^{-}(t)N(t).
\end{equation}
\end{Remark}

In the setting of Example~\ref{ex:a representative SIS epidemic model} we observe that if $\lim_{t\uparrow\infty}N(t) = N_{\infty}$ and $\lim_{t\uparrow\infty}\beta_{1}(t)$ $= \beta_{1,\infty}$ for some $N_{\infty} > 0$ and $\beta_{1,\infty}\in\R$, then
\begin{equation}\label{eq:limits for a solution}
m_{I} := \limsup_{t\uparrow\infty} \beta_{1}(t)\frac{\E\big[I_{t}\big]}{N(t)}\quad\text{and}\quad n_{I} := \liminf_{t\uparrow\infty} \beta_{1}(t)\frac{\E\big[I_{t}\big]}{N(t)}
\end{equation}
satisfy $-\beta_{1,\infty}^{-} \leq n_{I} \leq m_{I} \leq \beta_{1,\infty}^{+}$ for any positive solution $I$ to~\eqref{eq:a representative SIS epidemic model}. In this case, any constants $u_{I},v_{I}\geq 0$ satisfying~\eqref{eq:condition on the coefficients} are of the form
\begin{equation}\label{eq:choice of coefficients 2}
u_{I} = N_{\infty}(\beta_{1,\infty}^{+} - x)\quad\text{and}\quad v_{I} = N_{\infty}(\beta_{1,\infty}^{-} + y)
\end{equation}
for some $x\in [m_{I},\beta_{1,\infty}^{+}]$ and $y\in [-\beta_{1,\infty}^{-},n_{I}]$. Now we can readily show that Corollary~\ref{co:persistence 2} generalises the persistence results of the SIS models in~\cite{WanCaiDinGui18, CaiCaiMao19, GraGreHuMaoPan11, CaiCaiMao19-2, BerLan22}.

\begin{Example}[Persistence]\label{ex:persistence 1}
In model~\eqref{model:1} of Example~\ref{ex:a representative SIS epidemic model} let $N_{\infty} > 0$, $\gamma_{\infty},\mu_{\infty}\geq 0$ and $\beta_{1,\infty}\in\R$ satisfy~\eqref{eq:specific limits} and $\lim_{t\uparrow\infty} \beta_{1}(t) = \beta_{1,\infty}$. Then the constants
\begin{equation*}
u_{1} = \beta_{e} + \beta_{1,\infty}^{+},\quad v_{1} = \beta_{e}  - \beta_{1,\infty}^{-},\quad v_{2} = -\frac{\xi^{2}}{4\theta},\quad v_{3} =\frac{\xi^{2}}{2\theta},\quad d_{\infty} = 0
\end{equation*}
are feasible in~\eqref{eq:choice of coefficients} and Corollary~\ref{co:persistence 2} when~\eqref{eq:choice of coefficients 2} is valid for some $x\in [m_{I},\beta_{1,\infty}^{+}]$ and $y\in [-\beta_{1,\infty}^{-},n_{I}]$. So, we obtain three assertions extending~\cite[Theorem~3.2]{WanCaiDinGui18} as follows:
\begin{enumerate}[(i)]
\item If $(\beta_{e}  + y)N_{\infty} - \frac{\xi^{2}}{4\theta}N^{2}_{\infty} > \mu_{\infty} + \gamma_{\infty}$, then $I$ persists above the level
\begin{equation*}
x_{0} := N_{\infty} - (\beta_{e} + x)\frac{2\theta}{\xi^{2}} + \sqrt{(\beta_{e} + x)^{2}\frac{4\theta^{2}}{\xi^{4}} - (\mu_{\infty} + \gamma_{\infty} + (x - y)N_{\infty})\frac{4\theta}{\xi^{2}}},
\end{equation*}
which belongs to $]0,N_{\infty}]$, and $x_{0} < N_{\infty}$ follows from $\mu_{\infty} + \gamma_{\infty} > 0$.

\item If $(\beta_{e}  + x)N_{\infty} - \frac{\xi^{2}}{4\theta}N^{2}_{\infty} \geq \mu_{\infty} + \gamma_{\infty} > (x - y)N_{\infty}$, then $\liminf_{t\uparrow\infty} I_{t} \leq y_{0}$ a.s.~for
\begin{equation*}
y_{0} := N_{\infty} - (\beta_{e} + y)\frac{2\theta}{\xi^{2}} + \sqrt{(\beta_{e} + y)^{2}\frac{4\theta^{2}}{\xi^{4}} - (\mu_{\infty} + \gamma_{\infty} + (y - x)N_{\infty})\frac{4\theta}{\xi^{2}}},
\end{equation*}
lying in $[0,N_{\infty}[$, and $y_{0} > 0$ whenever the first required inequality is strict.

\item If $\lim_{t\uparrow\infty}\beta_{1}(t)\E[I_{t}]$ exists, $x = y = m_{I}$ and $(\beta_{e}  + m_{I})N_{\infty} - \frac{\xi^{2}}{4\theta}N^{2}_{\infty} > \mu_{\infty} + \gamma_{\infty} > 0$, then $0 < x_{0} = y_{0} < N_{\infty}$ and $I$ persists around $x_{0}$.
\end{enumerate}
\end{Example}

\begin{Example}[Persistence]\label{ex:persistence 2}
In model~\eqref{model:2} of Example~\ref{ex:a representative SIS epidemic model} assume that $N_{\infty} > 0$, $\beta_{\infty},\gamma_{\infty},\mu_{\infty},a_{1,\infty},a_{3,\infty},\sigma_{1,\infty}, \sigma_{2,\infty} \geq 0$ and $\beta_{1,\infty},a_{2,\infty}\in\R$ are such that~\eqref{eq:specific limits} and
\begin{equation*}
\lim_{t\uparrow\infty}\beta(t) = \beta_{\infty},\quad \lim_{t\uparrow\infty} \beta_{1}(t) = \beta_{1,\infty},\quad \lim_{t\uparrow\infty} a_{i}(t)  = a_{i,\infty},\quad \lim_{t\uparrow\infty}\sigma_{j}(t) = \sigma_{j,\infty}
\end{equation*}
hold for all $i\in\{1,2,3\}$ and $j\in\{1,2\}$. Then in~\eqref{eq:choice of coefficients} and Corollary~\ref{co:persistence 2} we may choose the constants
\begin{align*}
u_{1} & =\beta_{\infty} + \beta_{1,\infty}^{+} - \frac{1}{2}(a_{2,\infty}^{2} + a_{3,\infty}^{2})\sigma_{2,\infty}^{2},\quad v_{1} = \beta_{\infty} -\beta_{1,\infty}^{-} - \frac{1}{2}(a_{2,\infty}^{2}+a_{3,\infty}^{2})\sigma_{2,\infty}^{2},\\
v_{2} & = -\frac{1}{2}a_{1,\infty}^{2}\sigma_{1,\infty}^{2},\quad v_{3}=a_{1,\infty}^{2}\sigma_{1,\infty}^{2},\quad d_{\infty} = -a_{1,\infty}a_{2,\infty}\sigma_{1,\infty}\sigma_{2,\infty}
\end{align*}
when~\eqref{eq:choice of coefficients 2} holds for $x\in [m_{I},\beta_{1,\infty}^{+}]$ and $y\in [-\beta_{1,\infty}^{-},n_{I}]$. Hence, the following assertions generalise~\cite[Theorem 4.1]{CaiCaiMao19}, \cite[Theorem 4.1]{CaiCaiMao19-2} and \cite[Theorem 5.1]{GraGreHuMaoPan11}:
\begin{enumerate}[(i)]
\item The solution $I$ persists above the smallest zero $x_{0}\in ]0,N_{\infty}]$ of function~\eqref{eq:auxiliary function for persistence} if
\begin{equation}\label{eq:sufficient condition for persistence in model 2}
\begin{split}
&\bigg(\beta_{\infty} + y - \frac{1}{2}(a_{2,\infty}^{2} + a_{3,\infty}^{2})\sigma_{2,\infty}^{2}\bigg)N_{\infty}\\
&\quad - \frac{1}{2}a_{1,\infty}^{2}\sigma_{1,\infty}^{2}N_{\infty}^{2} + a_{1,\infty}a_{2,\infty}\sigma_{1,\infty}\sigma_{2,\infty}N_{\infty}^{\frac{3}{2}} > \mu_{\infty} + \gamma_{\infty},
\end{split}
\end{equation}
and $x_{0} < N_{\infty}$ whenever $\gamma_{\infty} + \mu_{\infty} > 0$. If in addition to~\eqref{eq:sufficient condition for persistence in model 2} we have $a_{2,\infty}=0$, $\sigma_{1,\infty}>0$ and $a_{1,\infty} = a_{3,\infty} = 1$, then
\begin{equation*}
x_{0} = N_{\infty} - \frac{\beta_{\infty} + x - \frac{1}{2}\sigma_{2,\infty}^{2}}{\sigma_{1,\infty}^{2}} + \sqrt{\frac{(\beta_{\infty} + x - \frac{1}{2}\sigma_{2,\infty}^{2})^{2}}{\sigma_{1,\infty}^{4}}  - 2\frac{\mu_{\infty} + \gamma_{\infty} + (x - y)N_{\infty}}{\sigma_{1,\infty}^{2}}}.
\end{equation*}

\item Let~\eqref{eq:sufficient condition for persistence in model 2} hold for $x$ in lieu of $y$ when the inequality is not strict, and $\mu_{\infty} + \gamma_{\infty}$ $> (x - y)N_{\infty}$, then the largest zero $y_{0}\in [0,N_{\infty}[$ of function~\eqref{eq:auxiliary function 2 for persistence} satisfies
\begin{equation*}
\liminf_{t\uparrow\infty} I_{t} \leq y_{0}\quad\text{a.s.,}
\end{equation*}
and $y_{0} > 0$ whenever the first required inequality is strict. Further, if $a_{2,\infty}=0$, $\sigma_{1,\infty}>0$ and $a_{1,\infty} = a_{3,\infty} = 1$, then
\begin{equation*}
y_{0} = N_{\infty} - \frac{\beta_{\infty} + y - \frac{1}{2}\sigma_{2,\infty}^{2}}{\sigma_{1,\infty}^{2}} + \sqrt{\frac{(\beta_{\infty} + y - \frac{1}{2}\sigma_{2,\infty}^{2})^{2}}{\sigma_{1,\infty}^{4}}  - 2\frac{\mu_{\infty} + \gamma_{\infty} + (y - x)N_{\infty}}{\sigma_{1,\infty}^{2}}}.
\end{equation*}

\item If $m_{I} = n_{I} = x = y$, inequality~\eqref{eq:sufficient condition for persistence in model 2} holds for $y = m_{I}$ and $\mu_{\infty} + \gamma_{\infty} > 0$, then $0 < x_{0} = y_{0} < N_{\infty}$ and $I$ persists around $x_{0}$.
\end{enumerate}
\end{Example}

\begin{Example}[Persistence]\label{ex:persistence 3}
In model~\eqref{model:3} of Example~\ref{ex:a representative SIS epidemic model} suppose that $N_{\infty} > 0$ and $\beta_{\infty},\gamma_{\infty},\mu_{\infty},\sigma_{\infty}\geq 0$ satisfy~\eqref{eq:specific limits}, $\lim_{t\uparrow\infty}\beta(t) = \beta_{\infty}$ and $\lim_{t\uparrow\infty}\sigma(t) = \sigma_{\infty}$. Then
\begin{equation*}
u_{1} = \beta_{\infty} + \beta_{1,\infty}^{+},\quad v_{1} = \beta_{\infty} - \beta_{1,\infty}^{-},\quad v_{2}=0,\quad v_{3}= -\frac{1}{2}\sigma_{\infty}^{2},\quad d_{\infty} = 0
\end{equation*}
are feasible constants in~\eqref{eq:choice of coefficients} and Corollary~\ref{co:persistence 2} when~\eqref{eq:choice of coefficients 2} is valid for $x\in [m_{I},\beta_{1,\infty}^{+}]$ and $y\in [-\beta_{1,\infty}^{-},n_{I}]$. Consequently, the results in \cite[Section 3.1]{BerLan22} extend as follows:
\begin{enumerate}[(i)]
\item If $(\beta_{\infty} + y)N_{\infty} > \mu_{\infty} + \gamma_{\infty}$, then $I$ persists above $x_{0}\in ]0,N_{\infty}]$ defined by
\begin{equation*}
x_{0} := \frac{1}{2}N_{\infty} + \frac{\beta_{\infty} + x}{\sigma_{\infty}^{2}} - \sqrt{\bigg(\frac{1}{2}N_{\infty} - \frac{\beta_{\infty} + x}{\sigma_{\infty}^{2}}\bigg)^{2} + 2\frac{\mu_{\infty} + \gamma_{\infty} + (x - y)N_{\infty}}{\sigma_{\infty}^{2}}},
\end{equation*}
if $\sigma_{\infty} > 0$, and $x_{0} := \frac{- (\mu_{\infty} + \gamma_{\infty}) + (\beta_{\infty} + y)N_{\infty}}{\beta_{\infty} + x}$, otherwise. Moreover, if $\mu_{\infty} + \gamma_{\infty} > 0$, then $x_{0} < N_{\infty}$.

\item If $(\beta_{\infty} + x)N_{\infty} \geq \mu_{\infty} + \gamma_{\infty} > (x - y)N_{\infty}$, then $\liminf_{t\uparrow\infty}I_{t} \leq y_{0}$ a.s.~for $y_{0}\in [0,N_{\infty}[$ given by
\begin{equation*}
y_{0} := \frac{1}{2}N_{\infty} + \frac{\beta_{\infty} + y}{\sigma_{\infty}^{2}} - \sqrt{\bigg(\frac{1}{2}N_{\infty} - \frac{\beta_{\infty} + y}{\sigma_{\infty}^{2}}\bigg)^{2} + 2\frac{\mu_{\infty} + \gamma_{\infty} + (y - x)N_{\infty}}{\sigma_{\infty}^{2}}},
\end{equation*}
if $\sigma_{\infty} > 0$, and $y_{0} := \frac{- (\mu_{\infty} + \gamma_{\infty}) + (\beta_{\infty} + x)N_{\infty}}{\beta_{\infty} + y}$, otherwise. If in fact the first required inequality is strict, then $y_{0} > 0$.

\item If $m_{I} = n_{I} = x = y$ and $(\beta_{\infty} + m_{I})N_{\infty} > \mu_{\infty} + \gamma_{\infty} > 0$, then $0 < x_{0} = y_{0} < N_{\infty}$ and $I$ persists around $x_{0}$.
\end{enumerate}
\end{Example}

\begin{Remark}
The condition $a_{\infty} > d_{\infty}N_{\infty} ^{\frac{3}{2}}$ in Examples~\ref{ex:persistence 1}--\ref{ex:persistence 3} corresponds to the assumption that the stochastic reproduction number defined in~\cite{WanCaiDinGui18, CaiCaiMao19, GraGreHuMaoPan11, CaiCaiMao19-2, BerLan22} exceeds one.
\end{Remark}

\section{Numerical approximation and implementation}\label{se:4}

\subsection{A distribution-dependent Euler--Maruyama scheme}\label{se:4.1}

In what follows, for a finite time horizon $T > 0$ and $n\in\N$ let $\mathbb{T}_{n}$ be a partition of $[0,T]$ of the form $\mathbb{T}_{n} = (t_{j,n})_{j\in\{0,\dots,k_{n}\}}$ for some $k_{n}\in\N$ and $t_{0,n},\dots,t_{k_{n},n}\in [0,T]$ satisfying $0 = t_{0,n} < \cdots < t_{k_{n},n} = T$. Given $p\geq 2$, we take an $\mathcal{F}_{t_{j,n}}$-measurable map
\begin{equation*}
L_{j,n}\colon\Omega\rightarrow\mathcal{P}_{p}(\R)
\end{equation*}
for any $j\in\{0,\dots,k_{n}-1\}$ and define the \emph{continuous-time Euler--Maruyama scheme} along $\mathbb{T}_{n}$ with respect to~\eqref{eq:SIS epidemic model} and $(L_{j,n})_{j\in\{0,\dots,k_{n}-1\}}$ to be the $\mathbb{F}$-adapted right-continuous process $I^{(n)}\colon [0,T]\times\Omega\rightarrow\R$ first recursively given by $I_{0}^{(n)} := \xi_{0}$ and
\begin{equation}\label{eq:Euler--Maruyama scheme}
\begin{split}
I_{t_{j+1},n}^{(n)} &:= I_{t_{j},n}^{(n)} + \sum_{i=0}^{k}\overline{b}_{i}(t_{j,n},N(t_{j,n}),L_{j,n})\big((I_{t_{j,n}}^{(n)})^{+}\wedge N(t_{j,n})\big)^{i}(t_{j+1,n} - t_{j,n}) \\
&\quad + \sum_{i=1}^{d}\overline{f}_{i}\big(t_{j,n},I_{t_{j,n}}^{(n)},N(t_{j,n}) - I_{t_{j,n}}^{(n)}\big)(W_{t_{j+1},n}^{(i)} - W_{t_{j,n}}^{(i)})
\end{split}
\end{equation}
for each $j\in\{0,\dots,k_{n}-1\}$ and then constantly extended by setting $I_{t}^{(n)} := I_{t_{j,n}}^{(n)}$ for all $t\in ]t_{j,n},t_{j+1,n}[$. This yields the representation $I^{(n)} = I_{T}^{(n)}\mathbbm{1}_{\{T\}} + \sum_{j=0}^{k_{n}-1} I_{t_{j,n}}^{(n)}\mathbbm{1}_{[t_{j,n},t_{j+1,n}[}$.

By the \emph{interpolated Euler--Maruyama scheme} we shall mean the unique $\mathbb{F}$-adapted continuous process $\hat{I}^{(n)}\colon [0,T]\times\Omega\rightarrow\R$ satisfying
\begin{equation}\label{eq:interpolated Euler--Maruyama scheme}
\begin{split}
\hat{I}_{t}^{(n)} &= I_{t_{j,n}}^{(n)} + \sum_{i=0}^{k}\overline{b}_{i}(t_{j,n},N(t_{j,n}),L_{j,n})\big((I_{t_{j,n}}^{(n)})^{+}\wedge N(t_{j,n})\big)^{i}(t - t_{j,n})\\
&\quad + \sum_{i=1}^{d}\overline{f}_{i}\big(t_{j,n},I_{t_{j,n}}^{(n)},N(t_{j,n}) - I_{t_{j,n}}^{(n)}\big)(W_{t}^{(i)} - W_{t_{j,n}}^{(i)})
\end{split}
\end{equation}
for any $j\in\{0,\dots,k_{n}-1\}$ and $t\in [t_{j,n},t_{j+1,n}]$. Then $\hat{I}^{(n)}$ is a semimartingale and $\hat{I}_{t_{j,n}}^{(n)} = I_{t_{j,n}}^{(n)}$ for all $j\in\{0,\dots,k_{n}\}$, by definition.

To deduce a quantitative estimate of the absolute $p$th moment of $\hat{I}^{(n)}$, we impose a \emph{boundedness and Lipschitz condition} on $b_{0},\dots,b_{k}$ that is stronger than~\eqref{co:1}.
\begin{enumerate}[label=(C.\arabic*), ref=C.\arabic*, leftmargin=\widthof{(C.7)} + \labelsep]
\setcounter{enumi}{6}
\item\label{co:7} For every $i\in\{0,\dots,k\}$ there are some measurable locally bounded functions $\hat{b}_{i},\hat{\lambda}_{i}\colon\R_{+}\times\R_{+}\rightarrow\R_{+}$ such that
\begin{equation*}
|b_{i}(s,y,\nu)| \leq \hat{b}_{i}(s,y)\quad\text{and}\quad |b_{i}(s,y,\nu) - b_{i}(s,y,\tilde{\nu})| \leq \hat{\lambda}_{i}(s,y)\vartheta(\nu,\tilde{\nu})
\end{equation*}
for any $s,y\geq 0$ and $\nu,\tilde{\nu}\in\mathcal{P}_{0}(\R)$.
\end{enumerate}

This requirement yields an \emph{affine growth condition} for the drift coefficient of~\eqref{eq:SIS epidemic model}. More specifically,
\begin{equation}\label{eq:affine growth condition on the drift coefficient}
\bigg|\sum_{i=0}^{k}b_{i}(\cdot,N,\nu)x^{i}\bigg| \leq \hat{b}_{0}(\cdot,N) + \hat{b}_{k+3}(\cdot,N)x
\end{equation}
for any $x\in [0,N]$ and $\nu\in\mathcal{P}_{0}(\R)$ with the $\R_{+}$-valued locally bounded function $\hat{b}_{k+3}$ on $\R_{+}\times\R_{+}$ defined via
\begin{equation*}
\hat{b}_{k+3}(s,y) : =\sup_{x\in [0,y]}\sup_{\nu\in\mathcal{P}_{0}(\R)}\bigg|\sum_{i=1}^{k}b_{i}(s,y,\nu)x^{i-1}\bigg|,
\end{equation*}
which satisfies $\hat{b}_{k+3}(\cdot,y) \leq \sum_{i=1}^{k} \hat{b}_{i}(\cdot,y)y^{i-1}$ for all $y\geq 0$. We recall that $c_{p} = \frac{p-1}{2}$ and state an \emph{explicit $L^{p}$-growth estimate} for $\hat{I}^{(n)}$, which applies to $I^{(n)}$ for any $n\in\N$.

\begin{Proposition}\label{pr:EM-estimate}
Assume that~\eqref{co:3},~\eqref{co:6} and~\eqref{co:7} are satisfied. Then
\begin{equation*}
\E\big[|\hat{I}_{t}^{(n)}|^{p}\big] \leq e^{k_{p,j,n}t}\E\big[\xi_{0}^{p}\big] + e^{l_{p,j,n}t}t\max_{i=0,\dots,j}\hat{b}_{0}(t_{i,n},N(t_{i,n}))
\end{equation*}
for all $n\in\N$, $j\in\{0,\dots,k_{n}-1\}$ and $t\in [t_{j,n},t_{j+1,n}]$ with the two non-negative constants
\begin{align*}
k_{p,j,n} &:= \max_{i=0,\dots,j} \big((p-1)\hat{b}_{0} + p(\hat{b}_{k+3} + c_{p}l^{2})\big)(t_{i,n},N(t_{i,n})),
\\
l_{p,j,n} &:= \max_{i=0,\dots,j} \big(2(p-1)\big(\hat{b}_{0} + c_{p}l^{2}\big) + (2p-1)\hat{b}_{k+3}\big)(t_{i,n},N(t_{i,n})).
\end{align*}
\end{Proposition}

\begin{Remark}
As the function $l$ in~\eqref{eq:linear growth condition on the diffusion coefficient} is locally bounded, we obtain the quantitative estimate
\begin{equation*}
\sup_{n\in\N}\max_{j= 0,\dots,k_{n}-1}\sup_{t\in [t_{j,n},t_{j+1,n}]}
\E\big[|\hat{I}_{t}^{(n)}|^{p}\big] \leq  e^{l_{p}T}\big(\E\big[\xi_{0}^{p}\big] + T\hat{c}_{0}\big)
\end{equation*}
with the constants $\hat{c}_{0}:= \sup_{n\in\N}\max_{i=0,\dots,k_{n}-1}\hat{b}_{0}(t_{i,n},N(t_{i,n}))$ and $l_{p} := \sup_{n\in\N}l_{p,k_{n} -1,n}$.
\end{Remark}

Let us next investigate the increments of the interpolated Euler--Maruyama scheme $\hat{I}^{(n)}$ along $\mathbb{T}_{n}$ for each $n\in\N$.

\begin{Lemma}\label{le:EM-estimate}
Let~\eqref{co:3},~\eqref{co:6} and~\eqref{co:7} hold. Then
\begin{equation*}
\E\big[|\hat{I}_{t}^{(n)} - I_{t_{j,n}}^{(n)}\big|^{p}\big]^{\frac{1}{p}} \leq m_{p,j,n}(t-t_{j,n})^{\frac{1}{2}}
\end{equation*}
for any $n\in\N$, $j\in\{0,\dots,k_{n}-1\}$ and $t\in [t_{j,n},t_{j+1,n}]$ with the non-negative constant
\begin{equation}\label{eq:error coefficient 1}
\begin{split}
m_{p,j,n} &:= T^{\frac{1}{2}}\big(\hat{b}_{0}(t_{j,n},N(t_{j,n})) + \hat{b}_{k+3}(t_{j,n},N(t_{j,n})) N(t_{j,n})\big)\\
&\quad + \E\big[\big(Z_{1}^{2} + \cdots + Z_{d}^{2}\big)^{\frac{p}{2}}\big]^{\frac{1}{p}}l(t_{j,n},N(t_{j,n}))N(t_{j,n}),
\end{split}
\end{equation}
where $Z_{1},\dots,Z_{d}$ are independent $\mathcal{N}(0,1)$-distributed random variables.
\end{Lemma}

\begin{Remark}
This auxiliary result immediately implies that
\begin{equation*}
\sup_{n\in\N}\max_{j=0\dots,k_{n}-1}\sup_{t\in [t_{j,n},t_{j+1,n}]} \E\big[|\hat{I}_{t}^{(n)} - I_{t_{j,n}}^{(n)}\big|^{p}\big]^{\frac{1}{p}} \leq m_{p}|\mathbb{T}_{n}|^{\frac{1}{2}}
\end{equation*}
for the real constant $m_{p} := \sup_{n\in\N}\max_{j=0,\dots,k_{n}-1} m_{p,j,n}$.
\end{Remark}

To derive an intermediate error estimate for $\hat{I}^{(n)}$, we first require a \emph{Hölder continuity condition} on $b_{0},\dots,b_{k}$:
\begin{enumerate}[label=(C.\arabic*), ref=C.\arabic*, leftmargin=\widthof{(C.8)} + \labelsep]
\setcounter{enumi}{7}
\item\label{co:8} For each $i\in\{0,\dots,k\}$ there is a constant $\overline{b}_{i}\geq 0$ such that $|b_{i}(s,y,\nu) - b_{i}(\tilde{s},\tilde{y},\nu)|$ $\leq \overline{b}_{i}(|s-\tilde{s}|^{\frac{1}{2}} + |y-\tilde{y}|)$ for all $s,\tilde{s}\geq 0$, $y,\tilde{y}\geq 0$ and $\nu\in\mathcal{P}_{0}(\R)$.
\end{enumerate}

Under~\eqref{co:7} and~\eqref{co:8}, a \emph{continuity condition} for the drift coefficient of~\eqref{eq:SIS epidemic model} holds, which implies~\eqref{eq:partial Lipschitz condition on the drift coefficient}. Namely,
\begin{equation}\label{eq:continuity condition on the drift coefficient}
\begin{split}
\bigg|\sum_{i=0}^{k}b_{i}(s,N(s),\nu)&x^{i} - b_{i}(\tilde{s},N(\tilde{s}),\tilde{\nu})\tilde{x}^{i}\bigg| \leq  \overline{b}_{k+1}(N(\tilde{s}))|s-\tilde{s}|^{\frac{1}{2}}\\
&\quad + \big(\overline{b}_{k+1}(N(\tilde{s})) + \overline{b}_{k+2}(s,N(s),N(\tilde{s}))\big)|N(s) - N(\tilde{s})|\\
&\quad + \hat{b}_{k+4}(s,N(s))|x - \tilde{x}| + \hat{\lambda}_{k+1}(s,N(s))\vartheta(\nu,\tilde{\nu})
\end{split}
\end{equation}
for all $s,\tilde{s}\geq 0$, $x\in [0,N(s)]$, $\tilde{x}\in [0,N(\tilde{s})]$ and $\nu,\tilde{\nu}\in\mathcal{P}_{0}(\R)$, where $\overline{b}_{k+1}\colon\R_{+}\rightarrow\R_{+}$, $\overline{b}_{k+2}\colon \R_{+}\times\R_{+}\times\R_{+}\rightarrow\R_{+}$ and $\hat{b}_{k+4}\colon\R_{+}\times\R_{+}\rightarrow\R_{+}$ are given by
\begin{align*}
\overline{b}_{k+1}(y) &:= \sum_{i=0}^{k}\overline{b}_{i}y^{i},\quad
\overline{b}_{k+2}(\cdot,y,\tilde{y}) := \sup_{x\in [0,y\vee\tilde{y}]}\sup_{\nu\in\mathcal{P}_{0}(\R)}\bigg|\sum_{i=1}^{k}b_{i}(s,y,\nu)ix^{i-1}\bigg|,\\
\hat{b}_{k+4}(s,y) &:= \sup_{x\in [0,y]}\sup_{\nu\in\mathcal{P}_{0}(\R)}\bigg|\sum_{i=1}^{k}b_{i}(s,y,\nu)ix^{i-1}\bigg|.
\end{align*}
Here, we note that $\hat{b}_{k+4}(\cdot,y) \leq \sum_{i=1}^{k}\hat{b}_{i}(\cdot,y)iy^{i-1}$ for each $y\geq 0$. In addition, we impose a \emph{Hölder continuity condition on compact sets} on $f$:
\begin{enumerate}[label=(C.\arabic*), ref=C.\arabic*, leftmargin=\widthof{(C.9)} + \labelsep]
\setcounter{enumi}{8}
\item\label{co:9} For any $n\in\N$ there is $\lambda_{n,0}\geq 0$ such that $|f(s,x,y) - f(\tilde{s},x,y)| \leq \lambda_{n,0}|s-\tilde{s}|^{\frac{1}{2}}$ for any $s,\tilde{s}\geq 0$ and $x,y\in [0,n]$.
\end{enumerate}

\begin{Example}
Assume that the representation for $f$ in Example~\ref{ex:sums of power functions} holds and there is $\overline{g}\in\R_{+}^{d\times m}$ satisfying $|g_{i,j}(s) - g_{i,j}(\tilde{s})| \leq \overline{g}_{i,j}|s-\tilde{s}|^{\frac{1}{2}}$ for any $i\in\{1,\dots,d\}$, $j\in\{1,\dots,m\}$ and $s,\tilde{s}\geq 0$. Then~\eqref{co:9} is valid.
\end{Example}

From~\eqref{co:5} and~\eqref{co:9} we obtain the following \emph{continuity condition} for the diffusion coefficient of~\eqref{eq:SIS epidemic model}, which is stronger than~\eqref{eq:Lipschitz condition on the diffusion coefficient}:
\begin{equation}\label{eq:Hoelder condition on the diffusion coefficient}
\begin{split}
|f(s,x,N(s) - &x) - f(\tilde{s},\tilde{x},N(\tilde{s}) - \tilde{x})| \leq \lambda_{0}(N(\tilde{s}))|s-\tilde{s}|^{\frac{1}{2}}\\
&\quad + \lambda(s,N(s)\vee N(\tilde{s}))\bigg(\frac{1}{2}|N(s) - N(\tilde{s})| + |x-\tilde{x}|\bigg)
\end{split}
\end{equation}
for all $s,\tilde{s}\geq 0$, $x\in [0,N(s)]$ and $\tilde{x}\in [0,N(\tilde{s})]$, where the measurable locally bounded function $\lambda_{0}\colon\R_{+}\rightarrow\R_{+}$ is given by $\lambda_{0}(y) := \lambda_{\lceil y\rceil,0}$ for $y > 0$ and $\lambda_{0}(0) := 0$.

If in addition~\eqref{co:7} and~\eqref{co:8} hold, then for any $n\in\N$ and $j\in\{0,\dots,k_{n}-1\}$ we may introduce four $\R_{+}$-valued measurable bounded functions on $[t_{j,n},t_{j+1,n}]$ by
\begin{align*}
\delta_{p,j,n}^{(1)} &:=  \overline{b}_{k+1}(N(t_{j,n})) + 6c_{p}\lambda_{0}(N(t_{j,n}))^{2},\\
\delta_{p,j,n}^{(2)} &:=  \overline{b}_{k+1}(N(t_{j,n})) + \overline{b}_{k+2}(\cdot,N,N(t_{j,n})) + \frac{3}{2}c_{p}\lambda(\cdot,N\vee N(t_{j,n}))^{2},\\
\delta_{p,j,n}^{(3)} &:= \hat{b}_{k+4}(\cdot,N) + 12c_{p}\lambda(\cdot,N\vee N(t_{j,n}))^{2}
\end{align*}
and
\begin{equation}\label{eq:error coefficient 2}
\begin{split}
\lambda_{p,j,n} &:= (p-1)\big(2\overline{b}_{k+1}(N(t_{j,n})) + \overline{b}_{k+2}(\cdot,N,N(t_{j,n}))\big) + (2p-1)\hat{b}_{k+4}(\cdot,N)\\
&\quad + 3(p - 2)c_{p}\lambda_{0}(N(t_{j,n}))^{2} + 13(p-1)c_{p}\lambda(\cdot,N\vee N(t_{j,n}))^{2}.
\end{split}
\end{equation}
We also define a measurable bounded function $\delta_{p,j,n}\colon [t_{j,n},t_{j+1,n}]\rightarrow\R_{+}$, by using the coefficient~\eqref{eq:error coefficient 1}, via
\begin{equation}\label{eq:error coefficient 3}
\delta_{p,j,n}(s) := \big(\delta_{p,j,n}^{(1)}(s) + \delta_{p,j,n}^{(3)}(s)m_{p,j,n}^{p}\big)|s - t_{j,n}|^{\frac{p}{2}} +  \delta_{p,j,n}^{(2)}(s)|N(s) - N(t_{j,n})|^{p}.
\end{equation}
Then the \emph{intermediate $L^{p}$-error estimate} for $\hat{I}^{(n)}$ in terms of $(L_{j,n})_{j\in\{0,\dots,k_{n}-1\}}$ takes the following form.

\begin{Proposition}\label{pr:intermediate error estimate}
Let~\eqref{co:3},~\eqref{co:5} and~\eqref{co:7}--\eqref{co:9} hold. Then the unique strong solution $I$ to~\eqref{eq:SIS epidemic model} with $I_{0} = \xi_{0}$ a.s.~satisfies
\begin{align*}
u(t)\E\big[|\hat{I}_{t}^{(n)} &- I_{t}|^{p}\big] \leq 
u(t_{j,n})\E\big[|I_{t_{j,n}}^{(n)} - I_{t_{j,n}}|^{p}\big]\\
&\quad  + \int_{t_{j,n}}^{t}u(s)\big(\delta_{p,j,n}(s) + \hat{\lambda}_{k+1}(s,N(s))p\E\big[|\hat{I}_{s}^{(n)} - I_{s}|^{p-1}\mathcal{W}_{p}(\mathcal{L}(I_{s}),L_{j,n})\big]\big)\,\mathrm{d}s\\
&\quad + \int_{t_{j,n}}^{t}\big(\dot{u}(s) + u(s)\lambda_{p,j,n}(s)\big)\E\big[|\hat{I}_{s}^{(n)} - I_{s}|^{p}\big]\,\mathrm{d}s
\end{align*}
for all $n\in\N$, $j\in\{0,\dots,k_{n}-1\}$ and $t\in [t_{j,n},t_{j+1,n}]$ and each absolutely continuous function $u\colon [t_{j,n},t_{j+1,n}]\rightarrow\R_{+}$.
\end{Proposition}

Lastly, Proposition~\ref{pr:intermediate error estimate} allows us to establish a strong error estimate for the interpolated Euler--Maruyama scheme applied to an \emph{interacting particle system approximating~\eqref{eq:SIS epidemic model}}, by choosing $L_{j,n}$ as \emph{empirical measures} for any $n\in\N$ and $j\in\{0,\dots,k_{n}-1\}$.

To this end, let $(\xi_{\ell})_{\ell\in\N}$ be a sequence of independent and identically distributed random variables such that $\mathcal{L}(\xi_{1}) = \mathcal{L}(\xi_{0})$ and $(W^{(\ell)})_{\ell\in\N}$ be a sequence of independent standard $d$-dimensional $\mathbb{F}$-Brownian motions. For each $\ell\in\N$ we consider the McKean--Vlasov SDE~\eqref{eq:SIS epidemic model} driven by $W^{(\ell)}$ instead of $W$:
\begin{equation}\label{eq:copy of the SIS epidemic model}
\mathrm{d}I_{t}^{(\ell)} = \sum_{i=0}^{k}b_{i}\big(t,N(t),\mathcal{L}(I_{t}^{(\ell)})\big)\big(I_{t}^{(\ell)}\big)^{i}\,\mathrm{d}t + f\big(t,I_{t}^{(\ell)},N(t) - I_{t}^{(\ell)}\big)\,\mathrm{d}W_{t}^{(\ell)}
\end{equation}
for $t\geq 0$ with initial condition $I_{0}^{(\ell)} = \xi_{\ell}$. For each $n\in\N$ we take $M_{n}\in\N$ and obtain an $M_{n}$-dimensional system~\eqref{eq:copy of the SIS epidemic model} of SDEs once the distribution $\mathcal{L}(I_{t}^{(\ell)})$ in~\eqref{eq:copy of the SIS epidemic model} is replaced by the empirical measure $\frac{1}{M_{n}}\sum_{m=1}^{M_{n}}\delta_{I_{t}^{(m)}}$ for all $\ell = 1,\dots,M_{n}$.

The \emph{continuous-time Euler--Maruyama scheme} along $\mathbb{T}_{n}$ for this system is given by the $\mathbb{F}$-adapted right-continuous processes $I^{(n,1)},\dots,I^{(n,M_{n})}\colon [0,T]\times\Omega\rightarrow\R$ first recursively defined by $I_{0}^{(n,\ell)} := \xi_{\ell}$ and
\begin{equation}\label{eq:Euler--Maruyama scheme for the particle system}
\begin{split}
I_{t_{j+1},n}^{(n,\ell)} &:= I_{t_{j},n}^{(n,\ell)} + \sum_{i=0}^{k}\overline{b}_{i}(t_{j,n},N(t_{j,n}),L_{j,n})\big((I_{t_{j,n}}^{(n,\ell)})^{+}\wedge N(t_{j,n})\big)^{i}(t_{j+1,n} - t_{j,n})\\
&\quad + \sum_{i=1}^{d}\overline{f}_{i}\big(t_{j,n},I_{t_{j,n}}^{(n,\ell)},N(t_{j,n}) - I_{t_{j,n}}^{(n,\ell)}\big)(W_{t_{j+1},n}^{(\ell,i)} - W_{t_{j,n}}^{(\ell,i)})
\end{split}
\end{equation}
with the empirical measure
\begin{equation}\label{eq:empirical measure of interacting particles}
L_{j,n} = \frac{1}{M_{n}}\sum_{m=1}^{M_{n}}\delta_{I_{t_{j,n}}^{(n,m)}}
\end{equation}
for all $\ell\in\{1,\dots,M_{n}\}$ and $j\in \{0,\dots,k_{n}-1\}$ and then constantly extended by setting $I_{t}^{(n,\ell)} := I_{t_{j,n}}^{(n,\ell)}$ for any $t\in ]t_{j,n},t_{j+1,n}[$. Here, $W^{(\ell,i)}$ denotes the $i$th coordinate of $W^{(\ell)}$ for any $i\in\{1,\dots,d\}$. The corresponding \emph{interpolated Euler--Maruyama scheme} is represented by the unique $\mathbb{F}$-semimartingales $\hat{I}^{(n,1)},\dots,\hat{I}^{(n,M_{n})}$ satisfying
\begin{equation}\label{eq:interpolated Euler--Maruyama scheme for the particle system}
\begin{split}
\hat{I}_{t}^{(n,\ell)} &= I_{t_{j,n}}^{(n,\ell)} + \sum_{i=0}^{k}\overline{b}_{i}(t_{j,n},N(t_{j,n}),L_{j,n})\big((I_{t_{j,n}}^{(n,\ell)})^{+}\wedge N(t_{j,n})\big)^{i}(t - t_{j,n})\\
&\quad + \sum_{i=1}^{d}\overline{f}_{i}\big(t_{j,n},I_{t_{j,n}}^{(n,\ell)},N(t_{j,n}) - I_{t_{j,n}}^{(n,\ell)}\big)(W_{t}^{(\ell,i)} - W_{t_{j,n}}^{(\ell,i)})
\end{split}
\end{equation}
with the same empirical measure~\eqref{eq:empirical measure of interacting particles} for all $\ell\in\{1,\dots,M_{n}\}$, $j\in\{0,\dots,k_{n}-1\}$ and $t\in [t_{j,n},t_{j+1,n}]$. We notice that, under~\eqref{co:3},~\eqref{co:5} and~\eqref{co:7}, Proposition~\ref{pr:existence, uniqueness and a growth estimate} yields a unique strong solution $I^{(\ell)}$ to~\eqref{eq:copy of the SIS epidemic model} with $I_{0}^{(\ell)} = \xi_{l}$ a.s.~for each $\ell\in\N$.

Moreover, as Proposition~\ref{pr:pathwise uniqueness} ensures pathwise uniqueness for~\eqref{eq:copy of the SIS epidemic model}, $(I_{t}^{(\ell)})_{\ell\in\N}$ is a sequence of independent and identically distributed random variables for any $t\geq 0$. Hence,~\cite[Theorem~1]{FouGui15} yields for any $q > p$ with $q\neq 2p$ a constant $c_{p,q} > 0$ that only depends on $p$ and $q$ and is explicitly computed there such that
\begin{equation}\label{eq:estimate for the empirical measure}
\E\bigg[\mathcal{W}_{p}\bigg(\mathcal{L}(I_{s}^{(1)}),\frac{1}{M_{n}}\sum_{\ell=1}^{M_{n}}\delta_{I_{s}^{(\ell)}}\bigg)^{p}\bigg] \leq c_{p,q}\E\big[I_{s}^{q}\big]^{\frac{p}{q}}\big(M_{n}^{-\frac{1}{2}} + M_{n}^{\frac{p}{q} - 1}\big)
\end{equation}
for all $s\geq 0$. If in addition~\eqref{co:8} and~\eqref{co:9} hold, then for any $j\in\{0,\dots,k_{n}-1\}$ we may define measurable bounded functions $\hat{\delta}_{p,j,n},\hat{\lambda}_{p,j,n}\colon [t_{j,n},t_{j+1,n}]\rightarrow\R_{+}$ by
\begin{align*}
\hat{\lambda}_{p,j,n} &:= \max_{i=0,\dots,j}\lambda_{p,i,n}+ (3p-2)\hat{\lambda}_{k+1}(\cdot,N),\\
\hat{\delta}_{p,j,n} &:= \max_{i=0,\dots,j}\delta_{p,i,n} + \hat{\lambda}_{k+1}(s,N(s))m_{p,i,n}^{p}(s - t_{i,n})^{\frac{p}{2}}
\end{align*}
in terms of the coefficients~\eqref{eq:error coefficient 1},~\eqref{eq:continuity condition on the drift coefficient},~\eqref{eq:error coefficient 2} and~\eqref{eq:error coefficient 3}. We notice that if $N$ is $\alpha$-Hölder continuous on $[0,T]$ for some $\alpha\in ]0,\frac{1}{2}]$, that is, there is $\hat{c}_{\alpha} \geq 0$ such that
\begin{equation}\label{eq:Hoelder condition on the function N}
|N(s) - N(t)| \leq \hat{c}_{\alpha}|s - t|^{\alpha}\quad\text{for all $s,t\in [0,T]$,}
\end{equation}
then~\eqref{eq:error coefficient 3} yields an explicit constant $\hat{c}_{\alpha,0} \geq 0$ such that $\sup_{s\in [0,T]}\hat{\delta}_{p,k_{n}-1,n}(s)$ $\leq \hat{c}_{\alpha,0}|\mathbb{T}_{n}|^{\alpha p}$ for all $n\in\N$. Finally, we set
\begin{equation*}
\hat{c}_{p} := e^{\frac{1}{p}\int_{0}^{T}\hat{\lambda}_{p,k_{n}-1,n}(s)\,\mathrm{d}s}\quad\text{and}\quad \hat{c}_{p,0} := \int_{0}^{T}\hat{\lambda}_{k+1}(s,N(s))N(s)^{p}\,\mathrm{d}s.
\end{equation*}
These preparations lead us to an \emph{explicit strong $L^{p}$-error estimate} for the interpolated Euler--Maruyama scheme $\hat{I}^{(n,1)},\dots,\hat{I}^{(n,M_{n})}$ applied to an interacting particle system.

\begin{Theorem}\label{thm:strong error estimate}
Let~\eqref{co:3},~\eqref{co:5} and~\eqref{co:7}--\eqref{co:9} hold, and for each $\ell\in\N$ let $I^{(\ell)}$ denote the unique strong solution to~\eqref{eq:copy of the SIS epidemic model} with $I_{0}^{(\ell)} = \xi_{\ell}$ a.s. Then
\begin{equation*}
\max_{\ell=1,\dots,M_{n}}\E\big[|\hat{I}_{t}^{(n,\ell)} - I_{t}^{(\ell)}|^{p}\big] \leq e^{\int_{0}^{t}\hat{\lambda}_{p,j,n}(s)\,\mathrm{d}s}\int_{0}^{t}\hat{\delta}_{p,j,n}(s) + 2c_{p,q}\hat{\lambda}_{k+1}(s,N(s))N(s)^{p}M_{n}^{-\frac{1}{2}}\,\mathrm{d}s
\end{equation*}
for all $n\in\N$, $j\in \{0,\dots,k_{n}-1\}$, $t\in[t_{j,n}, t_{j+1,n}]$ and $q > 2p$. In particular, if~\eqref{eq:Hoelder condition on the function N} and $M_{n}\geq \hat{c}_{0}|\mathbb{T}_{n}|^{-2\alpha p}$ hold for all $n\in\N$ and some $\alpha\in ]0,\frac{1}{2}]$ and $\hat{c}_{0},\hat{c}_{\alpha} \geq 0$, then
\begin{equation*}
\sup_{t\in[0,T]}\max_{\ell=1,\dots,M_{n}}\E\big[|\hat{I}_{t}^{(n,\ell)} - I_{t}^{(\ell)}|^{p}\big]^{\frac{1}{p}} \leq \hat{c}_{p,q,\alpha}|\mathbb{T}_{n}|^{\alpha}
\end{equation*}
for all $n\in\N$ and $q > 2p$, where $\hat{c}_{p,q,\alpha} := \hat{c}_{p}(\hat{c}_{\alpha,0}T + 2c_{p,q}\hat{c}_{0}^{-1/2}\hat{c}_{p,0})^{1/p}$.
\end{Theorem}

\subsection{Simulation of the epidemic model}\label{se:4.2}

To illustrate the dynamics of our representative mean-field SIS model~\eqref{eq:a representative SIS epidemic model}, we simulate the extended law-dependent version of the model~\cite{GraGreHuMaoPan11} discussed in the introduction. In this setting, the McKean--Vlasov SDE describing the evolution of the number of infected individuals takes the following form:
\begin{equation}\label{eq:simulated model}
\mathrm{d}I_{t} = \bigg(\beta\bigg(1 + \alpha\frac{\E\big[I_{t}\big]}{N}\bigg)I_{t}(N - I_{t}) - (\mu - \gamma)I_{t}\bigg)\,\mathrm{d}t + \sigma I_{t}(N - I_{t})\,\mathrm{d}B_{t},
\end{equation}
where $N > 0$, $\beta,\gamma,\mu,\sigma\geq 0$ and $\alpha \geq - 1$. Here, we assume that the dependence of the disease transmission rate on the average infection level is proportional to the transmission rate $\beta$ within any subpopulation in the absence of inter-population interactions. That is, the parameter $\beta_{1}$ in~\eqref{eq:standard SIS epidemic model with an expected value} is given by $\beta_{1} = \alpha\beta$. This model constitutes a special case of~\eqref{model:2}, and Proposition~\ref{pr:existence, uniqueness and a growth estimate} yields a unique strong solution $I$ with positive paths satisfying $I_{0} = i_{0}$ a.s.~for $i_{0}\in ]0,N[$.

With our simulations, we examine the impact of the mean-field interaction term on the epidemiological dynamics. More specifically, we provide numerical visualizations to analyse the effect of the \emph{interaction parameter} $\alpha$ and to support our theoretical findings on extinction and persistence established in Section~\ref{se:3}. Moreover, we numerically investigate the transition between these two regimes in terms of $\alpha$, illustrating the \emph{critical influence of the mean-field interaction on the long-term behaviour of the epidemic}.

Simulations are performed using Euler--Maruyama scheme~\eqref{eq:Euler--Maruyama scheme for the particle system} with an equidistant partition $\mathbb{T}_{n}$, step size $|\mathbb{T}_{n}| = 0.001$ and $M_{n} = 10000$ particles. For a direct comparison with the original model in~\cite{GraGreHuMaoPan11}, we choose parameter values considered there and analyse representative values of $\alpha$, including $\alpha=0$. As in \cite{GraGreHuMaoPan11}, we assume that the unit of time is one day and that the population size is measured in units of one million.

To isolate the effect of the mean-field interaction term, we couple our simulations by using identical Brownian increments for each particle across all values of $\alpha$. As a result, differences between paths or empirical means for different values of $\alpha$ are solely due to the mean-field interaction rather than stochastic variability, enhancing visual comparability.

\begin{Example}\label{ex:parameters extinction}
We consider model~\eqref{eq:simulated model} with the following parameters, agreeing with those in~\cite[Example 4.4]{GraGreHuMaoPan11} and ensuring exponential extinction of $I$ for $\alpha < 0.54$:
\begin{equation}\label{eq:parameters 1}
N=100,\quad \beta = 0.5,\quad\mu = 20, \quad \gamma = 25, \quad \sigma = 0.08.
\end{equation}
Indeed, inequality~\eqref{eq:extinction in the epidemic models 11} in Example~\ref{ex:extinction 2} reduces to $\beta(1 + \alpha^{+})N - \frac{1}{2}\sigma^{2}N^{2}  < \mu + \gamma$, which holds if and only if $50(1 + \alpha^{+}) - 32 < 45$.
\end{Example}

For parameters~\eqref{eq:parameters 1} and $i_{0}=50$, 
Panel~(a) in Figure~\ref{fig:extinction} shows four sample paths for representative values of $\alpha$. To illustrate the contribution of the mean-field term to the dynamics, panel~(b) displays the respective empirical means $\frac{1}{M_{n}}\sum_{\ell=1}^{M_{n}}I^{(n,\ell)}$ of the particle systems $(I^{(n,1)},\dots,I^{(n,M_{n})})$. In view of the initial motivation of~\eqref{eq:simulated model}, the empirical mean can be interpreted as the number of infected individuals averaged over all subpopulations in the underlying population system.

\begin{figure}[h]
\centering
\includegraphics[height = 0.79\linewidth]{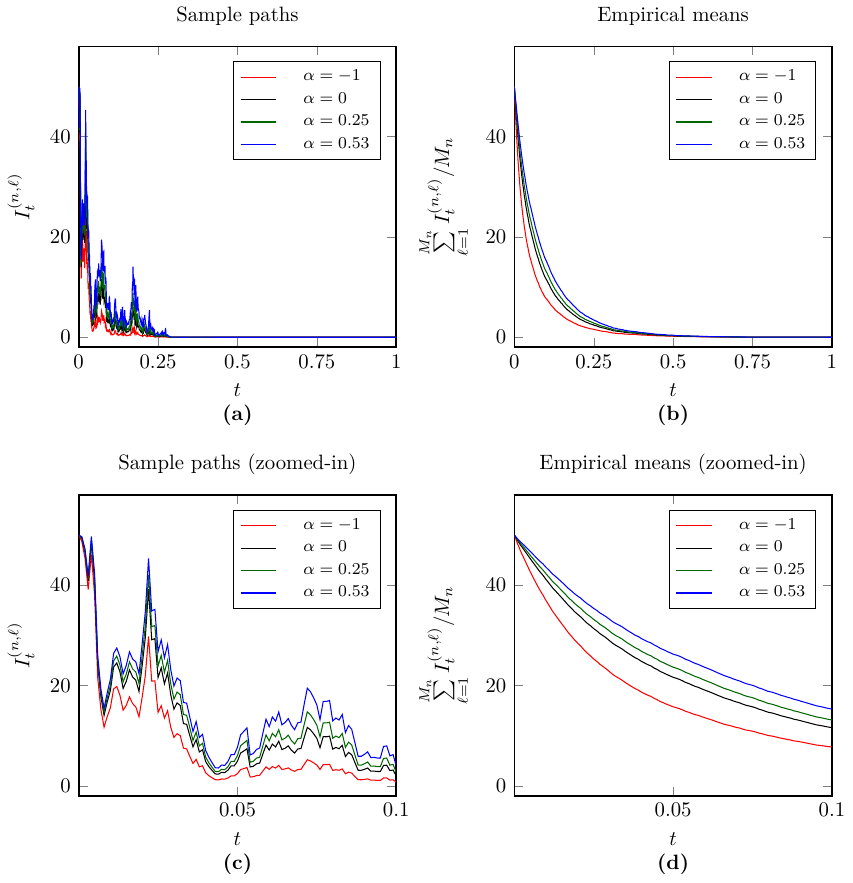}
\caption{\small{Simulation of model~\eqref{eq:simulated model} with parameters~\eqref{eq:parameters 1} and $i_{0}=50$, using Euler--Maruyama scheme~\eqref{eq:Euler--Maruyama scheme for the particle system} with $T = 1$, $|\mathbb{T}_{n}| = 0.001$ and $M_{n} = 10000$. Panels~(a) and~(b) display four sample paths, coupled across $\alpha$, alongside the respective empirical means, while panels~(c) and~(d) provide magnified views of these plots.}}
\label{fig:extinction}
\end{figure}

Confirming our theoretical results, Figure~\ref{fig:extinction} shows that the epidemic undergoes rapid extinction for parameters~\eqref{eq:parameters 1}. As the differences between the trajectories and the empirical means in panels~(a) and~(b) are barely visible, magnified views of the respective trajectories are provided in panels~(c) and~(d).

We observe that the paths and the corresponding empirical means are monotonically increasing in $\alpha$. This behaviour is in line with our epidemiological interpretation of the mean-field interaction, since increasing $\alpha$ enhances the disease transmission rate, resulting in a higher number of infections, while decreasing $\alpha$ has the opposite effect.

\begin{Example}\label{ex:parameters persistence}
Let the parameters of model~\eqref{eq:simulated model} be given as in~\cite[Example 5.3]{GraGreHuMaoPan11}. Namely,
\begin{equation}\label{eq:parameters 2}
N=100,\quad \beta = 0.5,\quad  \mu= 20,\quad \gamma = 25,\quad \sigma = 0.01. 
\end{equation}
Then Example~\ref{ex:persistence 2} ensures that $I$ persists above $x_{\alpha}\in ]0,100[$ for each $\alpha > -0.09$, where
\begin{equation*}
x_{\alpha} := N -\beta\frac{1 + \alpha^{+}}{\sigma^{2}} + \sqrt{\beta^{2}\frac{(1 +\alpha^{+})^{2}}{\sigma^{4}} - 2\frac{\mu + \gamma + \beta|\alpha|N}{\sigma^{2}}}.
\end{equation*}
In fact,~\eqref{eq:sufficient condition for persistence in model 2} for $y = -\beta\alpha^{-}$ becomes $\beta(1 - \alpha^{-})N - \frac{1}{2}\sigma^{2}N^{2} > \mu + \gamma$, which is equivalent to $50(1-\alpha^{-}) - 0.5 > 45$. In particular, by rounding to four decimal places, we obtain
\begin{equation}\label{eq:lower persistence levels}
x_{-0.08} \approx 1.0203,\quad x_{0} \approx 9.1751,\quad x_{0.5} \approx 6.0786,\quad x_{1} \approx 4.5444.
\end{equation}
\end{Example}

For parameters~\eqref{eq:parameters 2} and $i_{0} = 50$, panel~(a) in Figure~\ref{fig:persistence} shows four trajectories, coupled across $\alpha$. The respective values of the lower persistence level $x_{\alpha}$ are displayed as dashed horizontal lines. Panel~(b) depicts the corresponding empirical means, while panels~(c) and~(d) show, for better visibility, magnified views of the paths for $\alpha\in\{-0.08,0\}$.

\begin{figure}[h]
\centering
\includegraphics[height = 0.79\linewidth]{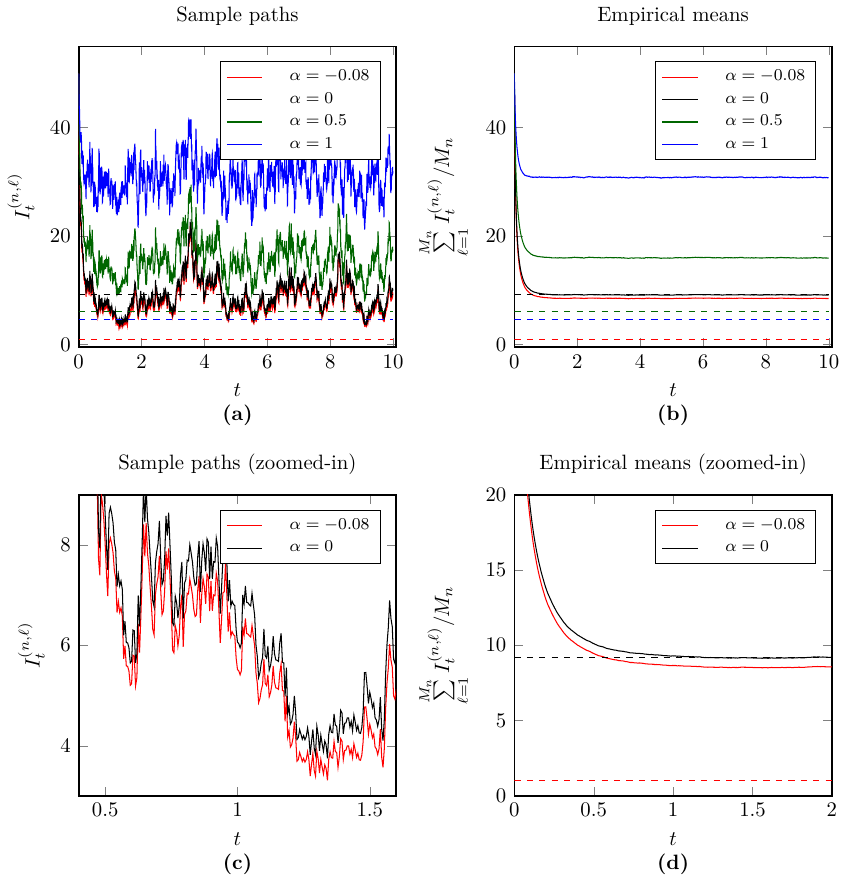}
\caption{Simulation of model~\eqref{eq:simulated model} with parameters~\eqref{eq:parameters 2} and $i_{0} = 50$, using scheme~\eqref{eq:Euler--Maruyama scheme for the particle system} with $T = 10$, $|\mathbb{T}_{n}| = 0.001$ and $M_{n} = 10000$. Panels~(a) and~(b) show four sample paths, coupled across $\alpha$, alongside the corresponding empirical means. Panels~(c) and~(d) provide magnified views of the red and black curves, and the dashed lines indicate the respective lower persistence levels~\eqref{eq:lower persistence levels}.}
\label{fig:persistence}
\end{figure}

Once again, we notice that the paths and empirical means are monotonically increasing in $\alpha$. However, Figure~\ref{fig:persistence} indicates that persistence levels~\eqref{eq:lower persistence levels} are certainly improvable for $\alpha > 0$. If more information about the limiting behaviour of $\alpha\E[I]$ were available, Examples~\ref{ex:persistence 1}--\ref{ex:persistence 3} would yield more precise persistence levels.

\begin{Example}\label{ex:parameter persistence 2}
For parameters~\eqref{eq:parameters 2} and $i_{0} = 50$, Figure~\ref{fig:persistence} indicates that the term $\alpha\E[I_{t}]$ converges to some $m^{\alpha}_{\infty}\in\R$ as $t\uparrow\infty$ for each $\alpha\in \{-0.08, 0.5, 1\}$. Specifically, we obtain the numerical estimates
\begin{equation}\label{eq:estimates empirical mean}
m^{-0.08}_{\infty} \approx -0.6840,\quad  m^{0.5}_{\infty} \approx 8.0154,\quad m^{1}_{\infty} \approx 30.8548,
\end{equation} 
based on the scaled empirical mean $\frac{\alpha}{M_{n}}\sum_{\ell=1}^{M_{n}}I_{T}^{(n,\ell)}$. As inequality~\eqref{eq:sufficient condition for persistence in model 2} for $y = \beta\frac{m_{\infty}^{\alpha}}{N}$ reduces to $50 + 0.5 m_{\infty}^{\alpha} - 0.5 > 45$, Example~\ref{ex:persistence 2} implies persistence around the level
\begin{equation*}
x_{\alpha} = 100 -5000(1+m^{\alpha}_{\infty}/100) + 10000\sqrt{0.25(1+m_{\infty}^{\alpha}/100)^{2} - 0.009}, 
\end{equation*}
provided that $m_{\infty}^{\alpha} > -9$. Under the assumption that the estimates~\eqref{eq:estimates empirical mean} are sufficiently precise, this yields the following estimated persistence levels:
\begin{equation}\label{eq:approximated persistence levels}
x_{-0.08} \approx 8.5379,\quad x_{0.5} \approx 16.0257,\quad x_{1} \approx 30.8561.
\end{equation}
Figure~\ref{fig:persistence-levels} illustrates the accuracy of these approximations and shows that the persistence level $x_{\alpha}$ aligns with the limit $\lim_{t\uparrow\infty}\E[I_{t}]$.
\end{Example}

\begin{figure}[h]
\centering
\includegraphics[height = 0.51\linewidth]{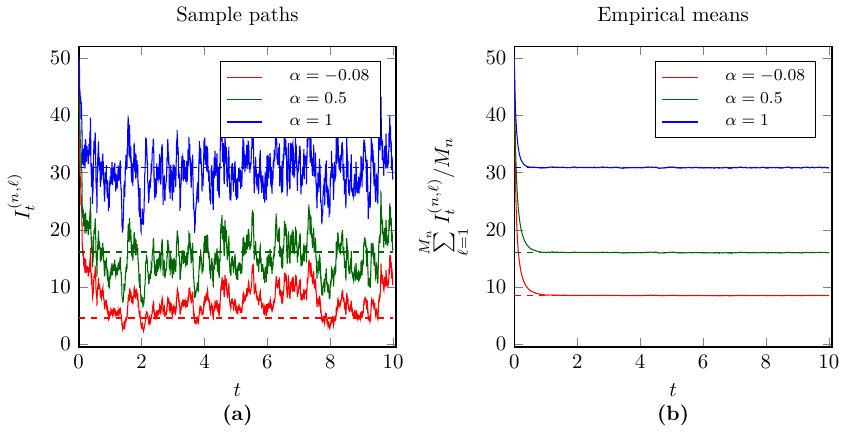}
\caption{Simulation of model~\eqref{eq:simulated model} with parameters~\eqref{eq:parameters 2} and $i_{0}=50$, using scheme~\eqref{eq:Euler--Maruyama scheme for the particle system} with $T = 10$, $|\mathbb{T}_{n}| = 0.001$ and $M_{n} = 10000$. Panels~(a) and~(b) display three sample paths, coupled across $\alpha$, alongside the respective empirical means. The dashed lines indicate the corresponding approximated persistence levels~\eqref{eq:approximated persistence levels}.}
\label{fig:persistence-levels}
\end{figure}

Finally, in the presence of mean-field interaction, that is, when $\alpha \neq 0$, our results on the pathwise asymptotic behaviour of the epidemic, derived in Section~\ref{se:3}, do not completely characterise the \emph{transition between extinction and persistence of the disease}. In contrast to the case $\alpha=0$, in which the stochastic reproduction number $R_{0}^{S} = \frac{\beta N - \frac{1}{2}\sigma^{2}N^{2}}{\mu + \gamma}$ provides insight into this transition, our simulations in Figure~\ref{fig:behaviour-transition} indicate that outside the parameter regime covered by our analysis the epidemic outcome may depend on the choice of $i_{0}$.

\begin{figure}[t]
\centering
\includegraphics[height = 0.51\linewidth]{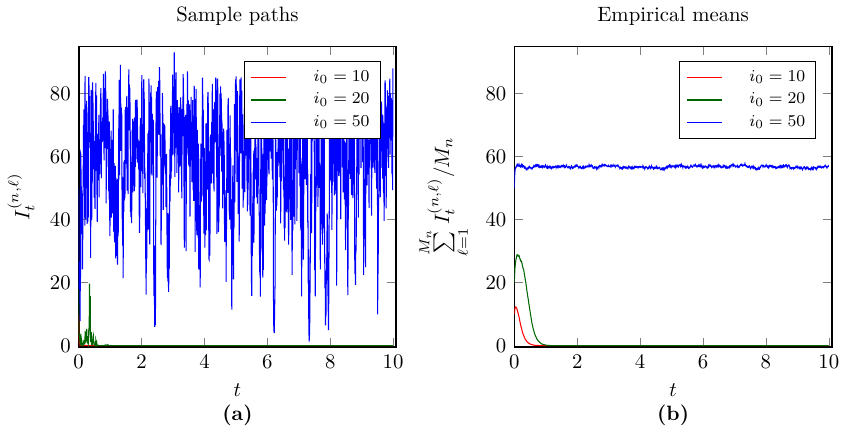}
\caption{Simulation of~\eqref{eq:simulated model} with parameters~\eqref{eq:parameters 1} and $\alpha =2.5$, using scheme~\eqref{eq:Euler--Maruyama scheme for the particle system} with $T = 10$, $|\mathbb{T}_{n}| = 0.001$ and $M_{n} = 10000$. Panels (a) and~(b) display three sample paths, coupled across $i_{0}$, alongside the respective empirical means.}
\label{fig:behaviour-transition}
\end{figure}

\section{Proofs of the main results}\label{se:5}

\subsection{Maxima and zeros of a sum of power functions}

In this self-contained section, we maximise functions $f\colon [0,y]\rightarrow\R$ of the specific form $f(x)$ $= a + bx + cx^{2} + d(y-x)^{\frac{3}{2}}$ for all $x\in [0,y]$ and consider their zeros, where $a,b,c,d\in\R$ and $y \geq 0$. The pathwise asymptotic analysis in Section~\ref{se:3} builds upon these findings.

\begin{Lemma}\label{le:maximisation 1}
For $a,b,c\in\R$ and $y\geq 0$ we have
\begin{equation*}
\max_{x\in [0,y]} a + bx + cx^{2} = 
\begin{cases}
a - \frac{(b^{+})^{2}}{4c} & \text{if $c < 0$ and $b < -2cy$}\\
a + by + cy^{2}& \text{if $c < 0$ and $b\geq -2 cy$}\\
a + (b + cy)^{+}y & \text{otherwise}
\end{cases}
.
\end{equation*}
\end{Lemma}

\begin{proof}
It suffices to consider the case $y > 0$ and $c\neq 0$. Then $a + bx + cx^{2} = c(x - x_{b,c})^{2} + y_{b,c}$ for all $x\in\R$ with $x_{b,c}:= - \frac{b}{2c}$ and $y_{b,c}:= a - c x_{b,c}^{2}$ by completing the square. So, for $c < 0$ we have
\begin{equation*}
\max_{x\in [0,y]} a + bx + cx^{2} = c(x_{b,c}^{+}\wedge y - x_{b,c})^{2} + y_{b,c} = 
\begin{cases}
a - \frac{(b^{+})^{2}}{4c} & \text{if $b < - 2cy$}\\
a + by + cy^{2} & \text{otherwise}
\end{cases}
.
\end{equation*}
If instead $c > 0$, then  $\max_{x\in [0,y]} a + bx + cx^{2} = a + bx_{*} + cx_{*}^{2} = a + (b + cy)^{+}y$ for $x_{*} := y\mathbbm{1}_{[0,\frac{y}{2}]}(x_{b,c})$.
\end{proof}

\begin{Lemma}\label{le:maximisation 2}
For $a,b,d\in\R$ and $y\geq 0$ it holds that
\begin{equation*}
\max_{x\in[0,y]} a + bx + d(y-x)^{\frac{3}{2}} =
\begin{cases}
a + by - \frac{4}{27}\frac{b^{3}}{d^{2}}  & \text{if $d<0$, $b<0$ and $y > \frac{4}{9}\frac{b^{2}}{d^{2}}$}\\
a + \max\{b,dy^{\frac{1}{2}}\}y & \text{otherwise}
\end{cases}
.
\end{equation*}
\end{Lemma}

\begin{proof}
Once again, we may assume that $y > 0$ and $d\neq 0$, by Lemma~\ref{le:maximisation 1}. The derivative of the continuously differentiable function $f\colon [0,y]\rightarrow\R$ given by $f(x) :=  a + bx$ $+\, d(y-x)^{\frac{3}{2}}$ is of the form $f'(x) = b - \frac{3}{2}d(y-x)^{1/2}$ for all $x\in [0,y]$.

Thus, for $d > 0$ the function $f$ is strictly convex and attains its maximum over $[0,y]$ on the boundary. If $d<0$ and $b\geq 0$, then $f$ is strictly increasing, ensuring that $\max_{x\in[0,y]}f(x) = f(y)$.

In the case that $b,d < 0$ we have $f'(x) \geq 0$ if and only if $x \leq x_{*}$ for any $x\in [0,y]$, where $x_{*} := y - \frac{4}{9}(\frac{b}{d})^{2}$. If in addition $x_{*} \geq 0$, then $x_{*}$ is the only critical point of $f$ and $\max_{x\in[0,y]}f(x) = f(x_{*})$. Otherwise, $\max_{x\in [0,y]}f(x) = f(0)$.
\end{proof}

\begin{Proposition}\label{pr:maximisation 3}
Let $a,b,d\in\R$, $c < 0$, $y\geq 0$ and $e := \frac{9}{4}d^{2} + (b + 2cy)8c$. If both $b + 2cy\geq 0$ and $d > 0$ or we have $b + 2cy < 0$, then
\begin{equation*}
\max_{x\in [0,y]} a + bx + cx^{2} + d(y - x)^{\frac{3}{2}} = a + by + cy^{2}  + x_{*}^{3}\bigg(cx_{*} + \frac{1}{2}d\bigg)^{-}
\end{equation*}
as soon as $e > 0$ and $x_{*} := -\frac{1}{4c}(\frac{3}{2}d + \sqrt{e})$ satisfies $x_{*} < \sqrt{y}$. Otherwise,
\begin{equation*}
\max_{x\in [0,y]} a + bx + cx^{2} + d(y - x)^{\frac{3}{2}} = a + \max\{b + cy, dy^{\frac{1}{2}}\}y.
\end{equation*}
\end{Proposition}

\begin{proof}
It suffices to consider the case that $y > 0$. Then the derivative of the continuously differentiable function $f\colon [0,y]\rightarrow\R$ defined by $f(x) := a + bx + cx^{2} + d(y-x)^{\frac{3}{2}}$ satisfies $f'(x) = b + 2cy - \frac{3}{2}dz - 2cz^{2}$ for all $x\in [0,y]$ and $z\in [0,y^{\frac{1}{2}}]$ with $z = (y-x)^{\frac{1}{2}}$.

Further, the polynomial function $p\colon\mathbb{C}\rightarrow\mathbb{C}$ given by $p(z) := b + 2cy - \frac{3}{2}dz - 2cz^{2}$ admits the two complex zeros
\begin{equation*}
z_{-} := -\frac{1}{4c}\bigg(\frac{3}{2}d - \sqrt{e}\bigg)\quad\text{and}\quad z_{+} := -\frac{1}{4c}\bigg(\frac{3}{2}d + \sqrt{e}\bigg),
\end{equation*}
which are real if and only if $e\geq 0$. In this case, $z_{-}\leq z_{+}$, and from $c < 0$ we infer that $p(x)\geq 0$ for $x\leq z_{-}$, $p(x)\leq 0$ for $x\in [z_{-},z_{+}]$ and $p(x)\geq 0$ for $x\geq z_{+}$. Further, we may set $x_{+} := y - z_{+}^{2}$.

So, let $e > 0$ and $x_{*} < \sqrt{y}$. If $b + 2cy\geq 0$ and $d > 0$, then $x_{-} := y - z_{-}^{2}$ lies in $]x_{+},y[$, which entails that $f'\geq 0$ on $[0,x_{+}]$, $f'\leq 0$ on $[x_{+},x_{-}]$ and $f'\geq 0$ on $[x_{-},y]$. Hence, as $p(x_{*}) = 0$, we obtain that
\begin{equation*}
\max_{x\in[0,y]} f(x) = \max\{f(x_{+}),f(y)\} = a + by + cy^{2}  + x_{*}^{3}\bigg(cx_{*} + \frac{1}{2}d\bigg)^{-}
\end{equation*}
in this case. If instead $b + 2cy < 0$, then $f'\geq 0$ on $[0,x_{+}]$ and $f'\leq 0$ on $[x_{+},y]$, from which we readily deduce that $\max_{x\in[0,y]} f(x) = f(x_{+})$.

If neither of these two scenarios holds, then a case distinction based on the polynomial function $p$ shows that $\max_{x\in [0,y]} f(x) = \max\{f(0),f(y)\} = a + \max\{b + cy, dy^{\frac{1}{2}}\}y$.
\end{proof}

\begin{Lemma}\label{le:unique zero of a sum of power functions}
Let $a,b,c,d\in\R$ and $y>0$ satisfy $a + dy^{\frac{3}{2}} > 0$. Then $f\colon [0,y]\rightarrow\R$ given by $f(x) := a + bx + cx^{2}+d(y-x)^{\frac{3}{2}}$ admits a unique zero in the following two cases:
\begin{enumerate}[(i)]
\item $d = 0$ and either $c\leq 0$ and $f(y)\leq 0$ or $c > 0$ and $f(y) < 0$.

\item $c\leq 0$, $d\neq 0$ and $f(y) < 0$.
\end{enumerate}
Moreover, in the first case, the zero $x_{0}$ of $f$ is of the form $x_{0} = \frac{-b-\sqrt{b^{2} - 4ac}}{2c}$, if $c\neq 0$, and $x_{0} = - \frac{a}{b}$, otherwise.
\end{Lemma}

\begin{proof}
(i) It suffices to consider the case $c\neq 0$, in which $f$ has at least one and at most two zeros. If $c<0$, then $\frac{-b+\sqrt{b^{2}-4ac}}{2c} < 0$, which verifies the assertions. If $c > 0$, then from $f(y) < 0$ and $\lim_{x\uparrow\infty} a + bx + cx^{2} = \infty$ we infer that $f$ admits the unique asserted zero.

(ii) The function $g\colon\R\rightarrow\R$ defined by $g(z) := a + by + cy^{2} - (b+2cy)z^{2} + dz^{3} + cz^{4}$ satisfies $g(z) = f(y - z^{2})$ for all $z\in [0,\sqrt{y}]$. Hence, the assertion follows if we can show that $g$ has exactly one zero in $]0,\sqrt{y}[$.

Since $g'(0) = 0$, we see that $g$ has at most two local extrema in $\R\setminus\{0\}$. If $g$ has no local extrema in $]0,\sqrt{y}[$, then from $g(0) = f(y) < 0$ and $g(\sqrt{y}) = f(0) > 0$ it follows that $g$ is increasing on $]0,\sqrt{y}[$ and has exactly one zero in this interval.

If $g$ has exactly one local extremum in $]0,\sqrt{y}[$, then there is $z_{0}\in ]0,\sqrt{y}[$ such that either $g$ is increasing on $]0,z_{0}[$ and decreasing on $]z_{0},\sqrt{y}[$ or $g$ is decreasing on $]0,z_{0}[$ and increasing on $]z_{0},\sqrt{y}[$. Since $g(0) < 0$ and $g(\sqrt{y}) > 0$, both cases imply that $g$ has a unique zero in $]0,\sqrt{y}[$.

Lastly, $g$ can have two local extrema in $]0,\sqrt{y}[$ only if $c<0$. In this case, $\lim_{z\downarrow -\infty} g(z)$ $= \lim_{z\uparrow \infty} g(z) = -\infty$. Hence, there are $z_{1},z_{2}\in]0,\sqrt{y}[$ with $z_{1} < z_{2}$, such that $g$ is decreasing on $]0,z_{1}[$ and $]z_{2},\sqrt{y}[$ and increasing on $]z_{1},z_{2}[$. As $g(0) < 0$ and $g(y) > 0$, we again conclude that $g$ has exactly one zero in $]0,\sqrt{y}[$.
\end{proof}

\begin{Remark}\label{re:unique zero of a sum of power functions}
If $c > 0$, $d = 0$ and $f(y) = 0$, then $f$ may have two zeros. In this case, either $\frac{-b-\sqrt{b^{2} - 4ac}}{2c}$ equals $y$ and is the only zero or it is the smallest zero of $f$.
\end{Remark}

\subsection{Proofs for the derivation of a unique strong solution}

\begin{proof}[Proof of Proposition~\ref{pr:pathwise uniqueness}]
As any weak solution to~\eqref{eq:SIS epidemic model} also solves~\eqref{eq:extended McKean--Vlasov SDE} weakly, it suffices to show that pathwise uniqueness holds for the latter equation relative to the functional $\Theta:\R_{+}\times\mathcal{P}_{1}(\R)\times\mathcal{P}_{1}(\R)\rightarrow\R_{+}$ defined by $\Theta(s,\nu,\tilde{\nu}) := \hat{\lambda}_{k+1}(s,N(s))\mathcal{W}_{1}(\nu,\tilde{\nu})$.

According to~\cite[Section~2.2]{KalMeyPro24-2}, this holds if any two weak solutions $I$ and $\tilde{I}$ to~\eqref{eq:extended McKean--Vlasov SDE} on a common filtered probability space relative to one standard $d$-dimensional Brownian motion are indistinguishable whenever $I_{0} = \tilde{I}_{0}$ a.s.~and the measurable function $I\rightarrow\R_{+}$, $s\mapsto\Theta(s,\mathcal{L}(I_{s}),\mathcal{L}(\tilde{I}_{s}))$ is locally integrable.

To this end, we observe that the drift $\overline{b}\colon\R_{+}\times\R\times\mathcal{P}_{1}(\R)\rightarrow\R$ and diffusion coefficient $\overline{\sigma}\colon\R_{+}\times\R\rightarrow\R^{1\times d}$ of the McKean--Vlasov SDE~\eqref{eq:extended McKean--Vlasov SDE} are given by
\begin{equation}\label{eq:extended coefficients}
\overline{b}(s,x,\nu) := \sum_{i=0}^{k}\overline{b}_{i}(s,N(s),\nu)\big(x^{+}\wedge N(s)\big)^{i}\quad\text{and}\quad\overline{\sigma}(s,x) := \overline{f}(s,x,N(s) - x).
\end{equation}
The partial Lipschitz condition~\eqref{eq:partial Lipschitz condition on the drift coefficient} for the drift coefficient of~\eqref{eq:SIS epidemic model}, which is ensured by~\eqref{co:1}, carries over to $\overline{b}$ when $\vartheta$ is replaced by $\mathcal{W}_{1}$. Namely, we have
\begin{equation}\label{eq:partial Lipschitz condition on the extended drift coefficient}
\sgn(x-\tilde{x})\big(\overline{b}(\cdot,x,\nu) - \overline{b}(\cdot,\tilde{x},\tilde{\nu})\big) \leq \hat{b}_{k+1}(\cdot,N)|x-\tilde{x}| + \hat{\lambda}_{k+1}(\cdot,N)\mathcal{W}_{1}(\nu,\tilde{\nu})
\end{equation}
for all $x,\tilde{x}\in\R$ and $\nu,\tilde{\nu}\in\mathcal{P}_{1}(\R)$. Indeed, since the function $b_{k+1}$ in~\eqref{eq:extended functions} satisfies $|b_{k+1}(x,y) - b_{k+1}(\tilde{x},y)| \leq |x-\tilde{x}|$ for any $x,\tilde{x}\in\R$ and $y\geq 0$, we infer from~\eqref{eq:domination condition} that
\begin{equation*}
\vartheta\big(\nu\circ b_{k+1}(\cdot,N)^{-1},\tilde{\nu}\circ b_{k+1}(\cdot,N)^{-1}\big)\leq \mathcal{W}_{1}(\nu,\tilde{\nu}),
\end{equation*}
which entails~\eqref{eq:partial Lipschitz condition on the extended drift coefficient}. The $\frac{1}{2}$-Hölder continuity condition~\eqref{co:2} on compact sets extends to $|\overline{f}(\cdot,x,y) - \overline{f}(\cdot,\tilde{x},\tilde{y})| \leq \lambda_{n}(|x-\tilde{x}|^{1/2} + |y-\tilde{y}|^{1/2})$ for all $n\in\N$ and $x,\tilde{x},y,\tilde{y}\in [-n,n]$. Consequently,
\begin{equation}\label{eq:Hoelder condition on compact sets on the extended diffusion coefficient}
|\overline{\sigma}(s,x) - \overline{\sigma}(s,\tilde{x})| \leq  2\lambda_{n + \lceil N(s)\rceil}(s)|x - \tilde{x}|^{\frac{1}{2}}
\end{equation}
for any $n\in\N$, $s\geq 0$ and $x,\tilde{x}\in [-n,n]$. Thus, pathwise uniqueness for~\eqref{eq:extended McKean--Vlasov SDE} relative to $\Theta$ follows from~\cite[Corollary~3.1]{KalMeyPro24-2}, as the required conditions~(C.1) and~(C.3) therein hold.

Regarding the remaining assertions, the difference $Y :=I - \tilde{I}$ is a random It{\^o} process in the sense of~\cite[Section~4.1]{KalMeyPro24-2}. That is, it is a semimartingale of the form
\begin{equation*}
Y = Y_{0} + \int_{0}^{\cdot}\mathrm{B}_{s}\,\mathrm{d}s + \int_{0}^{\cdot}\Sigma_{s}\,\mathrm{d}W_{s}\quad\text{a.s.,}
\end{equation*}
where the two $\mathbb{F}$-progressively measurable processes $\mathrm{B}$ and $\Sigma$ with respective values in $\R$ and $\R^{1\times d}$ are given by
\begin{equation}\label{eq:random drift and diffusion coefficients}
\begin{split}
\mathrm{B}_{s} &:= \sum_{i=0}^{k} b_{i}\big(s,N(s),\mathcal{L}(I_{s})\big)I_{s}^{i} - b_{i}\big(s,N(s),\mathcal{L}(\tilde{I}_{s})\big)\tilde{I}_{s}^{i},\\
\Sigma_{s} &:=  f(s,I_{s},N(s) - I_{s}) - f(s,\tilde{I}_{s},N(s) - \tilde{I}_{s}).
\end{split}
\end{equation}
Hence, the partial Lipschitz condition~\eqref{eq:partial Lipschitz condition on the drift coefficient} for the drift coefficient of~\eqref{eq:SIS epidemic model} entails that $\sgn(Y_{s})\mathrm{B}_{s}  \leq \hat{b}_{k+1}(s,N(s))|Y_{s}| +  \hat{\lambda}_{k+1}(s,N(s))\vartheta(\mathcal{L}(I_{s}),\mathcal{L}(\tilde{I}_{s}))$ for all $s\geq 0$. For this reason, the remaining claims are implied by~\cite[Theorem~4.2]{KalMeyPro24-2}.
\end{proof}

In the spirit of~\cite[Lemma~3.5]{BriGraKal24}, we provide sufficient verifiable estimates ensuring~\eqref{co:4}.

\begin{Lemma}\label{le:verifiable estimates}
If the following two conditions hold for some $m\in\N$ and $\zeta,\eta\in [\frac{1}{2},\infty[^{d\times m}$, then~\eqref{co:4} is satisfied:
\begin{enumerate}[(i)]
\item There is a measurable and locally square-integrable map $g\colon\R_{+}\rightarrow\R^{d\times m}$ such that $|f_{i}(\cdot,x,y)| \leq |\sum_{j=1}^{m} g_{i,j} x^{\zeta_{i,j}}y^{\eta_{i,j}}|$ for all $i\in\{1,\dots,d\}$ and $x,y\geq 0$.

\item Each $\nu\in\mathcal{P}_{0}(\R)$ satisfies
\begin{align*}
\max_{x\in [0,N]} \frac{1}{2}\sum_{i=1}^{d}\sum_{\substack{j_{1},j_{2}=1,\\ \zeta_{i,j_{1}} + \zeta_{i,j_{2}} < 2}}^{m}g_{i,j_{1}}g_{i,j_{2}}x^{\zeta_{i,j_{1}} + \zeta_{i,j_{2}} - 1}(N-x)^{\eta_{i,j_{1}} + \eta_{i,j_{2}}} &\leq b_{0}(\cdot,N,\nu),\\
\max_{x\in [0,N]}\frac{1}{2}\sum_{i=1}^{d}\sum_{\substack{j_{1},j_{2}=1,\\ \eta_{i,j_{1}} + \eta_{i,j_{2}} < 2}}^{m} g_{i,j_{1}}g_{i,j_{2}} (N-x)^{\zeta_{i,j_{1}} + \zeta_{i,j_{2}}}x^{\eta_{i,j_{1}} + \eta_{i,j_{2}} - 1} &\leq \dot{N} - \sum_{i=0}^{k}b_{i}(\cdot,N,\nu)N^{i}.
\end{align*}
\end{enumerate}
\end{Lemma}

\begin{proof}
By hypothesis, $|f(\cdot,x,y)|^{2} \leq \sum_{i=1}^{d}\sum_{j_{1},j_{2} = 1}^{m}g_{i,j_{1}}g_{i,j_{2}} x^{\zeta_{i,j_{1}} + \zeta_{i,j_{2}}}y^{\eta_{i,j_{1}} + \eta_{i,j_{2}}}$ for any $x,y\geq 0$. Further, 
\begin{align*}
\frac{1}{2x^{2}}\sum_{i=1}^{d}\sum_{j_{1},j_{2}=1}^{m}g_{i,j_{1}}g_{i,j_{2}}x^{\zeta_{i,j_{1}} + \zeta_{i,j_{2}}}(N-x)^{\eta_{i,j_{1}} + \eta_{i,j_{2}}} &\leq \frac{b_{0}(\cdot,N,\nu)}{x} + \frac{1}{2}c_{0}
\quad\text{and}\\
\frac{1}{2x^{2}}\sum_{i=1}^{d}\sum_{j_{1},j_{2}=1}^{m}g_{i,j_{1}}g_{i,j_{2}} (N-x)^{\zeta_{i,j_{1}} + \zeta_{i,j_{2}}}x^{\eta_{i,j_{1}} + \eta_{i,j_{2}}} &\leq \frac{\dot{N} - \sum_{i=0}^{k}b_{i}(\cdot,N, \nu)N^{i}}{x}  + \frac{1}{2}c_{1}
\end{align*}
for all $x\in ]0,1]$ with $x\leq N$ and $\nu\in\mathcal{P}_{0}(\R)$, where
\begin{equation*}
c_{0}:= \sum_{i=1}^{d}\sum_{\substack{j_{1},j_{2}=1,\\  \zeta_{i,j_{1}} + \zeta_{i,j_{2}} \geq 2}}^{m}(g_{i,j_{1}}g_{i,j_{2}})^{+}N^{\eta_{i,j_{1}} + \eta_{i,j_{2}}},\quad c_{1}:=\sum_{i=1}^{d}\sum_{\substack{j_{1},j_{2}=1,\\  \eta_{i,j_{1}} + \eta_{i,j_{2}} \geq 2}}^{m}(g_{i,j_{1}}g_{i,j_{2}})^{+}N^{\zeta_{i,j_{1}} + \zeta_{i,j_{2}}}.
\end{equation*}
Therefore,~\eqref{co:4} holds, as claimed.
\end{proof}

\begin{proof}[Proof of Proposition~\ref{pr:existence, uniqueness and a growth estimate}]
The drift $\overline{b}\colon\R_{+}\times\R\times\mathcal{P}_{1}(\R)\rightarrow\R$ and diffusion coefficient $\overline{\sigma}\colon\R_{+}\times\R\rightarrow\R^{1\times d}$ of the extended McKean--Vlasov SDE~\eqref{eq:extended McKean--Vlasov SDE}, given by~\eqref{eq:extended coefficients}, satisfy the regularity conditions~\eqref{eq:partial Lipschitz condition on the extended drift coefficient} and~\eqref{eq:Hoelder condition on compact sets on the extended diffusion coefficient}.

Consequently,~\cite[Theorem~3.1]{KalMeyPro24-2} yields a unique strong solution $\overline{I}$ to~\eqref{eq:extended McKean--Vlasov SDE} such that $\overline{I}_{0} = \xi_{0}$ a.s.~and $\sup_{t\in [0,T]}\E[|\overline{I}_{t}|] < \infty$ for all $T > 0$. If we can show that $0\leq \overline{I}\leq N$ a.s., then any continuous modification $I$ of $\overline{I}$ satisfying $0\leq I \leq N$ serves as strong solution to~\eqref{eq:SIS epidemic model}, and the first assertion follows in view of Proposition~\ref{pr:pathwise uniqueness}.

We recall from~\cite[Example~2.2]{KalMeyPro24-2} that the law map $\overline{\nu}\colon\R_{+}\rightarrow\mathcal{P}_{1}(\R)$ of $\overline{I}$ given by $\overline{\nu}(t) := \mathcal{L}(\overline{I}_{t})$ is Borel measurable and define $\overline{b}_{\overline{\nu}}\colon\R_{+}\times\R\rightarrow\R$ by $\overline{b}_{\overline{\nu}}(t,x) := \overline{b}(t,x,\overline{\nu}(t))$. Then $\overline{I}$ solves the SDE
\begin{equation}\label{eq:induced SDE}
\mathrm{d}I_{t} = \overline{b}_{\overline{\nu}}(t,I_{t})\,\mathrm{d}t + \overline{\sigma}(t,I_{t})\,\mathrm{d}W_{t}
\end{equation}
for $t\geq 0$. Since $\overline{b}(\cdot,x,\nu) = \overline{b}_{0}(\cdot,N,\nu)\geq 0$ and $\overline{\sigma}(\cdot,0) = 0$ for all $x\leq 0$ and $\nu\in\mathcal{P}_{1}(\R)$, the process vanishing identically solves~\eqref{eq:induced SDE} after subtracting $\overline{b}_{0}(\cdot,N,\overline{\nu})$ from the drift coefficient. Consequently, a comparison principle for solutions to SDEs based on the Yamada--Watanabe approach, as given in~\cite[Proposition 2.18]{KarShr91}, ensures that $\overline{I}\geq 0$ a.s.

Next, we recall that $N$ is locally absolutely continuous and $N(0) \geq \xi_{0}$. Hence, it follows readily that the process $N - \overline{I}$ solves the McKean--Vlasov SDE
\begin{equation}\label{eq:transformed McKean--Vlasov SDE}
\mathrm{d}X_{t} = \big(\dot{N}(t) - \overline{b}\big(t,N(t) - X_{t},\mathcal{L}(N(t) - X_{t})\big)\big)\,\mathrm{d}t - \overline{\sigma}(t,N(t) - X_{t})\,\mathrm{d}W_{t}
\end{equation}
for $t\geq 0$. Noting that $\overline{b}(\cdot,N - x,\nu) = \sum_{i=0}^{k}\overline{b}_{i}(\cdot,N,\nu)N^{i} \leq \dot{N}$ and $\overline{\sigma}(\cdot,N) = 0$ for all $x\leq 0$ and $\nu\in\mathcal{P}_{1}(\R)$, a comparison argument based on the Yamada--Watanabe approach shows that $N \geq \overline{I}$ a.s. Therefore, the existence and uniqueness assertions are established.

The estimate~\eqref{eq:first moment growth estimate} and~(i) are direct consequences of~\cite[Theorem~4.2]{KalMeyPro24-2}. In fact, $I$ is a random It{\^o} process with drift $\mathrm{B}\colon\R_{+}\times\Omega\rightarrow\R$ and diffusion $\Sigma\colon\R_{+}\times\Omega\rightarrow\R^{1\times d}$ given by
\begin{equation}\label{eq:random drift and diffusion}
\mathrm{B}_{s} := \sum_{i=0}^{k} b_{i}\big(s,N(s),\mathcal{L}(I_{s})\big)I_{s}^{i}\quad\text{and}\quad \Sigma_{s} :=  f(s,I_{s},N(s) - I_{s}).
\end{equation}
Since the partial affine growth condition~\eqref{eq:partial growth condition on the drift coefficient} for the drift coefficient of~\eqref{eq:SIS epidemic model} ensures that $\mathrm{B}_{s} \leq \hat{b}_{0}(s,N(s)) + \hat{b}_{k+2}(s,N(s))I_{s}$ for all $s\geq 0$, the hypotheses of~\cite[Theorem~4.2]{KalMeyPro24-2} are satisfied.

To verify~(ii), we note that $I$ is a solution to the SDE~\eqref{eq:extended McKean--Vlasov SDE}, since $\overline{\nu}(s) = \mathcal{L}(I_{s})$ for all $s\geq 0$. Hence, we may apply~\cite[Proposition~3.3]{BriGraKal24}. To this end, it follows from~\eqref{co:3} and~\eqref{co:4} that
\begin{equation}\label{eq:random drift and diffusion estimate 1}
\frac{1}{2}\frac{|\overline{\sigma}(s,x)|^{2}}{x^{2}} \leq \frac{\overline{b}(s,x,\overline{\nu}(s))}{x} + c(s)\varphi(x) +  c_{0}(s)
\end{equation}
for all $s\geq 0$ and $x\in ]0,\varepsilon\wedge 1[$ with $x\leq N(s)$. Here, the function $c_{0} := \sum_{i=1}^{k} b_{i}^{-}(\cdot,N,\overline{\nu})$ is measurable and locally integrable. The measurable map $\nu_{0}\colon\R_{+}\rightarrow\mathcal{P}_{0}(\R)$ defined by $\nu_{0}(t) := \mathcal{L}(N(t) - I_{t})$ satisfies
\begin{equation}\label{eq:random drift and diffusion estimate 2}
\frac{1}{2}\frac{|\overline{\sigma}(s,N(s) - x)|^{2}}{x^{2}}  \leq \frac{\dot{N}(s) - \overline{b}(s,N(s) - x,\nu_{0}(s))}{x} + c(s)\varphi(x) + c_{1}(s)
\end{equation}
for any $s\geq 0$ and $x\in ]0,\varepsilon\wedge 1[$ with $x\leq N(s)$, where $c_{1} := \sum_{i=1}^{k}b_{i}^{-}(\cdot,N,\nu_{0})iN^{i-1}$ is measurable and locally integrable. Thus, condition~(V.5) in~\cite{BriGraKal24} for the application of Proposition~3.3 is valid, and the verification is complete.
\end{proof}

\begin{proof}[Proof of Proposition~\ref{pr:pathwise uniqueness 2}]
By the same reasoning as in the proof of Proposition~\ref{pr:pathwise uniqueness}, it suffices to verify pathwise uniqueness for~\eqref{eq:extended McKean--Vlasov SDE} relative to $\Theta:\R_{+}\times\mathcal{P}_{p}(\R)\times\mathcal{P}_{p}(\R)\rightarrow\R_{+}$ given by $\Theta(s,\nu,\tilde{\nu}) := \hat{\lambda}_{k+1}(s,N(s))\mathcal{W}_{p}(\nu,\tilde{\nu})^{2}$.

In this context, we recall that the drift $\overline{b}\colon\R_{+}\times\R\times\mathcal{P}_{p}(\R)\rightarrow\R$ and diffusion coefficient $\overline{\sigma}\colon\R_{+}\times\R\rightarrow\R^{1\times d}$ of~\eqref{eq:extended McKean--Vlasov SDE} are given by~\eqref{eq:extended coefficients}. Then from~\eqref{eq:domination condition} we obtain~\eqref{eq:partial Lipschitz condition on the extended drift coefficient} for all $x,\tilde{x}\in\R$ and $\nu,\tilde{\nu}\in\mathcal{P}_{p}(\R)$, which is a partial Lipschitz condition for $\overline{b}$. The Lipschitz condition~\eqref{eq:Lipschitz condition on the diffusion coefficient} for the diffusion coefficient of~\eqref{eq:SIS epidemic model} extends to
\begin{equation}\label{eq:Lipschitz condition on the extended diffusion coefficient}
|\overline{\sigma}(\cdot,x) - \overline{\sigma}(\cdot,\tilde{x})| \leq \lambda(\cdot,N)|x-\tilde{x}|
\end{equation}
for all $x,\tilde{x}\in\R$. Hence,~\cite[Corollary~3.5]{KalMeyPro24} implies pathwise uniqueness for~\eqref{eq:extended McKean--Vlasov SDE} relative to $\Theta$, as the required condition~(C.2) therein is satisfied.

Next, the random It{\^o} process $Y := I - \tilde{I}$ has random drift $\mathrm{B}$ and diffusion $\Sigma$ given by~\eqref{eq:random drift and diffusion coefficients} satisfying $Y_{s}\mathrm{B}_{s} \leq |Y_{s}|(\hat{b}_{k+1}(s,N(s))|Y_{s}| +  \hat{\lambda}_{k+1}(s,N(s))\vartheta(\mathcal{L}(I_{s}),\mathcal{L}(\tilde{I}_{s})))$ and $|\Sigma_{s}|$ $\leq \lambda(s,N(s))|Y_{s}|$ for any $s\geq 0$. Consequently, from~\cite[Theorem~4.6]{KalMeyPro24} the remaining claims follow.
\end{proof}

\begin{proof}[Proof of Proposition~\ref{pr:existence, uniqueness and a growth estimate 2}]
The drift and diffusion coefficients of~\eqref{eq:extended McKean--Vlasov SDE}, defined by~\eqref{eq:extended coefficients}, satisfy~\eqref{eq:partial Lipschitz condition on the extended drift coefficient} and~\eqref{eq:Lipschitz condition on the extended diffusion coefficient} for all $x,\tilde{x}\in\R$ and $\nu,\tilde{\nu}\in\mathcal{P}_{p}(\R)$. Thus,~\cite[Theorem~3.24]{KalMeyPro24} ensures the existence of a unique strong solution $\overline{I}$ to~\eqref{eq:extended McKean--Vlasov SDE} satisfying $\overline{I}_{0} = \xi_{0}$ a.s.~and $\sup_{t\in [0,T]}\E[|\overline{I}_{t}|^{p}] < \infty$ for all $T > 0$.

As in the proof of Proposition~\ref{pr:existence, uniqueness and a growth estimate}, we note that the map $\overline{\nu}\colon\R_{+}\rightarrow\mathcal{P}_{p}(\R)$ defined by $\overline{\nu}(t) := \mathcal{L}(\overline{I}_{t})$ is Borel measurable, and $\overline{I}$ solves the SDE~\eqref{eq:induced SDE}, where $\overline{b}_{\overline{\nu}}\colon\R_{+}\times\R\rightarrow\R$ is given by $\overline{b}_{\overline{\nu}}(t,x) := \overline{b}(t,x,\overline{\nu}(t))$.

As $\overline{b}(\cdot,x,\nu) \geq 0$, $\dot{N} - \overline{b}(\cdot,N - x,\nu) \geq 0$ and $\overline{\sigma}(\cdot,0) = \overline{\sigma}(\cdot,N) = 0$ for all $x\leq 0$ and $\nu\in\mathcal{P}_{p}(\R)$, a comparison principle for solutions to SDEs yields that $0\leq \overline{I} \leq N$ a.s. See~\cite[Proposition 2.18]{KarShr91}, for instance. Thus, any continuous modification $I$ of $\overline{I}$ satisfying $0\leq I \leq N$ is a strong solution to~\eqref{eq:SIS epidemic model}. By Proposition~\ref{pr:pathwise uniqueness 2}, this verifies the existence and uniqueness claims.

Now the estimate~\eqref{eq:pth moment growth estimate} and~(i) are implied by~\cite[Theorem 4.6]{KalMeyPro24}. Indeed, $I$ is a random It{\^o} process with random drift $\mathrm{B}$ and diffusion $\Sigma$ given by~\eqref{eq:random drift and diffusion}. Further, $\mathrm{B}_{s}$ $\leq \hat{b}_{0}(s, N(s))  + \hat{b}_{k+2}(s, N(s))I_{s}$ and $|\Sigma_{s}| \leq l(s, N(s))I_{s}$ for all $s\geq 0$, by the growth conditions~\eqref{eq:partial growth condition on the drift coefficient} and~\eqref{eq:linear growth condition on the diffusion coefficient}. So, the requirements of~\cite[Theorem 4.6]{KalMeyPro24} are satisfied.

The last assertion~(ii) follows as in the proof of Proposition~\ref{pr:existence, uniqueness and a growth estimate} from~\cite[Proposition~3.3]{BriGraKal24}, since~\eqref{co:3} and~\eqref{co:4} entail the inequalities~\eqref{eq:random drift and diffusion estimate 1}~and~\eqref{eq:random drift and diffusion estimate 2}.
\end{proof}

\subsection{Proofs for the pathwise asymptotic behaviour}

\begin{proof}[Proof of Lemma~\ref{le:pathwise asymptotic identities}]
Since $I$ takes positive values only, It{\^o}'s formula yields that
\begin{equation*}
\log(I) = \log(I_{0}) + \int_{0}^{\cdot}h\big(s,I_{s},N(s),\mathcal{L}(I_{s})\big)\,\mathrm{d}s + M\quad\text{a.s.}
\end{equation*}
for the continuous $\mathbb{F}$-local martingale $M := \int_{0}^{\cdot}f(s,I_{s},N(s)-I_{s})I_{s}^{-1}\,\mathrm{d}W_{s}$. By~\eqref{co:6} and monotone convergence, the quadratic variation of $\int_{0}^{\cdot}\frac{1}{1 + s^{\rho}}\,\mathrm{d}M_{s}$ satisfies
\begin{equation*}
\lim_{t\uparrow\infty}\bigg\langle\int_{0}^{\cdot}\frac{1}{1 + s^{\rho}}\,\mathrm{d}M_{s}\bigg\rangle_{t} \leq \int_{0}^{\infty} \frac{l(s,N(s))^{2}}{(1 + s^{\rho})^{2}}\,\mathrm{d}s < \infty\quad\text{a.s.}
\end{equation*}
Hence, the strong law of large numbers for continuous local martingales in~\cite[Theorem~1]{Lip80} implies that
\begin{equation*}
\lim_{t\uparrow\infty}\frac{M_{t}}{t^{\rho}} = 0\quad\text{a.s.}
\end{equation*}
These considerations show the two asserted identities.
\end{proof}

\begin{proof}[Proof of Proposition~\ref{pr:extinction}]
For any $\varepsilon > 0$ and $\omega\in\Omega$ there is some $t_{\varepsilon,\omega} \geq 0$ such that $h(s,I_{s}(\omega),N(s),\mathcal{L}(I_{s}))$ $< u(s) + \varepsilon s^{\rho-1}$ for all $s > t_{\varepsilon,\omega}$. Hence,
\begin{equation*}
\limsup_{t\uparrow\infty}\frac{1}{t^{\rho}}\int_{0}^{t}h\big(s,I_{s}(\omega),N(s),\mathcal{L}(I_{s})\big)\,\mathrm{d}s \leq \limsup_{t\uparrow\infty} \frac{1}{t^{\rho}}\int_{0}^{t} u(s)\,\mathrm{d}s + \frac{\varepsilon}{\rho},
\end{equation*}
and the first assertion follows from Lemma~\ref{le:pathwise asymptotic identities}. Next, since $\lim_{t\uparrow\infty} I_{t} = 0$ a.s.~and $N$ is bounded, we have $\lim_{t\uparrow\infty}\vartheta(\mathcal{L}(I_{t}),\delta_{0}) = 0$, by dominated convergence. For this reason,
\begin{equation*}
\lim_{s\uparrow\infty} h\big(s,I_{s}(\omega),N(s),\mathcal{L}(I_{s})\big) = h_{\infty}
\end{equation*}
and hence, $\lim_{t\uparrow\infty} \frac{1}{t}\int_{0}^{t}|h(s,I_{s}(\omega),N(s),\mathcal{L}(I_{s})) - h_{\infty}|\,\mathrm{d}s = 0$ for each $\omega\in\Omega$ satisfying $\lim_{t\uparrow\infty} I_{t}(\omega) = 0$. This shows the second claim, by another application of Lemma~\ref{le:pathwise asymptotic identities}.
\end{proof}

\begin{proof}[Proof of Proposition~\ref{pr:extinction 1}]
As~\eqref{a.2} is equivalent to $h_{3}(\cdot,N) < 0$, we infer from Lemma~\ref{le:maximisation 1} that
\begin{equation*}
\max_{x\in [0,N]} h(\cdot,x,N,\nu) \leq
\begin{cases}
h_{1}(\cdot,N,\nu) - \frac{1}{4}\frac{(h_{2}(\cdot,N,\nu_{0})^{+})^{2}}{h_{3}(\cdot,N)} & \text{if $h_{2}(\cdot,N,\nu_{0}) < - 2h_{3}(\cdot,N)N$}\\
-(\mu + \gamma) & \text{otherwise}
\end{cases}
\end{equation*}
for all $\nu\in\mathcal{P}_{0}(\R)$ with $\nu_{0} := \delta_{0}\mathbbm{1}_{\R_{+}}(\beta_{1}) + \delta_{N}\mathbbm{1}_{]-\infty,0[}(\beta_{1})$, since $h_{2}(\cdot,N,\nu) \leq h_{2}(\cdot,N,\nu_{0})$ and $h(\cdot,N,N,\nu) = -(\mu + \gamma)$. Hence, from
\begin{align*}
\limsup_{t\uparrow \infty} h_{1}\big(t,N(t),\mathcal{L}(I_{t})\big) &\leq - (\mu_{\infty} + \gamma_{\infty}) - u_{I} + u_{1}N_{\infty} + \bigg(c_{1,2,\infty} - \frac{1}{2}u_{3}\bigg)N_{\infty}^{2},\\
\lim_{t\uparrow \infty} h_{2}\big(t,N(t),\nu_{0}\big) &= u_{2} + \big(c_{2,1,\infty} + u_{3}\big)N_{\infty} + c_{2,2,\infty}N_{\infty}^{2}
\end{align*}
and $\lim_{t\uparrow \infty}h_{3}(t,N(t)) = -u_{4}$ we conclude that $\limsup_{t\uparrow\infty}h(t, I_{t}, N(t),\mathcal{L}(I_{t})) \leq u$. For this reason, Proposition~\ref{pr:extinction} yields the claim.
\end{proof}

\begin{proof}[Proof of Proposition~\ref{pr:extinction 2}]
Since $h(\cdot,N,N,\mathcal{L}(I_{\cdot})) = -(\mu + \gamma)$, Lemmas~\ref{le:maximisation 1} and~\ref{le:maximisation 2} and Proposition~\ref{pr:maximisation 3} give us that
\begin{equation}\label{eq:extinction 2}
\max_{x\in [0,N]} h(\cdot,x,N,\mathcal{L}(I_{\cdot})) = \max\big\{h_{1}(\cdot,N,\mathcal{L}(I_{\cdot})) - g_{1,1}g_{1,2}N^{\frac{3}{2}},-(\mu + \gamma)\big\}.
\end{equation}
Indeed, if $c_{1,2}(t) = c_{2,1}(t) = c_{2,2}(t) = g_{1,1}(t) = 0$ for some $t\geq 0$, then $h_{3}(t,y) = 0$ for all $y\geq 0$ and~\eqref{eq:extinction 2} follows directly from Lemma~\ref{le:maximisation 1}. Hence,
\begin{equation*}
\limsup_{t\uparrow\infty}\max_{x\in [0,N(t)]} h\big(t,x,N,\mathcal{L}(I_{t})\big) \leq u
\end{equation*}
for $u := -(\mu_{\infty}+\gamma_{\infty}) + (u_{1}N_{\infty} - u_{2}N_{\infty}^{\frac{3}{2}} + u_{3}N_{\infty}^{2})^{+}$. Since $u < 0$ is equivalent to $\mu_{\infty} + \gamma_{\infty} > 0$ and~\eqref{eq:extinction in the epidemic models 10}, the claim follows from Proposition~\ref{pr:extinction}.
\end{proof}

\begin{proof}[Proof of Lemma~\ref{le:extinction 3}]
Because $h_{3}(\cdot,N)\geq 0$, Lemma~\ref{le:maximisation 1} yields that $\max_{x\in [0,N]} h(\cdot,x,N,\nu)$ $= \max\{h_{1}(\cdot,N,\nu),-(\mu + \gamma)\}$ for any $\nu\in\mathcal{P}_{0}(\R)$. Hence,
\begin{equation*}
\limsup_{t\uparrow\infty}h_{1}\big(t,N(t),\mathcal{L}(I_{t})\big) \leq u
\end{equation*}
for $u := -(\mu_{\infty} + \gamma_{\infty} + u_{I}) + (u_{1} + u_{2}N_{\infty})N_{\infty}$, by~\eqref{eq:extinction in the epidemic models 4} and~\eqref{eq:extinction in the epidemic models 12}. Since $u_{I}\geq 0$, the claim is implied by Proposition~\ref{pr:extinction}.
\end{proof}

\begin{proof}[Proof of Proposition~\ref{pr:persistence}]
(i) By way of contradiction, suppose that $\{\limsup_{t\uparrow\infty} I_{t} < x_{0}\}$ fails to be null. Then the event $A:=\{\limsup_{t\uparrow\infty} I_{t} < x_{0} - \delta\}$ has positive probability for some $\delta\in ]0,x_{0}[$, by the $\sigma$-continuity of probability measures.

Furthermore, for any $\omega\in A$ and $\varepsilon > 0$ there is $s_{\omega,\varepsilon}\geq 0$ such that $I_{s}(\omega) < x_{0} - \delta$ and $h(s,x,N(s),\mathcal{L}(I_{s})) \geq w_{\delta}(s) - \varepsilon s^{\rho-1}$ for all $s > s_{\omega,\varepsilon}$ and $x\in ]0,N(s)]$ with $x < x_{0} - \delta$. This leads to
\begin{equation*}
\limsup_{t\uparrow\infty}\frac{1}{t^{\rho}}\int_{0}^{t}h\big(s,I_{s}(\omega),N(s),\mathcal{L}(I_{s})\big)\,\mathrm{d}s \geq \limsup_{t\uparrow\infty} \frac{1}{t^{\rho}}\int_{0}^{t}w_{\delta}(s)\,\mathrm{d}s - \frac{\varepsilon}{\rho}.
\end{equation*}
However, this implies the contradiction that $\limsup_{t\uparrow\infty} I_{t}(\omega) = \infty$ for a.e.~$\omega\in A$, according to Lemma~\ref{le:pathwise asymptotic identities}.

(ii) Assume that the almost sure inequality fails. Then there is $\delta > 0$ such that $A:=\{\liminf_{t\uparrow\infty} I_{t} > x_{0} + \delta\}$ has positive probability. Further, for any $\omega\in A$ and $\varepsilon > 0$ we may take $s_{\omega,\varepsilon}\geq 0$ such that $I_{s}(\omega) > x_{0} + \delta$ and $h(s,x,N(s),\mathcal{L}(I_{s})) < w_{\delta}(s) + \varepsilon s^{\rho-1}$ for any $s > s_{\omega,\varepsilon}$ and $x\in ]0,N(s)]$ with $x > x_{0} + \delta$. Hence,
\begin{equation*}
\liminf_{t\uparrow\infty}\frac{1}{t^{\rho}}\int_{0}^{t}h\big(s,I_{s}(\omega),N(s),\mathcal{L}(I_{s})\big)\,\mathrm{d}s \leq \liminf_{t\uparrow\infty} \frac{1}{t^{\rho}}\int_{0}^{t}w_{\delta}(s)\,\mathrm{d}s + \frac{\varepsilon}{\rho}.
\end{equation*}
However, for Lemma~\ref{le:pathwise asymptotic identities} to be valid, we must have $\liminf_{t\uparrow\infty} I_{t}(\omega) = 0$ for a.e.~$\omega\in A$, which is impossible.
\end{proof}

\begin{proof}[Proof of Corollary~\ref{co:persistence}]
We seek to infer both claims from Proposition~\ref{pr:persistence}. (i) Since $\liminf_{x\downarrow 0} f_{\infty}(x) > 0$, there are $c_{0} > 0$ and $\delta_{0}\in ]0,N_{\infty}[$ such that $f_{\infty}(x) \geq c_{0}$ for all $x\in ]0,\delta_{0}[$. By continuity, this ensures that $f_{\infty}$ has indeed a smallest zero and it takes the form $x_{0} = \min\{x\in ]0,N_{\infty}]\cap\R_{+}\mid f_{\infty}(x) = 0\}$.

Next, for $\varepsilon > 0$ we may choose $s_{\varepsilon}\geq 0$ such that $h(s,x,N(s),\mathcal{L}(I_{s})) \geq f_{\infty}(x) - \varepsilon$ for all $s > s_{\varepsilon}$ and $x\in ]0,N(s)]$ with $x < x_{0}$. Then for each $\delta\in ]0,x_{0}[$ the positive number $w_{\delta} := \inf_{x\in ]0,x_{0}-\delta[}f_{\infty}(x)$ satisfies $h(s,x,N(s),\mathcal{L}(I_{s})) \geq w_{\delta} - \varepsilon$ for any $s > s_{\varepsilon}$ and $x\in ]0,N(s)]$ with $x < x_{0} - \delta$. Hence, Proposition~\ref{pr:persistence} is applicable.

(ii) By hypothesis, there are $c_{0},\delta_{0} > 0$ and $c_{1} \in ]0,N_{\infty}[$ such that $g_{\infty}(x) \leq - c_{0}$ for all $x\in ]c_{1},N_{\infty} + \delta_{0}[$. As $g_{\infty}$ is continuous, this guarantees that $g_{\infty}$ has in fact a largest zero in $[0,N_{\infty}[$ and $y_{0} = \max\{x\in [0,N_{\infty}[\,\mid g_{\infty}(x) = 0\}$.

Then for any $\varepsilon > 0$ there is $s_{\varepsilon}\geq 0$ such that $N(s) < N_{\infty} + \delta_{0}$ and $h(s,x,N(s),\mathcal{L}(I_{s}))$ $\leq g_{\infty}(x) + \varepsilon$ for all $s > s_{\varepsilon}$ and $x\in ]0,N(s)]$ with $x > y_{0}$. So, for each $\delta > 0$ the number $w_{\delta}:= \sup_{x\in ]y_{0} +\delta\wedge\delta_{0},N_{\infty} + \delta_{0}[} g_{\infty}(x)$ is negative and $h(s,x,N(s),\mathcal{L}(I_{s})) \leq w_{\delta} + \varepsilon$ for any $s > s_{\varepsilon}$ and $x\in ]0,N(s)]$ with $x > y_{0} + \delta$, as required.
\end{proof}

\begin{proof}[Proof of Corollary~\ref{co:persistence 2}]
We show that both assertions are implied by Corollary~\ref{co:persistence}. (i) In view of Example~\ref{ex:transformed function in the representative model}, for each $y\in ]0,N_{\infty}]$ we may infer from~\eqref{eq:specific limits} and the definition of $a_{\infty},b_{\infty},c_{\infty},d_{\infty}$ that
\begin{equation*}
\liminf_{s\uparrow\infty}\inf_{x\in ]0,N(s)]:\, x < y} h\big(s,x,N(s),\mathcal{L}(I_{s})\big) - f_{\infty}(x)\geq 0
\end{equation*}
and $f_{\infty}(N_{\infty}) \leq \liminf_{s\uparrow \infty} h(s,N(s),N(s),\mathcal{L}(I_{s})) = -(\mu_{\infty} + \gamma_{\infty}) \leq 0$. By the intermediate value theorem, the conditions in~(i) of Corollary~\ref{co:persistence} are satisfied, and the representation of $x_{0}$ in the case $d_{\infty} = 0$  follows from Lemma~\ref{le:unique zero of a sum of power functions} and Remark~\ref{re:unique zero of a sum of power functions}.

(ii) By assumption, $g_{\infty}(0) \geq 0$ and $\lim_{x\uparrow N_{\infty}} g_{\infty}(x) < 0$. This ensures that $g_{\infty}$ has at least one zero, which in the case $d_{\infty} = 0$ is unique and of the asserted form, according to Lemma~\ref{le:unique zero of a sum of power functions}.

Furthermore, $\limsup_{s\uparrow\infty}\sup_{x\in ]0,N(s)]:\, x > y} h(s,x,N(s),\mathcal{L}(I_{s})) - g_{\infty}(x)\leq 0$ for every $y\in [0,N_{\infty}[$, due to~\eqref{eq:specific limits} and the definition of $\hat{a}_{\infty},\hat{b}_{\infty},\hat{c}_{\infty},\hat{d}_{\infty}$. Hence, the requirements in~(ii) of Corollary~\ref{co:persistence} are met, and the proof is complete.
\end{proof}

\subsection{Proofs for the Euler--Maruyama scheme}

\begin{proof}[Proof of Proposition~\ref{pr:EM-estimate}]
According to the growth conditions~\eqref{eq:affine growth condition on the drift coefficient} and~\eqref{eq:linear growth condition on the diffusion coefficient} on the drift and diffusion coefficients of~\eqref{eq:SIS epidemic model}, the $(\mathcal{F}_{s})_{s\in [t_{j,n},t_{j+1,n}]}$-progressively measurable process $Z^{(j,n)}\colon [t_{j,n},t_{j+1,n}]\times\Omega\rightarrow\R$ given by
\begin{align*}
Z_{s}^{(j,n)} &:=\hat{I}_{s}^{(n)}\sum_{i=0}^{k}\overline{b}_{i}(t_{j,n},N(t_{j,n}),L_{j,n})\big((I_{t_{j,n}}^{(n)})^{+}\wedge N(t_{j,n})\big)^{i}\\
&\quad + c_{p}\big|\overline{f}\big(t_{j,n},I_{t_{j,n}}^{(n)},N(t_{j,n}) - I_{t_{j,n}}^{(n)}\big)\big|^{2}
\end{align*}
satisfies
\begin{equation*}
Z_{s}^{(j,n)} \leq |\hat{I}_{s}^{(n)}|\big(\hat{b}_{0}(t_{j,n},N(t_{j,n})) + \hat{b}_{k+3}(t_{j,n},N(t_{j,n}))|I_{t_{j,n}}^{(n)}|\big) + c_{p}l(t_{j,n},N(t_{j,n}))^{2}\big(I_{t_{j,n}}^{(n)}\big)^{2}
\end{equation*}
for fixed $n\in\N$ and $j\in\{0,\dots,k_{n}-1\}$ and any $s\in [t_{j_{n}},t_{j+1,n}]$. Now we proceed as in the proof of Theorem~4.6 in~\cite{KalMeyPro24}. Namely, from the inequalities of Hölder and Young we infer that
\begin{equation*}
p\E\big[|\hat{I}_{s}^{(n)}|^{p-2}Z_{s}^{(j,n)}\mathbbm{1}_{\{\tau > s\}}\big] \leq \hat{b}_{0}(t_{j,n},N(t_{j,n})) + l_{p,j,n,2}\E\big[|I_{t_{j,n}}^{(n)}|^{p}\big] + l_{p,j,n,1}\E\big[|\hat{I}_{s}^{(n)}|^{p}\mathbbm{1}_{\{\tau > s\}}\big]
\end{equation*}
for all $s\in [t_{j,n},t_{j+1,n}]$ and any $(\mathcal{F}_{s})_{s\in [t_{j,n},t_{j+1,n}]}$-stopping time $\tau$ with the two constants
\begin{equation*}
l_{p,j,n,1} := \big((p-1)(\hat{b}_{0} + \hat{b}_{k+3}) + (p-2)c_{p}l^{2}\big)(t_{j,n},N(t_{j,n}))
\end{equation*}
and $l_{p,j,n,2} := (\hat{b}_{k+3} + 2c_{p}l^{2})(t_{j,n},N(t_{j,n}))$. Hence, it follows inductively from Lemma 4.2 in~\cite{KalMeyPro24} and Fatou's lemma that $\hat{I}^{(n)}$ is $p$-fold integrable and 
\begin{align*}
\E\big[|\hat{I}_{t}^{(n)}|^{p}\big] &\leq  e^{l_{p,j,n,1}(t-t_{j,n})}\E\big[|I_{t_{j,n}}^{(n)}|^{p}\big]\\
&\quad + \big(\hat{b}_{0}(t_{j,n},N(t_{j,n})) + l_{p,j,n,2}\E\big[|I_{t_{j,n}}^{(n)}|^{p}\big]\big)\int_{t_{j,n}}^{t}e^{l_{p,j,n,1}(t-s)}\,\mathrm{d}s\\
&\leq e^{(l_{p,j,n,1} + l_{p,j,n,2})(t - t_{j,n})}\E\big[|I_{t_{j,n}}^{(n)}|^{p}\big] + e^{l_{p,j,n,1}(t-t_{j,n})}(t-t_{j,n})\hat{b}_{0}(t_{j,n},N(t_{j,n}))
\end{align*}
for all $t\in [t_{j,n},t_{j+1,n}]$. Consequently, we infer the claimed estimate inductively by utilising that $l_{p,i,n,1} + l_{p,i,n,2} \leq k_{p,j,n}$ and $2l_{p,i,n,1} + l_{p,i,n,2} \leq l_{p,j,n}$ for all $i,j\in\{0,\dots,k_{n}-1\}$ with $i\leq j$.
\end{proof}

\begin{proof}[Proof of Lemma~\ref{le:EM-estimate}]
From the triangle inequality in the $L^{p}$-norm, the Cauchy--Schwarz inequality and the two growth conditions~\eqref{eq:affine growth condition on the drift coefficient} and~\eqref{eq:linear growth condition on the diffusion coefficient} we obtain that
\begin{align*}
\E\big[|\hat{I}_{t}^{(n)} - I_{t_{j,n}}^{(n)}\big|^{p}\big]^{\frac{1}{p}} &\leq \big(\hat{b}_{0}(t_{j,n},N(t_{j,n})) + \hat{b}_{k+3}(t_{j,n},N(t_{j,n}))N(t_{j,n})\big)(t-t_{j,n})\\
&\quad  + l(t_{j,n},N(t_{j,n}))N(t_{j,n})\E\big[|W_{t} - W_{t_{j,n}}|^{p}\big]^{\frac{1}{p}}.
\end{align*}
This verifies the desired estimate.
\end{proof}

\begin{proof}[Proof of Proposition~\ref{pr:intermediate error estimate}]
We may assume that the term $\E[|\hat{I}^{(n)} - I|^{p-1}\mathcal{W}_{p}(\mathcal{L}(I),L_{j,n})]$ is integrable over $[t_{j,n},t_{j+1,n}]$, as otherwise the asserted estimate, which we seek to infer from~\cite[Lemma~4.2]{KalMeyPro24}, is infinite.

First of all, Proposition~\ref{pr:EM-estimate} shows that the random It{\^o} process $\hat{I}^{(n)} - I$ and the $(\mathcal{F}_{s})_{s\in [t_{j,n},t_{j+1,n}]}$-progressively measurable process $Z^{(j,n)}\colon [t_{j,n},t_{j+1,n}]\times\Omega\rightarrow\R$ given by
\begin{align*}
Z_{s}^{(j,n)} &:= (\hat{I}_{s}^{(n)} - I_{s})\sum_{i=0}^{k}\overline{b}_{i}(t_{j,n},N(t_{j,n}),L_{j,n})\big((I_{t_{j,n}}^{(n)})^{+}\wedge N(t_{j,n})\big)^{i} - b_{i}\big(s,N(s),\mathcal{L}(I_{s})\big)I_{s}^{i}\\
&\quad + c_{p}\big|\overline{f}\big(t_{j,n},I_{t_{j,n}}^{(n)},N(t_{j,n}) - I_{t_{j,n}}^{(n)}\big) - f(s,I_{s}, N(s) - I_{s})\big|^{2}
\end{align*}
are $p$-fold integrable for fixed $n\in\N$ and $j\in\{0,\dots,k_{n}-1\}$. Further, the two continuity conditions~\eqref{eq:continuity condition on the drift coefficient} and~\eqref{eq:Hoelder condition on the diffusion coefficient} lead to the following respective inequalities:
\begin{align*}
\bigg|\sum_{i=0}^{k}b_{i}(s,N(s),\nu)x^{i} &- \overline{b}_{i}(\tilde{s},N(\tilde{s}),\tilde{\nu})(\tilde{x}^{+}\wedge N(\tilde{s})\big)^{i}\bigg| \leq  \overline{b}_{k+1}(N(\tilde{s}))|s-\tilde{s}|^{\frac{1}{2}}\\
&\quad + \big(\overline{b}_{k+1}(N(\tilde{s})) + \overline{b}_{k+2}(s,N(s),N(\tilde{s}))\big)|N(s) - N(\tilde{s})|\\
&\quad + \hat{b}_{k+4}(s,N(s))|x - \tilde{x}| + \hat{\lambda}_{k+1}(s,N(s))\mathcal{W}_{p}(\nu,\tilde{\nu})
\end{align*}
and
\begin{align*}
|f(s,x,N(s) - &x) - \overline{f}(\tilde{s},\tilde{x},N(\tilde{s}) - \tilde{x})| \leq \lambda_{0}(N(\tilde{s}))|s-\tilde{s}|^{\frac{1}{2}}\\
&\quad + \lambda(s,N(s)\vee N(\tilde{s}))\bigg(\frac{1}{2}|N(s) - N(\tilde{s})| + |x-\tilde{x}|\bigg)
\end{align*}
for any $s,\tilde{s}\geq 0$, $x\in [0,N(s)]$, $\tilde{x}\in\R$, $\nu\in\mathcal{P}_{0}(\R)$ and $\tilde{\nu}\in\mathcal{P}_{p}(\R)$, since we have $f(\cdot,x,0)$ $= f(\cdot,0,y) = 0$ for all $x,y\geq 0$. Thus, from Young's inequality we derive that
\begin{align*}
p\E\big[|\hat{I}_{s}^{(n)} &- I_{s}\big|^{p-2}Z_{s}^{(j,n)}\big] \leq \hat{\lambda}_{k+1}(s,N(s))p\E\big[|\hat{I}_{s}^{(n)} - I_{s}|^{p-1}\mathcal{W}_{p}(\mathcal{L}(I_{s}),L_{j,n})\big]\\
&\quad + p\E\big[|\hat{I}_{s}^{(n)} - I_{s}\big|^{p-1}\big(\zeta_{j,n}(s) + \hat{b}_{k+4}(s,N(s))|I_{t_{j,n}}^{(n)} - I_{s}|\big)\big]\\
&\quad + 3pc_{p}\E\big[|\hat{I}_{s}^{(n)} - I_{s}\big|^{p-2}\big(\eta_{j,n}(s) + \lambda(s,N(s)\vee N(t_{j,n}))^{2}|I_{t_{j,n}}^{(n)} - I_{s}|^{2}\big)\big]\\
&\leq \delta_{p,j,n}(s) + \lambda_{p,j,n}(s)\E\big[|\hat{I}_{s}^{(n)} - I_{s}\big|^{p}\big]
\end{align*}
for all $s\in [t_{j,n},t_{j+1,n}]$, where the two measurable functions $\zeta_{j,n},\eta_{j,n}\colon[t_{j,n},t_{j+1,n}]\rightarrow\R_{+}$ are defined by
\begin{align*}
\zeta_{j,n}(s) &:= \overline{b}_{k+1}(N(t_{j,n}))|s - t_{j,n}|^{\frac{1}{2}}\\
&\quad + \big(\overline{b}_{k+1}(N(t_{j,n})) + \overline{b}_{k+2}(s,N(s),N(t_{j,n}))\big)|N(s) - N(t_{j,n})|,\\
\eta_{j,n}(s) &:= \lambda_{0}(N(t_{j,n}))^{2}|s-t_{j,n}| + \lambda(s,N(s)\vee N(t_{j,n}))^{2}\frac{1}{4}|N(s) - N(t_{j,n})|^{2}.
\end{align*}
In view of Lemma~\ref{le:EM-estimate}, the asserted estimate follows from~\cite[Lemma~4.2]{KalMeyPro24}.
\end{proof}

\begin{proof}[Proof of Theorem~\ref{thm:strong error estimate}]
As the first claimed estimate implies the second, it suffices to prove the first inequality. For notational convenience we shall use the representation~\eqref{eq:empirical measure of interacting particles} of $L_{j,n}$ for any fixed $n\in\N$ and $j\in\{0,\dots,k_{n}-1\}$. Then Proposition~\ref{pr:intermediate error estimate} yields that
\begin{equation}\label{eq:intermediate error estimate}
\begin{split}
&u(t)\E\big[|\hat{I}_{t}^{(n,l)} - I_{t}^{(\ell)}|^{p}\big] \leq 
u(t_{j,n})\E\big[|I_{t_{j,n}}^{(n,l)} - I_{t_{j,n}}^{(\ell)}|^{p}\big]\\
&\quad  + \int_{t_{j,n}}^{t}u(s)\big(\delta_{p,j,n}(s) +\hat{\lambda}_{k+1}(s,N(s))p\E\big[|\hat{I}_{s}^{(n,l)} - I_{s}^{(\ell)}|^{p-1}\mathcal{W}_{p}(\mathcal{L}(I_{s}^{(\ell)}),L_{j,n})\big]\big)\,\mathrm{d}s\\
&\quad + \int_{t_{j,n}}^{t}\big(\dot{u}(s) + u(s)\lambda_{p,j,n}(s)\big)\E\big[|\hat{I}_{s}^{(n,l)} - I_{s}^{(\ell)}|^{p}\big]\,\mathrm{d}s
\end{split}
\end{equation}
for fixed $\ell\in\{1,\dots,M_{n}\}$, all $t\in [t_{j,n},t_{j+1,n}]$ and each absolutely continuous function $u\colon [t_{j,n},t_{j+1,n}]\rightarrow\R_{+}$. The triangle inequality in $\mathcal{P}_{p}(\R)$ entails that
\begin{align*}
\mathcal{W}_{p}\big(\mathcal{L}(I_{s}^{(\ell)}),L_{j,n}\big) &\leq \mathcal{W}_{p}\bigg(\mathcal{L}(I_{s}^{(1)}),\frac{1}{M_{n}}\sum_{m=1}^{M_{n}}\delta_{I_{s}^{(m)}}\bigg)\\
&\quad + \bigg(\frac{1}{M_{n}}\sum_{m=1}^{M_{n}}|I_{s}^{(m)} - \hat{I}_{s}^{(n,m)}|^{p}\bigg)^{\frac{1}{p}} + \bigg(\frac{1}{M_{n}}\sum_{m=1}^{M_{n}}|\hat{I}_{s}^{(n,m)} - I_{t_{j,n}}^{(n,m)}|^{p}\bigg)^{\frac{1}{p}}
\end{align*}
for any $s\in [t_{j,n},t_{j+1,n}]$, since $\mathcal{W}_{p}(\frac{1}{m}\sum_{i=1}^{m}\delta_{x_{i}},\frac{1}{m}\sum_{i=1}^{m}\delta_{y_{i}})^{p} \leq \frac{1}{m}\sum_{i=1}^{m}|x_{i} - y_{i}|^{p}$ for all $m\in\N$ and $x,y\in\R^{m}$. Consequently, Young's inequality yields that
\begin{align*}
&p\E\big[|\hat{I}_{s}^{(n,l)} - I_{s}^{(\ell)}|^{p-1}\mathcal{W}_{p}(\mathcal{L}(I_{s}^{(\ell)}),L_{j,n})\big]\\
&\leq (3p-2)\E\big[|\hat{I}_{s}^{(n,l)} - I_{s}^{(\ell)}|^{p}\big] + \E\bigg[\mathcal{W}_{p}\bigg(\mathcal{L}(I_{s}^{(1)}),\frac{1}{M_{n}}\sum_{m=1}^{M_{n}}\delta_{I_{s}^{(m)}}\bigg)^{p}\bigg]  + \E\big[|\hat{I}_{s}^{(n,\ell)} - I_{t_{j,n}}^{(n,\ell)}|^{p}\big],
\end{align*}
by using the fact that not only $I_{s}^{(1)} - \hat{I}_{s}^{(n,1)},\dots,I_{s}^{(M_{n})} - \hat{I}_{s}^{(n,M_{n})}$ are identically distributed but also $\hat{I}_{s}^{(n,1)} - I_{t_{j,n}}^{(n,1)},\dots,\hat{I}_{s}^{(n,M_{n})} - I_{t_{j,n}}^{(n,M_{n})}$. Hence, from~\eqref{eq:intermediate error estimate} we obtain that
\begin{align*}
&e^{-\int_{t_{j,n}}^{t}\lambda_{p,j,n}(s) + (3p-2)\hat{\lambda}_{k+1}(s,N(s))\,\mathrm{d}s}\E\big[|\hat{I}_{t}^{(n,l)} - I_{t}^{(\ell)}|^{p}\big] \leq 
\E\big[|I_{t_{j,n}}^{(n,l)} - I_{t_{j,n}}^{(\ell)}|^{p}\big]\\
&  + \int_{t_{j,n}}^{t}
\delta_{p,j,n}(s) +\hat{\lambda}_{k+1}(s,N(s))\big(2c_{p,q}N(s)^{p}M_{n}^{-\frac{1}{2}} + m_{p,j,n}^{p}(s - t_{j,n})^{\frac{p}{2}}\big)\,\mathrm{d}s
\end{align*}
for all $t\in [t_{j,n},t_{j+1,n}]$ and $q > 2p$ by using~\eqref{eq:estimate for the empirical measure} and applying Lemma~\ref{le:EM-estimate}. Consequently, the claimed estimate follows by induction over $j\in\{0,\dots,k_{n} - 1\}$.
\end{proof}

\end{document}